\documentclass{amsart}

\usepackage[hmargin=3.5cm,top=3cm,bottom=3cm,includehead]{geometry}

\usepackage{babel}
\usepackage{pifont}

\usepackage{amsmath, amssymb, amsthm, mathrsfs, bbm}
\usepackage{graphicx,color,enumitem} 
\usepackage{multicol}
\usepackage{cancel}
\usepackage{abraces}
\usepackage{upgreek}
\usepackage{stmaryrd}

\usepackage[pdfborder={0 0 0}]{hyperref}

\usepackage[numbers,sort&compress]{natbib}

\usepackage{xparse}
\let\realItem\item 
\makeatletter
\NewDocumentCommand\myItem{ o }{%
   \IfNoValueTF{#1}%
      {\realItem}
      {\realItem[#1]\def\@currentlabel{#1}}
}
\makeatother

\usepackage{enumitem}    
\setlist[enumerate]{
    before=\let\item\myItem,       
    label=\textnormal{(\arabic*)}, 
    widest=(2')                    
}

\usepackage[dvipsnames]{xcolor}

\newcommand\N{{\mathbb N}}

\newcommand\R{{\mathbb R}}

\newcommand\Z{{\mathbb Z}}

\renewcommand{\d}{\mathrm{d}}
\newcommand{\dv}{\mathrm{d} v}
\newcommand{\dt}{\mathrm{d} t}
\newcommand{\ds}{\mathrm{d} s}

\newcommand{\dx}{\mathrm{d} x}

\def\BB{{\mathcal B}}
\def\CC{{\mathcal C}}
\def\DD{{\mathcal D}}
\def\EE{{\mathcal E}}
\def\FF{{\mathcal F}}

\def\HH{{\mathcal H}}
\def\II{{\mathcal I}}

\def\LL{{\mathcal L}}
\def\MM{{\mathcal M}}

\def\OO{{\mathcal O}}

\def\QQ{{\mathcal Q}}

\def\TT{{\mathcal T}}
\def\UU{{\mathcal U}}
\def\VV{{\mathcal V}}

\def\VV{{\mathcal V}}

\def\ZZ{{\mathcal Z}}

\def\CCC{{\mathscr C}}
\def\DDD{{\mathscr D}}

\def\FFF{{\mathscr F}}

\def\LLL{{\mathscr L}}
\def\MMM{{\mathscr M}}

\def\RRR{{\mathscr R}}
\def\SSS{{\mathscr S}}
\def\TTT{{\mathscr T}}
\def\UUU{{\mathscr U}}

\def\ZZZ{{\mathscr Z}}

\def\FFFF{{\mathfrak F}}

\def\RRRR{{\mathfrak R}}
\def\SSSS{{\mathfrak S}}

\def\ffff{{\mathfrak f}}
\def\gggg{{\mathfrak g}}

\def\qqqq{{\mathfrak q}}

\def\restrict#1{\raise-.5ex\hbox{\ensuremath|}_{#1}}

\newcommand{\wto}{\rightharpoonup}

\def\eps{{\varepsilon}}

\newcommand{\la}{\langle}
\newcommand{\ra}{\rangle}
\newcommand{\rra}{\rangle\!\rangle}
\newcommand{\lla}{\langle\!\langle}
\newcommand{\lv}{\lvert}
\newcommand{\rv}{\lvert}
\newcommand{\lvv}{\lVert}
\newcommand{\rvv}{\rVert }
\newcommand{\lvvv}{\lvert\hspace{-0.04cm} \lvert \hspace{-0.04cm} \lvert}
\newcommand{\rvvv}{\rvert\hspace{-0.04cm} \rvert\hspace{-0.04cm} \rvert}

\newcommand{\grad}{\nabla}

\newcommand{\Nt}{|\hskip-0.04cm|\hskip-0.04cm|}

\DeclareMathOperator{\Div}{div}

\newcommand\Ind{{\mathbf 1}}

\newtheorem{theo}{Theorem}[section]
\newtheorem{prop}[theo]{Proposition}
\newtheorem{lem}[theo]{Lemma}

\newtheorem*{thm*}{Theorem}
\theoremstyle{remark}
\newtheorem{rem}[theo]{Remark}

\newtheorem*{ex*}{Example}
\theoremstyle{definition}

\numberwithin{equation}{section}
\newcommand{\be}{\begin{equation}}
\newcommand{\ee}{\end{equation}}
\newcommand{\ba}{\begin{aligned}}
\newcommand{\ea}{\end{aligned}}
\newcommand{\beqn}{\begin{equation*}}
\newcommand{\eeqn}{\end{equation*}}
\newcommand{\bear}{\begin{eqnarray}}
\newcommand{\eear}{\end{eqnarray}}
\newcommand{\bean}{\begin{eqnarray*}}
\newcommand{\eean}{\end{eqnarray*}}
\newcommand{\beal}{\begin{aligned}}
\newcommand{\eeal}{\end{aligned}}

\title[The Boltzmann equation with non-isothermal Maxwell boundary conditions.]{The Boltzmann equation with non-isothermal Maxwell boundary conditions.}
\author[R. Medina]{R. MEDINA}

\address[R.~Medina]{Centre de Recherche en Math\'ematiques de
  la D\'ecision (CEREMADE, CNRS UMR 7534),
  Universit\'es PSL \& Paris-Dauphine, Place de Lattre de
  Tassigny, 75775 Paris 16, France}
\email{richard.medina-rodriguez@dauphine.psl.eu}
\date{\today}



\begin{document}

\begin{abstract}
In this paper, we study the Boltzmann equation in a close to the hydrodynamic limit regime, set in bounded spatial domains with non-isothermal Maxwell boundary conditions.
We establish the existence, uniqueness, and asymptotic stability of a non-equilibrium steady state under suitable smallness assumption on the temperature fluctuations at the boundary.
\end{abstract}

\maketitle

\tableofcontents

\section{Introduction}\label{sec:Intro_noniso}

In this paper, we study a particle system governed by the Boltzmann equation in a regime close to the \emph{hydrodynamic limit}. We assume that the system is confined to a bounded spatial domain and subject to \emph{non-isothermal} Maxwell boundary conditions; that is, thermostats are imposed at the boundary and maintained at prescribed (possibly space--dependent) temperatures throughout the evolution.

We consider two geometric settings: smooth $C^2$ domains with a space--dependent accommodation coefficient taking values in $[\iota_0, 1]$, for some fixed $\iota_0\in (0,1]$; and a right circular cylinder with diffusive boundary conditions on the bases and specular reflection on the lateral surface, which includes the case of different constant temperatures at the bases.

Because of boundary interactions with thermostats, this problem lies in the framework of non-equilibrium statistical mechanics, see \cite{MR3223169}. As a result, the steady solution of the problem is called a \emph{non-equilibrium steady state} (NESS), which cannot be of Maxwellian (Gaussian) form.

Our main goal is to prove the existence and uniqueness of the NESS, and to establish its stability in a perturbative regime by showing convergence in time of solutions to the steady state. To our knowledge, this is the first result of this type for the Boltzmann equation in bounded domains with non-isothermal Maxwell boundary conditions.

\subsection{Framework}\label{sec:Framework}
We consider a small $\eps\in (0,1)$ and we study the following Boltzmann equation
\be
\varepsilon\partial_t F = -v\cdot\grad_x F + \varepsilon^{-1} \QQ(F,F) \quad\text{ in }\UU:= (0,+\infty)\times\Omega\times\mathbb{R}^3,\label{eq:BE}
\ee
where $F = F(t, x, v)$ is a density function representing particles which at time $t\in (0,\infty)$, are located at position $x\in \Omega\subset \R^3$ and move with velocity $v\in \R^3$. 

The presence of the small parameter $\eps > 0$ in the equation reflects the fact that the system is close to the \emph{hydrodynamic limit}. For a detailed discussion of the physical interpretation and the main mathematical results concerning this type of limit, we refer the reader to \cite{MR2683475} and the references therein.

The \emph{Boltzmann collision operator} $\QQ$ represents the collisions between particles inside $\Omega$, and---for any two sufficiently regular functions $G, H:\R^3\to \R$---is given by the bilinear form
\beqn
\QQ(G,H):={1\over 2}\int_{\R^3} \int_{\mathbb{S}^{2}} \BB\, \left[ G(v_*')H(v') +H(v_*')G(v') -G(v_*)H(v) - G(v) H(v_*) \right] \d\sigma \d v_* ,
\eeqn
where we have defined the post-collisional velocities
\begin{equation*}
    v':= v -((v-v_*)\cdot \sigma)\sigma \quad \text{ and } \quad  v_*':= v_* +((v-v_*)\cdot \sigma)\sigma,
\end{equation*}
with $\sigma \in \mathbb{S}^2$. The \emph{collision kernel} $\BB=\BB(\lvert v-v_*\rvert,\sigma)$ describes the \emph{type of interaction} particles exhibit and, throughout this paper, we choose the so-called \emph{hard spheres} model by taking
\begin{equation*}
\BB(\lvert v-v_*\rvert,\sigma):= \lvert (v-v_*)\cdot\sigma\rvert.
\end{equation*}

\medskip
We assume $\Omega$ to be at least a bounded Lipschitz domain satisfying 
$$
\lv \Omega \rv = \int_\Omega \dx = 1 \quad \text{ and } \quad \int_\Omega x\, \dx =0,
$$
without loss of generality. Moreover, we assume that there exists $\delta \in W^{1,\infty}(\R^3, \R)$ such that $\Omega = \{x\in \R^3,\, \delta (x)>0 \}$, and $\lvert \delta (x) \rvert = \mathrm{dist}(x,\partial \Omega)$ on a neighborhood of the boundary. We can then define the outward normal vector 
$$
n_x = n(x):=-{\grad \delta (x)\over \lv \grad \delta (x)\rv } \quad \text{for almost every } x\in \bar \Omega.
$$
We now introduce the boundary set $\Sigma = \partial \Omega\times \R^3$ and we distinguish between the sets of \emph{outgoing} ($\Sigma_+$), \emph{incoming} ($\Sigma_-$), and \emph{grazing} ($\Sigma_0$) velocities at the boundary defined as
\beqn
\Sigma_\pm:= \{(x,v)\in \Sigma,\, \pm n_x\cdot v>0 \}, \quad \text{ and } \quad  \Sigma_0:= \{ (x,v)\in \Sigma, \, n_x\cdot v =0\}.
\eeqn
Moreover, we denote $\Gamma:= (0,\infty)\times\Sigma$ and, accordingly, $\Gamma_{\pm}:= (0,\infty)\times \Sigma_\pm$. We define $\gamma F$ as the trace function associated with $F$ over $\Gamma$ and $\gamma_\pm F:= \Ind_{\Gamma_{\pm}} \gamma F$. 
 
We then complement the Boltzmann equation~\eqref{eq:BE} with the \emph{non-isothermal Maxwell boundary condition}
\be\label{eq:BEBC}
\gamma_-F(t, x, v)=\RRR_\Theta \gamma_+F (t, x, v):= (1-\iota(x)) \SSS\gamma_+F(t, x, v) + \iota(x) \DDD_\Theta \gamma_+F(t, x, v) \quad  \text{ on }\Gamma_-,
\ee
for an \emph{accommodation coefficient} $\iota:\partial \Omega\to [0,1]$, where we have defined the \emph{specular reflection} operator
\begin{equation*}
\SSS \gamma_+F( t, x,v):= \gamma_+F(t , x,\VV_x v) \quad \text{ with } \quad \VV_x v = v - 2(n_x\cdot v)n_x,
\end{equation*}
and the \emph{diffusive reflection} operator
\begin{equation*}
\DDD_\Theta \gamma_+F(t , x,  v):=  \MMM_\Theta(v) \widetilde{\gamma_+F}  \quad \text{where}   \quad \widetilde{\gamma_+F}=\int_{\R^3}\gamma_+F(t, x,u) (n(x)\cdot u)_+du .
\end{equation*}
We have also defined the \emph{Maxwellian} distributions
\be\label{def:MaxwellianDist}
\MMM_\Theta:= \sqrt{2\pi \over \Theta} \MM_\Theta  \quad \text{ with } \quad \MM_\Theta = \MM_\Theta(x, v):= (2\pi \Theta)^{-3/2} \exp\left\{-{\lv v\rv^2 \over 2\Theta}\right\},
\ee
associated with a prescribed wall temperature $\Theta = \Theta(x) := 1 +  \eps^{\qqqq} \, \vartheta(x)$, with $\qqqq=12$ motivated by Remark~\ref{rem:Def_qqqq}, and a function $\vartheta \in L^{\infty} (\partial \Omega)$ describing the \emph{temperature fluctuation}, satisfying 
$$
\lvv \vartheta \rvv_{L^\infty(\partial \Omega)} \leq  \vartheta_0 \leq 1/8   .
$$
In particular, this implies that for every $x\in \partial \Omega$ there holds
\be\label{eq:ConditionsTheta}
\Theta_* \leq \Theta(x) \leq \Theta^*, \quad \text{ for some } \quad \frac 78 < 1-\vartheta_0  \leq \Theta_*\leq \Theta^*\leq 1+\vartheta_0< \frac 98 .
\ee
Furthermore, it is worth noting that the constant $\sqrt{2\pi / \Theta}$ in the definition of $\MMM_\Theta$ ensures the normalization condition $\widetilde {\MMM_\Theta}=1$, for every $x\in \partial \Omega$. \\

We now present the two types of \emph{geometric assumptions} for our domain $\Omega$, and the respective choice for the accommodation coefficient in each case. 
\begin{enumerate}[leftmargin=*]
\item[(H1)]\label{item:H1} Assume $\Omega\subset\R^3$ is an open $C^2$ domain, and $\delta\in C^{2}(\R^3, \R)\cap W^{3,\infty} (\R^3, \R)$. Moreover, take $\iota\in C(\partial\Omega)$ and assume that there is $\iota_0\in (0,1]$ such that for every $x\in \partial \Omega$ there holds $\iota(x) \in [\iota_0, 1]$.
\item[(H2)]\label{item:H2} Assume $\Omega =  (-L, L) \times \Omega_0$, for some $L>0$ and where $\Omega_0\subset \R^2$ is the $2$-dimensional disk of radius $\RRRR>0$ centered at the origin. In this case we also define 
\begin{equation*}
    \Lambda_1:= \{-L\}\times\Omega_0,\quad
    \Lambda_2:= \{L\}\times\Omega_0, \quad
    \Lambda_3:= (-L,L)\times \partial \Omega_0 ,\label{eq:cylinderDefinition}
\end{equation*}
and $\Lambda:= \Lambda_1\cup \Lambda_2\cup \Lambda_3$. Furthermore, we impose  \emph{mixed} boundary conditions by taking $\iota= \Ind_{\Lambda_1 \cup \Lambda_2}$, i.e. purely diffusive boundary condition on the bases of the cylinder ($\Lambda_1\cup\Lambda_2$), and specularity on the lateral surface ($\Lambda_3$).
\end{enumerate}

Finally, we complement Equation \eqref{eq:BE}-\eqref{eq:BEBC} with the initial condition
\be\label{eq:BEIC}
F (t = 0, \cdot ) = F_0 \quad \text{ in } \OO:= \Omega \times \R^3,
\ee
for some function $F_0$ satisfying $\lla F_0\rra_\OO := \displaystyle \int_{\OO} F_0   \, \dv\dx =1$.


\subsection{Main results and discussion}\label{sec:MainResult} 

In order to express our main results we need to introduce the set of the so-called \emph{admissible weight functions} $\omega:\R^3 \to [1,\infty)$ defined by 
$$
\omega (v)=e^{\zeta\lvert v\rvert^2}, \quad \text{ with }  \zeta \in \left( \frac 1{4(1-\vartheta_0 )}, \frac 1{2(1+\vartheta_0)} \right).
$$
In particular, we note that this choice of weight functions is such that 
$$
\zeta-\frac 1{2\Theta} <0 \quad \text{ and } \quad -2\zeta  + \frac 1{2\Theta} <0,
$$
thus there holds respectively
$$
\omega\MMM_\Theta \in L^1(\OO)\cap L^\infty(\OO) \quad \text{ and } \quad \omega^{-2} \MMM_\Theta^{-1}\in L^1(\OO)\cap L^\infty(\OO).
$$

Moreover, for a given measure space $(Z,\ZZZ,\mu)$, a weight function $\rho: Z \to (0,\infty)$, and an exponent $p \in [1,\infty]$, we define the weighted Lebesgue spaces $L^p_\rho(Z)$ 
associated to the norm 
\be\label{def:Lebesgue_weighted_spaces}
\| g \|_{L^p_\rho(Z)} = \| \rho g \|_{L^p(Z)}. 
\ee

\medskip
Finally, we define $\MM := \MM_1$, where we recall that $\MM_1$ is defined in \eqref{def:MaxwellianDist}. In this framework, we have the following results for the Boltzmann equation with non-isothermal Maxwell boundary conditions.

\begin{theo}\label{theo:NESS}
Assume that either Assumption \ref{item:H1} or Assumption \ref{item:H2} holds, and let $\omega$ be an admissible weight function.

There is a constant $\eps_\star > 0$ such that for every $\eps \in (0, \eps_\star)$, there is $\vartheta_\star = \vartheta_\star(\eps)>0$, satisfying $\vartheta_\star (\eps)\to 0$ as $\eps\to 0$, such that for every $\vartheta_0 \in (0, \vartheta_\star)$, there exists a unique normalized non-equilibrium steady state \hbox{$\FFFF \in L^\infty_\omega(\UU)$}, to the steady problem associated with the Boltzmann equation~\eqref{eq:BE}--\eqref{eq:BEBC}. 
Furthermore, there holds
\be\label{eq:BE_NESS_eps}
\lvv \FFFF - \MM \rvv_{L^\infty_\omega(\OO)} \leq  \lambda(\eps),
\ee
for some $\lambda(\eps)>0$, such that $\lambda(\eps)\to 0$ as $\eps\to 0$.
\end{theo}

\begin{rem}
The steady state in Theorem~\ref{theo:NESS} is given by $\FFFF=\MM+\FFF$, where $\FFF$ is the unique stationary solution with zero mass of the perturbed nonlinear equation arising from this decomposition, and constructed in Theorem~\ref{theo:NESS_L}.
\end{rem}

\medskip
Furthermore, we have the following stability result.

\begin{theo}\label{theo:Main}
Under the same conditions as in Theorem~\ref{theo:NESS}, for every $\eps\in (0,\eps_\star)$ and every $\vartheta_0 \in (0,\vartheta_{\star})$, there is $\eta(\eps)>0$, satisfying $\eta(\eps)\to 0$ as $\eps\to 0$, such that for every $F_0\in L^\infty_{\omega}(\OO)$ satisfying
\beqn
\lvv F_0 - \FFFF \rvv_{L^\infty_\omega(\OO)} \leq  (\eta(\eps))^2,
\eeqn
there exists $F\in  L^\infty_\omega(\UU)$ unique global weak solution to the Boltzmann equation~\eqref{eq:BE}--\eqref{eq:BEBC}--\eqref{eq:BEIC}.
Furthermore, there is a constructive constant $\theta>0$ such that 
\beqn
\lvv F_t  -  \FFFF\rvv_{L^\infty_\omega(\OO)}  \lesssim  \eta(\eps) \, e^{-\theta t}\label{eq:BEdecayFinal} \qquad  \forall t\geq 0.
\eeqn
\end{theo}

\begin{rem}
The solution of the Boltzmann equation given by Theorem~\ref{theo:Main} is constructed in Theorem~\ref{theo:Main_h} after the analysis from Remark~\ref{rem:Eq_h_0}.
\end{rem}

\subsection{Strategy for the proof of the main results}\label{ssec:Strategy}
In this subsection we briefly explain the main ideas leading to the proof of our main theorems. For clarity, we separate each of the fundamental steps leading to our results in different sub-subsections.

\subsubsection{Transformation of the problem}
We write $F=\MM + f$; then $f$ solves the following evolution equation
\begin{equation}
	\left\{\begin{array}{rlll}
		 \partial_t f &=& \displaystyle -  \eps^{-1} \,   v\cdot \grad_x  f + \eps^{-2} \, \CCC f + \eps^{-2} \,  \QQ(f,f)   & \text{ in } \UU \\
		\gamma_-f&=&\RRR_\Theta \gamma_+f  + \iota \psi&\text{ on }\Gamma_-\\
		f_{t=0}&=&f_0 &\text{ in }\OO,
	\end{array}\right.\label{eq:PBE}
\end{equation}
where we have defined the \emph{linearized Boltzmann collision operator} $\CCC f:=\QQ(\MM,f)+\QQ(f,\MM)$, and the inflow-type function
$$
\psi = \psi(x,v) := (\Theta)^{-1/2} \MM_{\Theta}  -\MM.
$$

\subsubsection{Hypocoercivity}\label{sssec:Hypo_decay}
Let us now explain the key estimate that allows us to prove our results. To do this, we introduce the auxiliary evolution equation 
\be
	\left\{\begin{array}{rlll}
		\partial_{t} g &=& -v\cdot \grad_x g + \CCC g &\text{ in }\UU\\
		\gamma_- g&=&\RRR_\Theta \gamma_+g  &\text{ on }\Gamma_{-}\\
		 g_{t=0}&=& g _0  &\text{ in }\OO,
	\end{array}\right.\label{eq:Hypo_decay_intro}
\ee
where, for the sake of simplicity, we have disregarded the rescaling parameter $\eps$, the inflow-type term $\psi$, and the non-linear term $\QQ$ from Equation~\eqref{eq:PBE}.

We consider the weighted Hilbert space $\HH :=  L^2_{\MM^{-1/2}}(\OO)$, as defined in \eqref{def:Lebesgue_weighted_spaces}, equipped with the scalar product
$$
\la g, h\ra_\HH := \int_\OO g(x,v)\, h (x,v) \MM^{-1} (v) \, \dv\dx,
$$
and its associated norm $\lvv\cdot \rvv_\HH$. 

The main novelty of this paper, established during Section~\ref{sec:Hypo_decay} for a general equation as the one above, consists in proving that, given that the temperature fluctuation ($\vartheta_0>0$) is small enough then: given $g_0\in \HH$, satisfying $\lla g_0\rra_\OO =0$, there is a unique global weak solution $g\in C(\R_+, \HH)$ to Equation~\eqref{eq:Hypo_decay_intro}, and there holds
\be\label{eq:Hypo_decay_HH_intro}
\| g_t \|_{\HH} \lesssim e^{-\kappa t} \| g_{0} \|_{\HH}  \qquad \forall \, t \ge 0.
\ee


Let us briefly explain the ideas and main challenges in the obtention of this estimate:
Exploiting the presence of the hypocoercive operator $\CCC$, we employ the techniques developed in \cite{MR4581432} and we obtain that there is $\kappa >0$ such that 
\be
    \lVert g_t \rVert_{ \HH}^2  \lesssim e^{-\kappa t}\lVert g_0\rVert_{ \HH} ^2
    +    \vartheta_0   \int_0^t e^{- \kappa (t-s)} \left\lVert  (\iota)^{1/2} (\widetilde{\gamma_+ g_s} ) \right\rVert^2_{L^2(\partial \Omega)} \ds  \qquad \forall t\geq 0. \label{eq:L2Decay_pert_intro_Hypo_decay}
\ee
Physically, this estimate is in accordance with the presence of non-equilibrium effects coming from the boundary. Mathematically, our computations align with the physical intuition and show that this is the reason why a direct application of hypocoercivity methods does not allow us to immediately recover a decay estimate as in \cite{MR4581432}.

Nevertheless, a key observation is that the additional boundary term in \eqref{eq:L2Decay_pert_intro_Hypo_decay} is controlled by the temperature fluctuation $\vartheta_0$. The idea is then to prove that there is $\kappa_1>0$ such that the following trace bound holds
\be\label{eq:A_priori_bound_Bdy_Hypo_decay_Intro}
\int_0^t \int_{\{ x\in \partial \Omega; \, \iota(x) \neq 0\} }\int_{\R^3}  (\gamma g_s)^2 \, (n_x\cdot v)^2 \, \la v\ra^{-2} \, \MM^{-1} \,  \dv\d\sigma_x\ds  \lesssim e^{\kappa_1  t} \lvv g_0\rvv_{\HH}^2 \qquad \forall t\geq0,
\ee
where $\d\sigma_x$ denotes the boundary measure. 
Using the Cauchy-Schwarz inequality we can then combine \eqref{eq:L2Decay_pert_intro_Hypo_decay} and \eqref{eq:A_priori_bound_Bdy_Hypo_decay_Intro} and we obtain that 
$$
    \lVert g_t \rVert_{ \HH}^2  \lesssim \left( e^{-\kappa t} +  \vartheta_0 e^{\kappa_1 t} \right) \lVert g_0\rVert_{ \HH} ^2.
$$
Choosing $T>0$ large enough and $\vartheta_0$ small enough we further deduce that 
$$
    \lVert g_T \rVert_{ \HH}^2  \leq \frac 12 \lVert g_0\rVert_{ \HH} ^2,
$$
which in turn yields the exponential decay for the solutions of Equation~\eqref{eq:Hypo_decay_intro}.

\subsubsection{Weighted $L^2$ estimates}
We now want to apply the ideas developed above to study Equation~\eqref{eq:PBE}.  To do this we introduce the \emph{linearized Boltzmann equation} 
\begin{equation}
	\left\{\begin{array}{rlll}
		 \partial_t f &=&  - \eps^{-1} v\cdot \grad_x f +  \eps^{-2} \CCC f  =: \LLL^\eps f &\text{ in }\UU\\
		\gamma_-f&=&\RRR_\Theta\gamma_+f  + \iota \,  \psi&\text{ on }\Gamma_-\\
		f_{t=0}&=&f_0 &\text{ in }\OO.
	\end{array}\right.\label{eq:LPBE}
\end{equation}
We also note that, 
employing the fundamental theorem of calculus, there holds
\bean
\psi &=&
 \int_1^{1+ \eps^{\qqqq}\vartheta} \left(\frac \d{\d z} {\MM_z  (v) \over \sqrt z} \right) \d z
 = {1 \over (2\pi)^{d/2}} \int_1^{1+\eps^{\qqqq}\vartheta}   \left[ -\frac 12 z^{-1}  -\frac 32 z^{-1-3/2} + \frac {\lv v\rv^2}{2z^2} \right]  z^{-1/2}  \exp\left\{ -{\lv v\rv^2 \over 2z} \right\} \,  \d z\\
&\leq & \vartheta_0 \, \eps^{\qqqq} \,  {1 \over (2\pi)^{3/2}}  \left[ \frac 12 \Theta_*^{-1} + \frac 32 \Theta_*^{-1-3/2} + \frac 1 {2\Theta_*^2} \right] \Theta_*^{-1/2}  \la v\ra^2 \exp\left\{ -{\lv v\rv^2 \over 2\Theta^*} \right\} ,
\eean
hence, taking the $L^\infty_x(\partial\Omega)$ norm in the above inequality, we obtain
\be\label{eq:Controlpsi}
\lvv \psi \rvv_{L^\infty_x(\partial \Omega)}  \lesssim \vartheta_0 \, \eps^{\qqqq}\,  \la v \ra^2 \exp\left\{ -{\lv v\rv^2 \over 2\Theta^*} \right\}.
\ee
Hence, due to the very definition of $\psi$, there also holds
\be\label{eq:PertTempBdy}
\lvv \MMM_\Theta - \MMM_1  \rvv_{L^\infty_x(\partial \Omega)}  \lesssim \vartheta_0 \, \eps^{\qqqq}\,  \la v \ra^2 \exp\left\{ -{\lv v\rv^2 \over 2\Theta^*} \right\}.
\ee

Using this and applying the ideas developed before, we prove during Section~\ref{sec:Hypo_Pert} that any solution of Equation~\eqref{eq:LPBE} satisfies the bound
\begin{multline}
    \lVert f_t \rVert_{ \HH}  \lesssim e^{-\kappa t}\lVert f_0\rVert_{ \HH} 
    +    \vartheta_0^{1/2} \eps^{(\qqqq-1)/2} \left( \int_0^t e^{- 2\kappa (t-s)} \left\lVert  (\iota)^{1/2} (\widetilde{\gamma_+ f_s} ) \right\rVert^2_{L^2(\partial \Omega)} \ds\right)^{1/2} \\
    +  \vartheta_0^{1/2} \eps^{(\qqqq-1)/2}  , \label{eq:L2Decay_pert_intro}
\end{multline}
for every $t\geq0$.

In particular, we note the presence of the last term in the right hand side of \eqref{eq:L2Decay_pert_intro}, which encodes the presence of the non-homogeneous term $\psi$ at the boundary. 

Let us also note that the above result is uniform in $\eps$ for every $\qqqq\geq 1$, and it is in accordance with previous results from \cite{MR4581432, GuoZhou2024}. 

We also note that, though possible, we do not go further than this estimate in the framework of the hydrodynamic limit as explained in Sub-subsection~\ref{sssec:Hypo_decay}. This would cause higher (exponential) dependencies of $\vartheta_0$ on $\eps$, whereas our objective is to consider the lowest possible dependencies of the parameters on $\eps$ such that we observe a long-time behavior of the solutions. 

We thus argue in the spirit of \cite{MR3740632, MR4265982}, and we use \eqref{eq:L2Decay_pert_intro} together with the $L^2-L^\infty$ method from \cite{BE_iso} to close an a priori bound for the solutions of Equation~\eqref{eq:LPBE}.

\subsubsection{Weighted $L^\infty$ estimates}
We now consider a function $G: \UU^\eps \to \R$ satisfying $\lla G_t\rra_\OO=0$ for every $t\geq0$, and we study the following perturbed evolution equation
\be
	\left\{\begin{array}{rlll}
		\partial_{t} f &=& \LLL^\eps f + \eps^{-2} \, G&\text{ in }\UU^\eps\\
		\gamma_- f&=&\RRR_\Theta \gamma_+f  + \iota \psi&\text{ on }\Gamma_{-}^\eps\\
		 f_{t=0}&=& f _0  &\text{ in }\OO^\eps.
	\end{array}\right.\label{eq:PLPBE}
\ee
Using the above $L^2$ estimate, and adapting the $L^2-L^\infty$ method from \cite{BE_iso} to this framework, we obtain in Section~\ref{sec:AprioriLinfty} that for $\eps, \vartheta_0 >0$ small enough there holds 
\be
\lVert  f_t \rVert_{L^{\infty}_{\omega}(\bar \OO)} \leq  \varpi_\eps e^{-\theta t} \left(   \lVert f_0\rVert_{L^{\infty}_{\omega}(\OO)} +  \underset{s\in[0,t]}{\sup}\left[ e^{\theta s} \lVert G_s \rVert_{L^{\infty}_{\omega\nu^{-1}}(\OO)}\right]\right)  + C_0 \eps^{(\qqqq-12)/2}\vartheta_0 \quad \forall t\geq 0, \label{eq:LinftyEstPert_Intro}
\ee
for some universal constants $C_0, \theta>0$, and for some $\varpi_\eps>0$ satisfying that $\varpi_\eps  \to \infty$ as $\eps \to 0$. 
In particular, we note that the choice of $\qqqq=12$ was taken to make the term related to $\vartheta_0$ in \eqref{eq:LinftyEstPert_Intro} uniform in $\eps$. 
In this way, we can identify the minimal requirements imposed by the linear analysis to derive uniform estimates on $\eps$ regarding the temperature fluctuation term $\vartheta_0$. This will be explored and exploited in future works studying the hydrodynamic limit of non-isothermal problems, which remains an open question. 

\begin{rem}\label{rem:Def_qqqq}
The role of the exponent $\qqqq$ in the definition of $\Theta$ is to understand and make explicit the minimal power of $\eps$ needed so that all terms involving $\vartheta_0$ remain uniform in $\eps$ in the linear estimates, both in $L^2$ and $L^\infty$, namely \eqref{eq:L2Decay_pert_intro} and \eqref{eq:LinftyEstPert_Intro} respectively.

By contrast, the additional $\eps$-dependence appearing in the bound for $\vartheta_\star$ in Theorem~\ref{theo:NESS} is not related to this linear mechanism. It instead arises from the nonlinear analysis, and it is reminiscent of the behavior of $\varpi_\eps$ in \eqref{eq:LinftyEstPert_Intro}.

It is also worth noting that, if one is not concerned with the long-time asymptotic behavior of $F$, a simpler choice $\qqqq = 1$ would suffice, as indicates \eqref{eq:L2Decay_pert_intro} and arguing as in \cite{GuoZhou2024}. In that case, one would moreover expect $\vartheta_\star$ to be independent of $\eps$. 
\end{rem}

\subsubsection{Existence and uniqueness of a NESS}
After establishing the above a priori estimates, we consider the steady problem 
 \be
	\left\{\begin{array}{rlll}
		- \LLL^\eps  \FF  &=& \eps^{-2} G  &\text{ in }\OO\\
		\gamma_-  \FF&=&\RRR_\Theta \gamma_+ \FF + \iota\psi &\text{ on }\Sigma_{-},
	\end{array}\right.\label{eq:PLBE_SS}
\ee
and arguing similarly to the evolution problem we obtain in Section~\ref{sec:AprioriLinfty_SS} that, if $\lla \FF\rra_\OO=0$, then for $\eps, \vartheta_0 >0$ small enough there holds
\be
\lvv \FF \rvv_{L^\infty_\omega(\bar \OO)}  \leq  \varpi_\eps  \lVert G \rVert_{L^{\infty}_{\omega\nu^{-1}}(\OO)}  + C_0 \vartheta_0 .
\label{eq:LinftyPerturbedFiniteTimeDecay_Summary_SS_Intro}
\ee
Using this estimate and the well-posedness result from Section~\ref{sec:WellPosedness}, we apply a fixed point theorem to establish the existence of a unique $\FFF$ solving the stationary equation
 \be
	\left\{\begin{array}{rlll}
		- \LLL^\eps  \FF  &=& \eps^{-2} \QQ(\FF,\FF)  &\text{ in }\OO\\
		\gamma_-  \FF&=&\RRR_\Theta \gamma_+ \FF + \iota\psi &\text{ on }\Sigma_{-},
	\end{array}\right.\label{eq:PLBE_SS_QQ}
\ee
proving an equivalent version of Theorem~\ref{theo:NESS}.

\subsubsection{Stability}
Finally, taking $f = h+ \FFF$ in Equation~\eqref{eq:PBE}, we study the resulting evolution equation for $h$ writing as
 \be
	\left\{\begin{array}{llll}
		\partial_{t} h &=& \LLL^\eps h  + \eps^{-2} \QQ(h, \FFF) + \eps^{-2} \QQ(\FFF, h) + \eps^{-2} \QQ(h,h)  &\text{ in }\UU\\
		\gamma_- h&=&\RRR_\Theta \gamma_+h  &\text{ on }\Gamma_{-}\\
		 h_{t=0}&=& h_0  &\text{ in }\OO,
	\end{array}\right.\label{eq:SLPBE_Intro}
\ee
Repeating the computations leading to \eqref{eq:LinftyEstPert_Intro}, we obtain that
\be
\lVert  h_t \rVert_{L^{\infty}_{\omega}(\bar \OO)} \leq  2\varpi_\eps \, e^{-\theta t} \left(   \lVert h_0\rVert_{L^{\infty}_{\omega}(\OO)} +  \underset{s\in[0,t]}{\sup}\left[ e^{\theta s} \lVert \QQ(h_s,h_s) \rVert_{L^{\infty}_{\omega\nu^{-1}}(\OO)}\right]\right)   ,
\label{eq:LinftyEstPert_h_Intro}
\ee
for every $t\geq0$, and $\varpi_\eps$ has been defined above. Using the estimates from Section~\ref{sec:AprioriLinfty} and the well-posedness results from Section~\ref{sec:WellPosedness}, we use again a fixed point argument to prove Theorem~\ref{theo:Main}.

\subsection{State of the art and discussion}
The Boltzmann equation aims to describe the evolution of interacting particles from a mesoscopic (or statistical) viewpoint. For a presentation of current trends of research and results regarding this problem we refer the reader to the surveys \cite{MR1942465, MR1359296}, as well as to the discussion of \cite{BE_iso}. We focus then our discussion on the novelty, significance, and limitations of our main results, as well as some possible future directions of research.

Regarding the problem of the Boltzmann equation in non-isothermal domains, previous results \cite{MR3085665, MR3740632, MR4265982} only address diffusive reflections at the boundary. This paper thus presents a novel result by considering more general boundary conditions.

\smallskip
Let us now discuss the limitations of our findings. 
To this end, we consider the first-order Hilbert expansion of $F$, given by
\be\label{eq:HilbertExpansion_Intro}
F = \MM + \eps \sqrt\MM f_1 .
\ee
It is expected (see \cite{MR2182829, MR4265982, MR2683475, GuoZhou2024}) that
$f_1$ should converge, as $\eps \to 0$, in a suitable way\footnote{More precisely, for well-prepared initial data, $f_1$ should converge towards $\displaystyle u_0 + v\cdot u_1 + u_2 {(\lv v\rv^2-3 )/ 2}$, where $u_0, u_1, u_2$ satisfy the INSF equations. We refer to \cite[Theorem~1.1]{GuoZhou18} for further details.}, to a solution of an incompressible Navier--Stokes--Fourier (INSF) system.

However, in the current literature, when the boundary is endowed with Maxwellian reflections (even in the isothermal case), existing methods fail to characterize the long-time behavior of $f_1$, beyond uniform-in-time bounds, see \cite{GuoZhou2024}. Our interest, however, lies in describing the asymptotic behavior of $F$ as time tends to infinity, which has been addressed in \cite{BE_iso} for the isothermal case. Motivated by this previous work, instead of using the above Hilbert expansion, we perturb the Boltzmann equation by writing
$$
F=\MM + \text{non-equilibrium correction } + f, 
$$
and, as a consequence, we observe that the smallness parameter ($\eta$ in Theorem~\ref{theo:Main}) depends on $\eps$. 
Making this dependence explicit, we obtain $\eta \approx \eps^{11/2}$, which means that as $\eps\to 0$, $f_1$ eventually fails to be close enough to the steady state, and thus the long-time behavior no longer holds. 

We do not claim that this result is sharp, and it naturally raises the question of whether the difficulty in simultaneously obtaining the INSF limit and long-time convergence for $f_1$ in regime of the hydrodynamic limit is due to limitations of current techniques, or whether there is a genuine mechanism preventing both from holding simultaneously.

\medskip
The novelty of our results lies in the extension of the hypocoercivity techniques from \cite{MR4581432}, originally developed for isothermal (equilibrium) problem, to the non-isothermal framework. 
It is worth noting that hypocoercivity-type results in non-isothermal domains are far from straightforward, and may even appear counterintuitive at first glance, since temperature variations destroy the usual coercive structure and introduce boundary contributions that must be handled carefully. In a broad sense, however, our techniques show that, if the temperature fluctuations at the boundary are small, then the regularizing effect of the linearized operator in the interior remains strong enough to generate a decay. 

It is also worth noting that, to our knowledge, Section~\ref{sec:Hypo_decay} presents the first hypocoercivity result for linear kinetic equations in domains equipped with non-isothermal Maxwell boundary conditions. Previous results \cite{MR3085665, MR3740632, MR4265982} do not fully derive such decay results at the linear level, and instead use complementary ultracontractivity estimates to close the a priori bounds. Furthermore, we note that our techniques are flexible enough to be adapted to problems involving, for instance, weakly coercive collision operators, as is the case for the linearized Landau equation.


\subsection{Organization of the paper} This paper is structured as follows:

In Section~\ref{sec:Hypo_decay} we use, adapt, and extend the methods from \cite{MR4581432} to treat non-isothermal Maxwell boundary conditions. We study a problem in the spirit of Equation~\eqref{eq:Hypo_decay_intro} and we establish the well-posedness of this equation in $\HH$ and the decay estimate \eqref{eq:Hypo_decay_HH_intro}

We devote Section~\ref{sec:Hypo_Pert} to employ the techniques developed in Section~\ref{sec:Hypo_decay} to deal now with a non-homogeneous equation including, besides the Maxwell boundary conditions, the inflow-type term $\psi$. We establish the well-posedness of Equation~\eqref{eq:LPBE} and the $L^2$ estimate \eqref{eq:L2Decay_pert_intro}. 

In Section~\ref{sec:AprioriLinfty}, we use the $L^2-L^\infty$ results from \cite{BE_iso}, together with the $L^2$ estimate from Section~\ref{sec:Hypo_Pert}, to obtain \eqref{eq:LinftyEstPert_Intro} for solutions of Equation~\eqref{eq:PLPBE}. Later in Section~\ref{sec:AprioriLinfty_SS}, we argue in the spirit of Section~\ref{sec:AprioriLinfty} to further establish \eqref{eq:LinftyPerturbedFiniteTimeDecay_Summary_SS_Intro} for the solutions of the associated steady problem Equation~\eqref{eq:PLBE_SS}. 

In Section~\ref{sec:WellPosedness}, we use the estimates obtained in the previous sections to establish the well-posedness of linear kinetic equations in non-isothermal domains in a weighted $L^\infty$ framework. These results rely on the theory developed in \cite[Section~6]{BE_iso}, and extend them to problems with non-isothermal boundary conditions. 

We devote then Section~\ref{sec:NESS} to use the well-posedness results from Section~\ref{sec:WellPosedness} together with the a priori estimate \eqref{eq:LinftyPerturbedFiniteTimeDecay_Summary_SS_Intro}, to construct the solution of Equation~\eqref{eq:PLBE_SS_QQ}, proving Theorem~\ref{theo:NESS}. 

Finally, in Section~\ref{sec:Stability} we prove Theorem~\ref{theo:Main} using the well-posedness results from Section~\ref{sec:WellPosedness}, the decay estimate \eqref{eq:LinftyEstPert_h_Intro} for the solutions of Equation~\eqref{eq:SLPBE_Intro}, and a fixed point argument in the spirit of \cite[Theorem 6.1]{BE_iso}.

\section{Hypocoercivity under small fluctuations of the boundary temperature}\label{sec:Hypo_decay}

In this section we present a general framework to establish hypocoercivity estimates for linearized kinetic equations in non-isothermal bounded domains, under smallness assumptions on the temperature fluctuation. 


To be more precise, we consider a prescribed boundary temperature $\TTT:\partial \Omega \to \R_+^*$, and we work on the (simplified) linear kinetic equation
\be
	\left\{\begin{array}{rlll}
		\partial_{t} g &=& -v\cdot \grad_x g + \CCC g =: \LL g &\text{ in }\UU\\
		\gamma_- g&=&\RRR_\TTT \gamma_+g  &\text{ on }\Gamma_{-}\\
		 g_{t=0}&=& g _0  &\text{ in }\OO.
	\end{array}\right.\label{eq:Hypo_decay}
\ee
We suppose that $\TTT(x) = 1+ \theta(x)$ for every\footnote{In general, we only demand this to hold for any $x\in \partial \Omega$ such that the diffusive boundary condition holds, i.e. $\iota(x)>0$.} $x\in \partial \Omega$, and we assume that $\theta:\partial \Omega \to \R$ satisfies
\be\label{eq:Condition_T_theta}
\lvv \theta\rvv_{L^\infty(\partial \Omega)} \leq \theta_0,
\ee
for some $\theta_0 \in (0, 1/8)$, so that, in particular, $\TTT\in (7/8,9/8)$. 

We then have the following hypocoercivity result,  and we devote the rest of this section to its proof.

\begin{theo}\label{theo:Hypo_decay}
Assume that either Assumption \ref{item:H1} or Assumption \ref{item:H2} holds, and assume that $g_0\in \HH$ satisfying $\lla g_0\rra_\OO =0$. 
There exists a unique global weak solution $g\in C(\R_+, \HH)$ to Equation~\eqref{eq:Hypo_decay}. 
Furthermore, there exist positive constants $\theta_\star, \kappa , C >0$ such that, for any $\theta_0 \in (0,\theta_\star)$, there holds 
\be\label{eq:HypoEquiv_Hypo_decay}
\| g_t \|_{\HH} \le C e^{-\kappa t} \| g_{0} \|_{\HH}  \qquad \forall \, t \ge 0.
\ee
\end{theo}

\begin{rem}
We note that, although in smooth domains we assumed $\iota \geq \iota_0>0$, for some $\iota_0>0$ (see Assumption~\ref{item:H1}), Theorem~\ref{theo:Hypo_decay} holds in general for $\iota :\partial \Omega \to [0,1]$, by following the same ideas developed here, complemented with the corresponding computations from \cite{MR4581432}.
\end{rem}

\begin{rem}
This result generalizes \cite[Theorem 1.2]{MR4581432} and \cite[Theorem 2.1]{BE_iso} by extending the validity of their hypocoercivity theory to non-isothermal domains. 
\end{rem}

\begin{rem}
When working in cylindrical domains (i.e. under Assumption~\ref{item:H2}), it is worth noting that a further (more delicate) control on the trace has to be given in order to conclude Theorem~\ref{theo:Hypo_decay} from Theorem~\ref{theo:Pert_hypo_Hypo_decay}. 

Indeed, it is known (see \cite[Remark 6.6]{BE_iso} as well as Proposition~\ref{prop:L2_AprioriBdy_Hypo_decay} and Remark~\ref{rem:L2_AprioriBdy_Hypo_decay}) that in the case of cylindrical domains, trace estimates are more degenerate due to the presence of an extra term on the boundary measure (defined in \eqref{def:zetaSSSS_Hypo_decay}) which vanishes at the irregular points of the domain (i.e. the set $\SSSS$ defined in \eqref{def:SSSS_Hypo_decay}). 

This makes standard boundary estimates more degenerate than usual in the space variable, resulting in the fact that we cannot control the boundary term coming from Theorem~\ref{theo:Pert_hypo_Hypo_decay} the same way as we would do in the case of smooth domains.

Nevertheless, exploiting the geometry of the domain, and the boundary conditions of our problem, we are able to obtain stronger, sharper estimates of the trace at the bases of the cylinder (see Proposition~\ref{prop:L2_AprioriBdyLambda12_Hypo_decay}), which is where the diffusion operator acts ($\iota>0$). Thus establishing boundary estimates strong enough to close a decay result in $L^2$. 
\end{rem}

The rest of this section is structured as follows: 
We devote Subsection~\ref{ssec:Hypo_Preliminaries} to present the properties of the linearized Boltzmann collision operator $\CCC$ and the conservation laws satisfied by the solutions of Equation~\eqref{eq:Hypo_decay}. 
In Subsection~\ref{ssec:Hypo_Elliptic_Hypo_decay} we recall some known results on the well-posedness and gain of regularity of solutions to some classes of elliptic equations.  
In Subsection \ref{ssec:Hypo_Perturbed_Hypo_decay}, we use and extend the techniques from \cite{MR4581432, BE_iso} to problems in non-isothermal domains. 
We devote then Subsections \ref{ssec:Hypo_Apriori_Hypo_decay} and \ref{ssec:Hypo_AprioriBdy_Hypo_decay} to establish some a priori estimates for the solutions of Equation~\eqref{eq:Hypo_decay}. Later, in Subsection~\ref{ssec:Hypo_L2WellPosedness_Hypo_decay} we use the a priori estimates obtained during the previous subsections to establish the well-posedness of Equation~\eqref{eq:LPBE} in $\HH$. 
Finally, in Subsection~\ref{ssec:Hypo_ProofMain_Hypo_decay} we prove Theorem~\ref{theo:Hypo_decay} by using the above results.

\subsection{Preliminaries}\label{ssec:Hypo_Preliminaries}
We recall that we have defined $\CCC g := \QQ (\MM, g) + \QQ(g,\MM)$, and we decompose it as $\CCC g = Kg + \nu g$, where, on the one hand, we define the non-local operator
\beqn
Kf = Kf(\cdot , v) := \int_{\R^3}\int_{\mathbb{S}^{2}} \BB \left[\MM(v'_*)f(v') + \MM(v') f(v'_*) -\MM(v) f(v_*)\right]d\sigma dv_* 
= \int_{\R^3} k(v,v_*) f (v_*)dv_*,
\eeqn
with 
\beqn
k = k(v,v_*):=  \sqrt{2\over \pi} \lvert v-v_* \rvert^{-1} e^{ -{1\over 8} {(\lvert v_*\rvert^2 - \lvert v\rvert^2)^2\over \lvert v-v_*\rvert^2} - {1\over 8} \lvert v-v_*\rvert^2  -{\lvert v\rvert^2 \over 4} + {\lvert v_*\rvert^2 \over 4} }  
 -{1\over 2} \lvert v-v_*\rvert e^{ -{\lvert v\rvert^2 \over 2} }, \label{def:Kkernel} 
\eeqn
see, for instance, \cite[Theorem 7.2.1]{MR1307620} for a derivation of $k$, up to a conjugate change of scale. On the other hand, we have 
$$
\nu = \nu(v):= \int_{\R^3}\int_{\mathbb{S}^{2}} \BB \,  \MM(v_*)d\sigma dv_* ,
$$ 
and there are constants $\nu_0, \nu_1 >0$ such that 
\be\label{eq:Controlnu}
\nu_0\leq \nu_0\la v\ra \leq\nu(v)\leq \nu_1\la v\ra,
\ee
where we define $\la v\ra:= \sqrt{1+\lvert v\rvert^2}$, and we refer to \cite[Section 4]{MR3779780} for a derivation of this.

We now observe that, for any regular enough functions $G,H, \varphi: \R^3 \to \R$, the Boltzmann collision operator classically satisfies 
\be\label{eq:pre_ConservationLaws}
\int_{\R^3} \QQ(G,H) \,  \varphi 
= \frac 18 \int_{\R^3}\int_{\R^3} \int_{\mathbb{S}^2} \BB  \left( G_*' H' +H_*'G' -G_*H - G H_* \right) \left( \varphi+\varphi_* - \varphi'-\varphi'_*\right) , 
\ee
where we have used the shorthands
$$
\phi=\phi(v), \quad \phi_* = \phi(v_*), \quad \phi' = \phi(v'), \quad \phi'_* = \phi(v'_*),
$$
and we recall that $v'$ and $v'_*$ are given in Subsection \ref{sec:Framework}. The interested reader can consult the derivation of \eqref{eq:pre_ConservationLaws} in \cite[Section 3.1]{MR1307620}.

In particular, if we set $\R^3 \ni v = (v_1, v_2, v_3)$, then \eqref{eq:pre_ConservationLaws} implies that choosing $\varphi = \varphi(v)$ to be either $1, v_1, v_2, v_3$ or $\lv v\rv^2$ 
\be\label{eq:ConservationLaws}
\int_{\R^3} \QQ(G,H) (v)\,  \varphi (v) \, \dv = 0.
\ee

Therefore, \eqref{eq:ConservationLaws} implies that for the previous choices of $\varphi$ it follows that
\be\label{eq:ConservationLawsCCC}
\int_{\R^3} (\CCC g)(v)\,  \varphi(v) \, \dv = 0. 
\ee

We now take as a momentary framework the Hilbert space $L^2_{\MM^{-1/2}}(\R^3)$ endowed with the scalar product 
$$
\la g, h \ra_{L^2_{\MM^{-1/2}}(\R^3)}:= \int_{\R^3} g(v)\, h(v)\, \MM^{-1}(v) \, \dv,
$$
and its associated norm as defined in \eqref{def:Lebesgue_weighted_spaces}. We observe that \cite[Theorem 7.2.4]{MR1307620} implies that we can set ${\rm Dom}(\CCC):= L^2_{\MM^{-1/2}}(\R^3)$, that $\CCC$ is a closed operator on its domain, and \eqref{eq:ConservationLawsCCC} further gives that
$$
{\rm ker} \, (\CCC) = {\rm span}\, \{ \MM, \, v_1 \MM, \, v_2 \MM, \, v_3 \MM, \, \lv v\rv^2 \MM\}.
$$
This motivates the definition of $\Pi g$ as the projection of $g\in {\rm Dom}(\CCC) $ onto $\mathrm{ker} (\CCC)$ given by 
\be\label{def:pi_Hypo_decay}
\Pi g = \left( \int_{\R^3} g (w) \, \d w \right) \MM
+ \left( \int_{\R^3} w g (w) \, \d w   \right) \cdot v \MM 
+ \left( \int_{\R^3} \frac{|w|^2-3}{\sqrt{6}} \, g (w) \, \d w   \right) \frac{|v|^2-3}{\sqrt{6}} \, \MM.
\ee
We also note that $\CCC$ is self-adjoint on its domain and negative, i.e. 
\beqn
\la \CCC g , g \ra_{L^2_{\MM^{-1/2}} (\R^3)} \le 0,
\eeqn
so that its spectrum is included in $\R_{-}$, and \eqref{eq:ConservationLawsCCC} holds true for any { $g \in \mathrm{Dom} (\CCC)$}. 
Furthermore, $\CCC$ satisfies a \emph{microscopic coercivity} estimate, more precisely  \cite[Theorem 1.1]{MR2231011} gives that there is $\kappa_0 >0$ such that for any $g \in \mathrm{Dom} (\CCC)$ one has
\be\label{eq:MicroCoercivity}
\la-\CCC g , g \ra_{L^2_{\MM^{-1/2}}(\R^3)} \ge \kappa_0 \|  g^\perp \|_{L^2_{\MM^{-1/2}} (\R^3)}^2,
\ee
where $g^\perp:= g - \Pi g$. Finally, we note that for any polynomial function $\phi=\phi(v): \R^3 \to \R$ of degree less or equal to $4$, we have  $\MM \phi \in \hbox{Dom}(\CCC)$, and using again \cite[Theorem 7.2.4]{MR1307620} we deduce that there exists a constant $C_\phi \in (0,\infty)$ such that 
\be\label{eq:ConditionBddCCCphi}
\bigl\| \CCC (\phi \MM) \bigr\|_{L^2_{\MM^{-1/2}}(\R^3)}  \le \bigl\| \phi \MM \bigr\|_{L^2_{\MM^{-1/2}}(\R^3)} <C_\phi.
\ee

Furthermore, we note that there is mass conservation along the boundary. Indeed, we observe that
\be\label{eq:ConservationMassBdy}
\begin{aligned}
\int_{\R^3} \gamma g (n_x\cdot v) \dv &= \int_{\R^3} \gamma_+ g (n_x\cdot v)_+ \dv  - \int_{\R^3} \gamma_- g (n_x\cdot v)_- \dv \\
&= \int_{\R^3} \gamma_+ g (n_x\cdot v)_+ \dv  - \int_{\R^3}  \RRR_\TTT\gamma_+ g (n_x\cdot v)_- \dv    \\
&= \int_{\R^3} \gamma_+ g (n_x\cdot v)_+ \dv  - (1-\iota) \int_{\R^3}  \gamma_+ g (n_x\cdot v)_+ \dv    \\
&= \left(1-(1- \iota) - \iota\right) \int_{\R^3} \gamma_+ g (n_x\cdot v)_+ \dv  =0
\end{aligned}
\ee
where we have used the change of variables $v\mapsto \VV_x v$ on the integral involving the specular reflection, together with the fact that $\lv n_x \cdot \VV_x v\rv = \lv n_x \cdot v\rv$, and $\widetilde \MMM_\TTT = 1$ for every $x\in \partial\Omega$, to obtain the last line.
Another way of stating this is that
$$
\int_{\R^3} \gamma_+ g  (n_x\cdot v)_+ \dv  = \int_{\R^3} \gamma_- g (n_x\cdot v)_- \dv ,
$$
which, in words, means that the amount of particles leaving the domain ($\Sigma_+$) is equal to the amount entering ($\Sigma_-$) for any boundary point $x\in \partial \Omega$, at any moment of time $t\in (0,\infty)$. In practice, this represents a system where no particle leaves the domain, nor any particle is introduced. 

In consequence, we note that the conservation laws \eqref{eq:ConservationLawsCCC} together with \eqref{eq:ConservationMassBdy}
imply that, at least formally, Equation~\eqref{eq:Hypo_decay} conserves mass, i.e. for any solution $g$ of Equation~\eqref{eq:Hypo_decay} we have $\lla g_t\rra_{\OO} = \lla g_0\rra_{\OO} = 0$ for every $t\geq 0$.

\subsection{Elliptic equations}\label{ssec:Hypo_Elliptic_Hypo_decay}
We now recall some results on the well-posedness and regularity gain for some classes of elliptic equations in both smooth and cylindrical domains. 

\subsubsection{Poisson problem} Let $\xi\in L^2(\Omega)$ and consider the Poisson equation
\begin{equation}\label{eq:PoissonEquation}
\left\{\begin{array}{rcll}
-\Delta u  &= &\xi & \text{ in }\Omega  \\
(2-\alpha(x)) \partial_n u + \alpha(x) u &=& 0& \text{ on } \partial\Omega,
\end{array}\right.
\end{equation}
where $\alpha$ is chosen satisfying one of the following conditions:
\begin{enumerate}
    \item[(P1)]\label{item:P1} either $\alpha = \iota$,
    \item[(P2)]\label{item:P2} or we assume $ \lla \xi \rra_{\Omega} =0$, and we take $\alpha\equiv 0$. 
\end{enumerate}
We then consider the functional spaces
\begin{equation*}
    V_1:= H^1(\Omega) \qquad\text{ and }\qquad V_2:=\left\{ u\in H^1(\Omega), \, \lla u \rra_{\Omega} =0\right\} ,
\end{equation*}
from where we define
\begin{equation*}
        V_k:=\left\{\begin{array}{cc}
        V_1 & \text{if \ref{item:P1} holds,} \\
        V_2 & \text{if \ref{item:P2} holds.}
    \end{array}\right.
\end{equation*}
and we have then the following well-posedness and regularity result from \cite[Theorem 2.1]{MR4581432} for the case of smooth domains, and  \cite[Theorem 2.3]{BE_iso} for the case of cylindrical domains. 

\begin{theo}\label{theo:PoissonRegularity}
Assume that either Assumption \ref{item:H1} or Assumption \ref{item:H2} holds, and assume that either \ref{item:P1} or \ref{item:P2} holds. For any $\xi\in L^2(\Omega)$ there exists a unique $w \in H^2(\Omega)\cap V_k$ variational solution to the Poisson Equation~\eqref{eq:PoissonEquation}, i.e. there holds
\be\label{eq:PoissonVariationalSln}
\int_{\Omega} \grad w(x) \cdot \grad v(x) \, \dx + \int_{\partial\Omega} {\alpha(x) \over 2-\alpha(x) } w(x)v(x) \, \d\sigma_x  = \int_{\Omega} \xi(x) \, v(x) \dx \qquad \forall v\in V_k.
\ee
Furthermore, there holds
\bean
	\lVert w\rVert_{H^2(\Omega)} &\leq& C  \lVert \xi\rVert_{L^2(\Omega)},\label{H2Schr} \\
	\lvv  w\rvv_{H^1(\Omega)} &\leq & C\left( \int_{\Omega} \grad w(x) \cdot \grad w(x) \, \dx + \int_{\partial\Omega} {\alpha(x) \over 2-\alpha(x) } w(x)w(x) \, \d\sigma_x \right),
\eean
for some constructive constant $C>0$.
\end{theo}

\subsubsection{Lamé system of equations} We now consider $\Xi\in L^2(\Omega)$ and study the Lamé system of elliptic equations
\begin{equation}
	\left\{\begin{array}{rlll}
		-\Div (\nabla^s U) &=& \Xi &\text{ in }\Omega\\
		U\cdot n(x)&=&0 &\text{ on }\partial\Omega\\
		(2-\iota(x))[\nabla^sU \cdot n(x) - (\nabla^sU: n(x)\otimes n(x))n(x)]+\iota(x) U&=&0 &\text{ on }\partial\Omega.
	\end{array}\right. \label{eq:LameSystem}
\end{equation}
We consider the functional space
\be
\UUU(\Omega):=\{ U\in H^1(\Omega), \, U(x) \cdot n(x) = 0 \,\text{ on } \partial\Omega \}, 
\ee 
and we have the following well-posedness and regularity result from \cite[Theorem 2.11]{MR4581432} for the case of smooth domains, and \cite[Theorem 2.4]{BE_iso} for the case of cylindrical domains. 
\begin{theo}\label{thm:LameRegularity}
Assume that either Assumption \ref{item:H1} or Assumption \ref{item:H2} holds. For any function $\Xi\in L^2(\Omega)$ there is a unique variational solution $W \in \UUU(\Omega) \cap  H^2(\Omega)$ to the Lamé system \eqref{eq:LameSystem}, i.e. for any $V\in \UUU(\Omega)$ there holds
\be\label{eq:WeakFormulationLame}
\int_{\Omega}\nabla^sW (x):\nabla^sV(x) \dx + \int_{\partial\Omega} {\iota(x) \over 2-\iota(x) }\, W(x)\cdot V(x)   \d\sigma_x = \int_{\Omega} \Xi (x) \cdot V(x) \dx.
\ee
Furthermore, there holds
\bean
    \lVert W\rVert_{H^2(\Omega)} &\leq& C\lVert \Xi\rVert_{L^2(\Omega)},\\
    \lVert W\rVert_{H^1(\Omega)} &\leq & C\left( \int_{\Omega}\nabla^sW (x):\nabla^sW(x) \dx + \int_{\partial\Omega} {\iota(x) \over 2-\iota(x) }\, W(x)\cdot W(x)   \d\sigma_x \right)
\eean
for some constructive constant $C>0$.
\end{theo}

\subsection{Hypocoercivity norm}\label{ssec:Hypo_Perturbed_Hypo_decay}
During the rest of this subsection we consider $g\in \HH$ solution of Equation~\eqref{eq:Hypo_decay} and we note that the computations are done in the sense of a priori estimates. 

\medskip
We define the associated mass, momentum and energy operators respectively by 
\be\label{def:MassMomentumEnergy_Hypo_decay}
\varrho  (x) := \int_{\R^3} g(v) \d v, \quad \mu(x):= \int_{\R^3} v g(v) \d v, \quad  \text{and}\quad  \EE (x) := \int_{\R^3} {\lv v\rv^2 -3 \over \sqrt {6} } g(v) \d v.
\ee
Furthermore, we recall the definition of the projection operator $\Pi$ given by \eqref{def:pi_Hypo_decay}, as well as  $g^\perp = g-\Pi g$. Altogether, this implies that we have the decomposition  
\be
g = \Pi g + g^\perp =  \varrho\,  \MM+ \mu \cdot v \, \MM + \EE\,  {(\lvert v\rvert^2-3)\over \sqrt{6}}\, \MM + f^\perp ,\label{eq:DecompositionMicroMacro_Hypo_decay}
\ee 
which implies that
\be\label{eq:NormDecompositionMicroMacro_Hypo_decay}
\lVert g\rVert^2_{ \HH} = \lVert \varrho \rVert^2_{ L^2(\Omega)} + \lVert  \mu \rVert^2_{ L^2(\Omega)}  + \lVert \EE \rVert^2_{ L^2(\Omega)}    + \lVert g^\perp\rVert^2_{ \HH} .
\ee

We further note that the mass conservation at the boundary \eqref{eq:ConservationMassBdy} is nothing but saying that  
\be\label{eq:mcdotn=0}
\mu (x) \cdot n(x) = 0 \quad \text{for every }  x \in \partial \Omega.
\ee
Moreover, we note that the mass conservation of Equation~\eqref{eq:Hypo_decay} translates into $\lla \varrho \rra_\Omega = 0$.

\smallskip
During this subsection we construct the norm, equivalent to that of $\HH$, in which the operator $\LL$ exhibits a coercive behavior in isothermal settings, see \cite{MR4581432, BE_iso}. 

\begin{theo}\label{theo:Pert_hypo_Hypo_decay}
Assume that either Assumption \ref{item:H1} or Assumption \ref{item:H2} holds. There exists a scalar product $\la\!\la \cdot , \cdot \ra \! \ra$ on the Hilbert space $\HH$ so that the associated norm $\Nt \cdot \Nt$ is equivalent to the usual norm $\| \cdot \|_{\HH}$, and for which the linear operator $\LL$ satisfies: there are positive constants $\kappa, \kappa'  >0$ such that 
\begin{equation}\label{eq:pert_coercivity-LLL} 
\la \! \la - \LL g , g \ra\!\ra \ge \kappa \,   \Nt g\Nt^2 -  \kappa'  \theta_0 \left\lVert  (\iota)^{1/2} (\widetilde{\gamma_+ g} ) \right\rVert^2_{L^2(\partial \Omega)} ,
\end{equation}
for any $g \in \mathrm{Dom}(\LL)$ satisfying the non-isothermal boundary condition of Equation~\eqref{eq:Hypo_decay}.
\end{theo}

The rest of this subsection is devoted to constructing the scalar product $\lla \cdot, \cdot\rra$ generating the norm $\lvvv \cdot \rvvv$, and we note that we follow the structure of the proof of \cite[Theorem 1.1]{MR4581432}. 

We proceed then as follow: In Sub-subsection~\ref{ssec:Microscopic_part_Hypo_decay} we extend the microscopic coercivity of the collision operator $\CC$ to the full operator $\LL$, generating a perturbed microscopic coercivity estimate. During Sub-subsection~\ref{ssec:BoundaryTerms_Hypo_decay} we obtain a general expression to handle boundary terms that often appear in the sequel. 
Following this, during Sub-subsections \ref{ssec:EnergyFunctional_Hypo_decay}, \ref{ssec:MomentumFunctional_Hypo_decay}, and \ref{ssec:MassFunctional_Hypo_decay}, we construct the necessary components to add to the usual scalar product in $\HH$ in order to control the macroscopic components of $g$. 
Finally, with all the necessary estimates, we conclude the proof of Theorem~\ref{theo:Pert_hypo_Hypo_decay} in Sub-subsection~\ref{ssec:conclusion_Hypo_decay}.

\subsubsection{Microscopic estimate on $\LL$} \label{ssec:Microscopic_part_Hypo_decay}

The objective of this subsection is to extend the microscopic coercivity property \eqref{eq:MicroCoercivity} exhibited by the collision operator $\CCC$ to \hbox{$\LL = -v\cdot \grad_x + \CCC$}, up to a perturbation coming from the non-isothermal conditions at the boundary.

\begin{lem}\label{lem:MicroCoercivity_Hypo_decay}
There is a constant $C>0$ such that
$$
\la -\LL g, g\ra_{\HH} \geq \kappa_0  \lVert g^\perp \rVert^2_{\HH} + \frac {1}4 \left \lVert \sqrt{\iota (2-\iota)}\DDD_1^\perp \gamma_+ g \right\rVert^2_{\partial\HH_+}  -  \theta_0\, C \left\lVert   (\iota)^{1/2}   \left( \widetilde {\gamma_+ g}\right) \right\rVert^2_{L^2(\partial\Omega)}  ,
$$
where we have defined $\DDD_1^\perp := Id - \DDD_1$, and $\partial \HH_+ = L^2(\Sigma_+,\, \MM^{-1} (v) \, \lvert n(x)\cdot v\rvert\,  \, \dv \d\sigma_x)$. 
\end{lem}

\begin{proof}
Using the Stokes theorem we compute
\beqn
 \la \LL g, g \ra_{\HH}  = -\frac {1}{2}  \int_{\OO} v\cdot \grad _x (g^2)  \MM^{-1} +     \la \CC g, g\ra_{\HH}  \leq  -\frac {1}{2}  \int_{\Sigma} (\gamma g)^2 \MM^{-1} (n_x\cdot v)  - \kappa_0   \|  g^\perp \|_{\HH}^2,
\eeqn
where we have used the microscopic coercivity property of $\CC$ \eqref{eq:MicroCoercivity} to obtain the last inequality. 
It is left then to control the first term of the right hand side above estimate.

We observe that the boundary condition of Equation~\eqref{eq:Hypo_decay} implies that 
\bean
\TT := -\int_{{\Sigma}} (\gamma g)^2 \MM^{-1} (n_x\cdot v) &=&- \int_{{\Sigma_+}} (\gamma_+ g)^2 \MM^{-1} (n_x\cdot v)_+ + \int_{{\Sigma_-}} (\gamma_- g)^2 \MM^{-1} (n_x\cdot v) _- \\
&=& - \int_{{\Sigma_+}} (\gamma_+ g)^2 \MM^{-1} (n_x\cdot v)_+ \\
&&+ \int_{{\Sigma_-}}  \left( (1-\iota) \SSS\gamma_+ g +\iota \DDD_\TTT \gamma _+ g  \right)^2 \MM^{-1} (n_x\cdot v) _-. 
\eean
Applying the change of variables $v\mapsto \VV_x v $ and using that $\DDD_\TTT \gamma_+ f (x,\VV_x v) = \DDD_\TTT \gamma_+ f(x,v)$, and $|n(x) \cdot \VV_x v| = |n(x) \cdot v|$,  it follows that
\bean
\TT &=&-\int_{{\Sigma_+}} (\gamma_+ g)^2 \MM^{-1} (n_x\cdot v)_+ + \int_{{\Sigma_+}}  \left( (1-\iota) \gamma_+ g +\iota \DDD_\TTT \gamma _+ g  \right)^2 \MM^{-1} (n_x\cdot v) _+ \\
&\leq & \int_{{\Sigma_+}} (\gamma_+ g)^2 \MM^{-1} (n_x\cdot v)_+ + \int_{{\Sigma_+}}  \left( (1-\iota) (\gamma_+ g)^2 +\iota (\DDD_\TTT \gamma _+ g)^2  \right) \MM^{-1} (n_x\cdot v) _+ \\
&\leq & \int_{{\Sigma_+}} -\iota (\gamma_+ g)^2 \MM^{-1} (n_x\cdot v)_+ + \int_{{\Sigma_+}} \iota (\DDD_\TTT \gamma _+ g)^2  \MM^{-1} (n_x\cdot v) _+ =: \TT_1 + \TT_2,
\eean
where we have used a convexity inequality for the function $x\mapsto x^2$ to obtain the second inequality.  
We now introduce the decomposition 
\be\label{eq:HypoBoundDecomposition_Hypo_decay}
\gamma_+ g = \DDD_1\gamma_+ g + \DDD_1^\perp \gamma_+ g,
\ee
where we recall that $\DDD_1^\perp = Id - \DDD_1$, and we note that $\DDD_1 \gamma_+ g \perp \DDD_1^\perp \gamma_+ g$ in $\partial\HH_+$. On the one hand, we observe that this implies that
\be\label{eq:Boundary_Estimate_TT1_Hypo_decay}
\TT_1= -\int_{{\Sigma_+} } \iota \left( \DDD_1 \gamma_+ g \right)^2 \MM^{-1} (n_x\cdot v)_+ -\int_{{\Sigma_+} } \iota \left( \DDD_1^\perp \gamma_+ g \right)^2 \MM^{-1} (n_x\cdot v)_+ .
\ee
On the other hand, expanding the term $\TT_2$ and using the very definition of the diffusive boundary condition, we also observe that
$$
\begin{aligned}
\TT_2 &=  \int_{\Sigma_+} \iota (\DDD_1 \gamma _+ g)^2  \MM^{-1} (n_x\cdot v) _+  + \int_{\Sigma_+} \iota \left( \MMM_{\TTT}^2 - \MMM_1^2 \right) (\widetilde{\gamma _+ g})^2  \MM^{-1} (n_x\cdot v) _+ \\
& =:\TT_{2,0} +  \TT_{2,1}.
\end{aligned} 
$$
We define
\be\label{eq:PolyHypo_decay}
P = P(T)  := \int_{\R^3}  \left( \MMM_{\TTT}^2 - \MMM_1^2 \right) \MM^{-1} (n_x \cdot v)_+ \dv =  \sqrt{2\pi}  \left( {1\over {\TTT}^2 (2-{\TTT})^2  - 1} \right) ,
\ee
where the last equality has been obtained by a direct computation of the integral, and we note that, since
$$
({\TTT}(2-{\TTT}))^2 = \left( 1-({\TTT}-1)^2\right)^2 \leq 1,
$$
we have thatit follows that $P (z)\geq 0$ for every $z\in (0,2)$, and $P(z) = 0$ if and only if $z=1$. 
Moreover, using the decomposition ${\TTT} = 1+ \theta$ we have that
\beqn\label{eq:BdyEstPTheta_Hypo_decay}
P({\TTT}) = \sqrt{2\pi}  \left( {1\over (1+\theta)^2  (1-\theta)^2}  - 1\right)  \lesssim    1-  (1-\theta^2)^2  \lesssim \theta^2 \lesssim \theta_0^2,
\eeqn
where we have used that $\lvv \theta\rvv_{L^\infty(\partial {\Omega})}\leq \theta_0$ with $\theta_0 \in (0,1/8)$.
Therefore, we deduce that 
$$
\TT_{2,1} \lesssim  \int_{\partial\Omega} \iota \,  \theta_0^2 \,  (\widetilde{\gamma _+ g})^2.
$$
Therefore, we have obtained that there is a constant $C>0$such that 
\be\label{eq:Boundary_Estimate_TT2_Hypo_decay}
\begin{aligned}
\TT_2 &\leq   \int_{{\Sigma_+}} \iota (\DDD_1 \gamma _+ g)^2  \MM^{-1} (n_x\cdot v) _+  + C \, \theta_0 \left\lVert  (\iota)^{1/2}   \left( \widetilde {\gamma_+ g}\right) \right\rVert^2_{L^2(\partial{\Omega})} .
\end{aligned}
\ee
Finally, recalling that $\TT \leq \TT_1 + \TT_2$, and using \eqref{eq:Boundary_Estimate_TT1_Hypo_decay} and \eqref{eq:Boundary_Estimate_TT2_Hypo_decay}, we have that
\bean
\TT &\leq&  -\int_{{\Sigma_+} } \iota \left( \DDD_1^\perp \gamma_+ g \right)^2 \MM^{-1} (n_x\cdot v)_+ + C \, \theta_0 \left\lVert  (\iota)^{1/2}   \left( \widetilde {\gamma_+ g}\right) \right\rVert^2_{L^2(\partial{\Omega})}  \\
&\leq &  - \frac 12 \int_{{\Sigma_+} } \iota (2-\iota) \left( \DDD_1^\perp \gamma_+ g \right)^2 \MM^{-1} (n_x\cdot v)_+ + C \, \theta_0 \left\lVert  (\iota)^{1/2}   \left( \widetilde {\gamma_+ g}\right) \right\rVert^2_{L^2(\partial{\Omega})} ,
\eean
where we have employed the inequality $\iota \geq \iota(2-\iota)/2$ to obtain the second line. This concludes the proof.
\end{proof}

\subsubsection{Boundary terms}\label{ssec:BoundaryTerms_Hypo_decay}
We now control the boundary terms which will be useful in the coming sections.

\begin{lem} \label{lem:BoundaryTerms_Hypo_decay}
Consider a function $\phi:\R^3\to \R$. For any $x\in \partial {\Omega}$ it follows that
$$
\begin{aligned}
\int_{\R^3} \phi (v) \, \gamma g \, (n_x\cdot v)\dv &= \int_{\R^3} \phi(v) \, \iota(x) \, \DDD_1^\perp \gamma_+ g\,  (n_x\cdot v)_+\dv  \\
&+ \int_{\R^3} \left[ \phi(v) - \phi(\VV_x v) \right] (1-\iota (x) ) \, \DDD_1^\perp \gamma_+ g \,  (n_x\cdot v)_+\dv \\
& + \int_{\R^3}  \left[ \phi(v) - \phi(\VV_x v) \right] \DDD_1 \gamma_+ g \,  (n_x\cdot v)_+\dv \\
& - \left( \widetilde{\gamma_+ g} \right)  \, \iota \left( \int_{\R^3}  \phi( v)   \, \left[\MMM_\TTT(x,v)-\MMM_1(v)\right] \, (n_x\cdot v)_- \dv \right)  
\end{aligned}
$$
\end{lem}

\begin{proof}[Proof of Lemma~\ref{lem:BoundaryTerms}]
The proof follows exactly the computations of \cite[Lemma 3.2]{MR4581432}, by writing $\DDD_\TTT \gamma_+ g = \DDD_1 \gamma_+ g + \left( \MMM_\TTT-\MMM_1\right) \widetilde{\gamma_+g}$, and following the computations accordingly. Thus, we skip it.
\end{proof}

\subsubsection{Energy functional}\label{ssec:EnergyFunctional_Hypo_decay}
We devote this sub-subsection to construct a functional in order to control the energy component of the macroscopic part $\Pi g$. 
We define the operator 
\be\label{eq:Energy_Hypo_decay}
\EE[h] := \int_{\R^3} \frac{|v|^2 - 3}{\sqrt{6}} \, h(v) \, \dv,
\ee
so that $\EE = \EE[g]$. 
We define $u[\EE]$ as the unique variational solution to the Poisson equation \eqref{eq:PoissonEquation} associated with $\xi = \EE \in L^2_x ({\Omega})$ and $\alpha = \iota$, i.e. the boundary condition \ref{item:P1} holds, thus the well-posedness of $u[\EE]$ is given by Theorem~\ref{theo:PoissonRegularity}.

We now introduce the vector $p = (p_i)_{1 \le i \le 3}$ defined by 
$$
p_i(v): = v_i \, \frac{(|v|^2 - 5)}{\sqrt{6}},
$$
and the associated moment functional  $M_p[g] = ( M_{p_i} [g])_{1 \le i \le 3}$ given by 
\begin{equation}\label{eq:def-Mp_Hypo_decay}
M_{p_i}[g] = \int_{\R^3} v_i \, \frac{(|v|^2 - 5)}{\sqrt{6}} \, g \, \dv.
\end{equation}

We first have the following lemma.

\begin{lem}\label{lem:ineqHypoEnergy_Hypo_decay}
It follows that 
    \be\label{eq:EnergyLLControl_Hypo_decay}
    \EE[\LL g] = -  \Div_x \mu - \sqrt{2\over 3} \Div_x M_p \quad \text{ and } \quad  M_p [g] = M_p [g^\perp].
    \ee
    Moreover, we have that
    \begin{multline}\label{eq:H1Energy_Hypo_decay}
    \lVert u[\EE[\LL g]]\rVert_{H^1 (\Omega)} \lesssim   \lVert \mu \rVert_{L^2(\Omega)} +   \lVert f^\perp\rVert_{ \HH} +   \left\lVert \sqrt{\iota  (2-\iota)}\DDD_1^\perp \gamma_+ g \right\rVert_{\partial\HH_+} \\ +   (\theta_0)^2 \left\lvv   \sqrt{\iota (2-\iota)}   \left( \widetilde{ \gamma_+ g} \right) \right\rvv_{L^2(\partial \Omega)}.
    \end{multline}
\end{lem}

\begin{proof} 
We first note that, \eqref{eq:EnergyLLControl_Hypo_decay} follows by repeating exactly the arguments from \cite[Lemma 3.3]{MR4581432}. To prove \eqref{eq:H1Energy_Hypo_decay}, we set $u:= u[\EE[\LL g ]]$ and using Theorem \ref{theo:PoissonRegularity} we have that there is a constant $\eta>0$ for which it follows that
$$
\eta  \lvv u \rvv_{H^1_x(\Omega)}^2  \leq \lvv \grad_x u\rvv^2_{L^2_x(\Omega)} +\left \lvv \sqrt{\iota \over 2-\iota} \, u \right\rvv^2_{L^2_x(\partial \Omega)}.
$$
Using the variational formulation together with \eqref{eq:EnergyLLControl_Hypo_decay} we further have that
\begin{multline}
\lvv \grad_x u\rvv^2_{L^2_x(\Omega)} + \left\lvv \sqrt{\iota \over 2-\iota} \, u \right\rvv^2_{L^2_x(\partial \Omega)}  
 =- \int_{\Omega}\left(  \Div_x M_p [g] + \sqrt{2\over 3}\Div_x \mu\right) u\,  \dx \\
=  \int_{\Omega}\left(   M_p[g] + \sqrt{2\over 3} \, \mu\right) \grad_x u\,  \dx -  \int_{\partial\Omega}    M_p[g] \cdot n_x \, u\,   \d\sigma_x  \label{eq:Energy_variational_formulation_Hypo_decay}
\end{multline}
where we have used integration by parts and \eqref{eq:mcdotn=0} to obtain the second line. 

Employing now Lemma \ref{lem:BoundaryTerms_Hypo_decay} and noting that $\lv \VV_x v\rv = \lv v\rv$, we control the boundary integral as follows
$$
\begin{aligned}
 M_p[g] \cdot n_x  &= \int_{\R^3}  \frac{(|v|^2 - 5)}{\sqrt{6}}  \iota(x) \DDD_1^\perp \gamma_+ g (n_x\cdot v)_+\dv  \\
&\qquad \qquad  - \left( \widetilde{\gamma_+ g} \right) \left( \int_{\R^3}   \frac{(|v|^2 - 5)}{\sqrt{6}} \, \iota(x)  \, \left[ \MMM_\TTT(x,v) - \MMM_1(v)\right] \, (n_x\cdot v)_-\dv \right)  
\end{aligned}
$$
Using now \eqref{eq:PertTempBdy} we deduce that 
$$
\left\lv \int_{\R^3}   \frac{(|v|^2 - 5)}{\sqrt{6}} \, \iota(x)   \, \left[ \MMM_\TTT(x,v) - \MMM_1(v)\right] \, (n_x\cdot v)_-\dv \right\rv \lesssim \theta_0 \,  \iota(x) ,
$$
Therefore, using the Cauchy-Schwarz inequality, we further have that 
$$
\begin{aligned}
\left\lv \int_{\partial\Omega}    M_p[g] \cdot n_x \, u\,   \d\sigma_x \right\rv &\lesssim  \left\lVert \sqrt{\iota (2-\iota)}\DDD_1^\perp \gamma_+ g \right\rVert_{\partial\HH_+}  \left\lvv \sqrt{\iota\over 2-\iota }u\right\rvv_{L^2(\partial \Omega)}\\
&\qquad\qquad  +\theta_0 \, \left\lvv  \sqrt{\iota (2-\iota)}   \left( \widetilde{ \gamma_+ g} \right) \right\rvv_{L^2(\partial \Omega)}    \left\lvv \sqrt{\iota\over 2-\iota }u\right\rvv_{L^2(\partial \Omega)}.
\end{aligned}
$$
Furthermore, using the Young inequality we obtain that 
\begin{multline}\label{eq:Control_boundary_term_Energy_Hypo_decay}
\left\lv \int_{\partial{\Omega}}    M_p[g] \cdot n_x \, u\,   \d\sigma_x \right\rv \leq C  \left\lVert \sqrt{\iota (2-\iota)}\DDD_1^\perp \gamma_+ g \right\rVert_{\partial\HH_+}^2+ \frac 12 \left\lvv \sqrt{\iota\over 2-\iota }u\right\rvv_{L^2(\partial {\Omega})}^2 \\
+  C  \, (\theta_0)^2  \left\lvv  \sqrt{\iota (2-\iota)}  \left( \widetilde{ \gamma_+ g} \right) \right\rvv_{L^2(\partial {\Omega})}  ^2   ,
\end{multline}
for some constant $C>0$.
On the other hand, using that
\be\label{eq:estimMpHH_Hypo_decay}
\lv M_p [g]\rv \lesssim  \lvv g\rvv_{\HH_0},
\ee
and arguing exactly as in the proof of \cite[Lemma 3.3]{MR4581432} we have that 
\be\label{eq:Control_interior_term_Energy_Hypo_decay}
\left\lv \int_{\Omega}\left(   M_p[g] + \sqrt{2\over d} \, \mu\right) \grad_x u\,  \dx \right\rv \leq  C   \lvv g^\perp\rvv_{\HH}^2 + C\lvv \mu\rvv_{L^2({\Omega})}^2 + \frac {1}2 \lvv \grad_x u\rvv_{L^2({\Omega})}^2 ,
\ee
for some constant $C>0$.
We conclude the proof by putting \eqref{eq:Control_boundary_term_Energy_Hypo_decay} and \eqref{eq:Control_interior_term_Energy_Hypo_decay} into \eqref{eq:Energy_variational_formulation_Hypo_decay}. 
\end{proof}

Using the lemma above, and arguing in the spirit of \cite[Lemma 3.4]{MR4581432}, we obtain then the following estimate.

\begin{lem}\label{lem:EnergyCoercivity_Hypo_decay}
There are constants $\kappa_1, C_1>0$ such that 
    \begin{multline*}
  \langle -\grad_x u[\EE], M_p [\LL g]\rangle_{L^2({\Omega})} + \langle -\grad_x u[\EE[\LL g]], M_p [ g]\rangle_{L^2({\Omega})}\\
       \geq \kappa_1 \lVert\EE\rVert^2_{L^2({\Omega})} - C_1   \lVert \mu\rVert_{L^2({\Omega})} \lVert g^{\perp}\rVert_{ \HH} - C_1   \left\lVert \sqrt{\iota(2-\iota)} \DDD_1^\perp \gamma_+g\right\rVert^2_{\partial\HH_+} \\
       - C_1  \lVert g^{\perp}\rVert^2_{ \HH}  - C_1 (\theta_0)^2   \left\lVert  \sqrt{\iota(2-\iota)}  \, \left( \widetilde{\gamma_+ g}\right) \right\rVert^2_{L^2(\partial {\Omega})}  .
    \end{multline*}
\end{lem}

\begin{proof}
We define
$$
E_1 := \la -\grad_x u[\EE], M_p[\LL g]\ra_{L^2({\Omega})} \quad \text{ and } \quad E_2 := \langle -\grad_x u[\EE[\LL g]], M_p[ g]\rangle_{L^2({\Omega})}.
$$
On the one hand, by using \eqref{eq:H1Energy_Hypo_decay}, \eqref{eq:estimMpHH_Hypo_decay}, and the Young inequality we have that
$$
\beal
\lvert E_2 \rvert &\lesssim  \lVert u[\EE[\LL g]] \rVert_{H^1({\Omega})} \lVert  g^\perp\rVert_{\HH} \lesssim  \lVert \mu \rVert_{L^2({\Omega})} \lVert  g^\perp\rVert_{\HH}  +   \lVert g^\perp\rVert_{ \HH}^2 \\
&\qquad\qquad\qquad \qquad \qquad \qquad \qquad   +   \left\lVert \sqrt{\iota  (2-\iota)}\DDD_1^\perp \gamma_+ g \right\rVert_{\partial\HH_+}^2 + \theta_0  \left\lvv  \sqrt{\iota (2-\iota)}   \left( \widetilde{ \gamma_+ g} \right) \right\rvv_{L^2(\partial {\Omega})}^2  
\eeal
$$
 On the other, from the fact that $M_p [\LL g]  = M_p [- v\cdot \grad_x g] + M_p [ \CC g^\perp]$ we deduce that $E_1 = E_{1, 1} + E_{1,2}$ with
$$
E_{1,1} =  \left\la \partial_{x_i} u[\EE] ,  \partial_{x_j} \int_{\R^3} p_i(v) v_j g\dv \right\ra_{L^2({\Omega})}  \quad \text{ and } \quad
E_{1,2} =  \left\la -\grad_x u[\EE] , \int_{\R^3} p(v)  \CC g^\perp \dv \right\ra_{L^2({\Omega})},
$$
and we have used the convention of the sum over repeated indices in the definition of $E_{1,1}$. 
We then note that, from the fact that $\CC$ is a self-adjoint operator in $\HH_0$, we have that
$$
\int_{\R^3} p(v) \, \CC g^\perp\, \dv = \la g^\perp, \CC(\MM p)\ra_{\HH_0} \lesssim \lvv g^\perp\rvv_{\HH_0},
$$
where we have used the Cauchy-Schwarz inequality and \eqref{eq:ConditionBddCCCphi} to obtain the last inequality. Using this estimate we further deduce that
$$
\lvert E_{1,2} \rvert \lesssim  \lVert \grad_x u[\EE]\rVert_{L^2({\Omega})} \lVert g^\perp\rVert_{\HH} \lesssim \lVert \EE\rVert_{L^2({\Omega})}  \lVert g^\perp\rVert_{\HH} ,
$$
where we have used the regularity estimate from Theorem~\ref{theo:PoissonRegularity} to obtain the second inequality.

Moreover, using integration by parts on $E_{1,1}$, we have that
\bean
E^1_1 &=& -\left\la \partial{x_i} \partial{x_j} u[\EE] , \int_{\R^3} p_i(v) v_j g \, \dv \right\ra_{L^2({\Omega})} +  \int_{\Sigma} \left( \grad_x u[\EE] \cdot p(v) \right) \, \gamma g \, (n_x\cdot v) \dv\d\sigma_x \\
&=:& E_{1,3} + E_{1,4},
\eean
and we compute each of them separately. 
Repeating exactly the arguments from \cite[Lemma 3.4]{MR4581432} we obtain that there is a constant $C>0$ such that
$$
E_{1,3} \geq \frac 12 \left( 1 + \frac 2 d\right) \lVert \EE\rVert^2_{L^2({\Omega})} - C \lVert g^\perp \rVert^2_{\HH}.
$$
On the other hand, we employ Lemma \ref{lem:BoundaryTerms} to expand the boundary term $E_{1,4}$ as follows
\bean
E_{1,4} &=& \int_{{\Sigma_+}} \grad_x u[\EE] \cdot p(v)  \, \iota(x) \DDD_1^\perp \gamma_+ g (n_x\cdot v)_+\dv \\
&& + \int_{{\Sigma_+}}  \grad_x u[\EE] \cdot  \left[ p(v) - p(\VV_x v) \right] (1-\iota(x)) \DDD_1^\perp \gamma_+ g (n_x\cdot v)_+\dv \\
&& + \int_{{\Sigma_+}} \grad_x u[\EE] \cdot  \left[ p(v) - p(\VV_x v) \right] \DDD_1 \gamma_+ g (n_x\cdot v)_+\dv \\
&& -  \int_{\partial {\Omega}}  (\widetilde{\gamma_+f}) \iota(x)  \grad_x u[\EE]  \left(  \int_{\R^3}   p( v)  \, \left[ \MMM_\TTT(x,v)-\MMM_1(v)\right]  \,(n_x\cdot v)_-\dv \right) \d\sigma_x \\
&=:& E_{1,5} + E_{1,6} + E_{1, 7} + E_{1,8} .
\eean
and we estimate each of these terms separately. To control the first three terms above we repeat exactly the computations of \cite[Lemma 3.4]{MR4581432} and we have that 
$$
\lv E_{1,5} + E_{1,6} + E_{1, 7} \rv \lesssim \lVert \EE  \rVert_{L^2({\Omega})} \lVert \iota \DDD_1^\perp \gamma_+ f  \rVert_{\partial {\HH}_+}.
$$
To estimate the remaining terms we use \eqref{eq:PertTempBdy}, and we have that
$$
\lv E_{1,8} \rv \lesssim  \theta_0 \, \lvv \EE \rvv_{L^2({\Omega})} \lvv \iota  (\widetilde{\gamma_+ f }) \rvv_{L^2(\partial {\Omega})},
$$
where we have used the Cauchy-Schwarz inequality, and the regularity estimate from Theorem~\ref{theo:PoissonRegularity}. 
 Using now the Young inequality we deduce that there is a constant $C>0$ such that 
$$
\beal
\lv E_{1,5} + E_{1,6} + E_{1, 7} + E_{1,8} \rv & \leq  \frac 14 \left( 1+ \frac 2d\right)  \lvv \EE \rvv_{L^2({\Omega})} ^2 + C \lVert \iota \DDD_1^\perp \gamma_+ g  \rVert_{\partial {\HH}_+}^2 \\
&\qquad  \qquad\qquad \qquad  + (\theta_0)^2  \, C\,  \left\lvv \iota  (\widetilde{\gamma_+ g }) \right\rvv_{L^2(\partial {\Omega})}^2 .
\eeal
$$
We conclude the proof by putting all the previous estimates together and by using the Young inequality together with the fact that $\iota \leq \sqrt{\iota/(2-\iota)}$. 
\end{proof}

\subsubsection{Momentum functional}\label{ssec:MomentumFunctional_Hypo_decay}
We devote this subsection to construct a functional in order to control the momentum component of the macroscopic part $\Pi g$.  

We denote 
\be\label{eq:Momentum_Hypo_decay}
\mu[h] := \int_{\R^3}  v h \, \dv , 
\ee
so that $\mu = \mu[g]$.
We also consider $U[\mu]$ the unique variational solution to the Lamé system of equations \eqref{eq:LameSystem} associated to $\Xi = \mu$, given by Theorem~\ref{thm:LameRegularity}. 

We consider the matrix $q_{ij} = (q_{ij})_{1 \le i , j \le 3}$ given by $q_{ij}(v) = v_i v_j - \delta_{ij}$,
and we define the associated moment functional $M_q[g] = (M_{q_{ij}} [g])_{1 \le i , j \le 3}$ as
\begin{equation}\label{eq:def-Mq_Hypo_decay}
M_{q_{ij}}[g] = \int_{\R^3} (v_i v_j - \delta_{ij}) g \, \dv.
\end{equation}

We first have the following result.
\begin{lem}\label{lem:ineqHypoMomentum_Hypo_decay}
There holds 
\be\label{eq:MomentumLLControl_Hypo_decay}
 \mu[\LL g] = - \grad_x \varrho -  \Div_x M_q[g] \quad \text{ and } \quad M_q[g] = \sqrt{2\over 3}\, \EE I_3  + M_q[g^\perp].
\ee
Furthermore, we have that 
\begin{multline}\label{eq:H1Momentum_Hypo_decay}
\lVert U[\mu[\LL g]]\rVert_{H^1({\Omega})} \lesssim  \lVert \varrho\rVert_{L^2({\Omega})} +  \lVert \EE\rVert_{L^2({\Omega})}  + \lVert g^\perp\rVert_{ {\HH}}  \\
 +  \left\lVert \sqrt{\iota(2-\iota)} \DDD_1^\perp \gamma_+ g \right \rVert_{\partial {\HH}_+} 
 +   (\theta_0)^2 \, \left\lvv   \sqrt{\iota (2-\iota)}   \left( \widetilde{ \gamma_+ g} \right) \right\rvv_{L^2(\partial {\Omega})}.
\end{multline}
\end{lem}

\begin{proof}
We first note that the proof of \eqref{eq:MomentumLLControl_Hypo_decay} follows by exactly repeating the arguments from \cite[Lemma 3.5]{MR4581432}. 

We concentrate then on the proof of \eqref{eq:H1Momentum_Hypo_decay}. Let $U := U[\mu[\LL g]]$ be the unique variational solution to \eqref{eq:LameSystem} associated to $\Xi = \mu[\LL g]$, which is given by Theorem~\ref{thm:LameRegularity}. 
Furthermore, using the regularity result from Theorem~\ref{thm:LameRegularity} we have that 
\begin{equation}\label{eq:UmLf1_Hypo_decay}
\lambda \| U \|_{H^1_x({\Omega})}^2  
\le \displaystyle \| \nabla^s U \|_{L^2_x({\Omega})}^2 
+ \left\lvv\sqrt{\frac{\iota}{2-\iota}} \, U \right\rvv_{L^2_x(\partial{\Omega})}^2, 
\end{equation}
for some $\lambda >0$. Moreover, from the variational formulation \eqref{eq:WeakFormulationLame} together with \eqref{eq:MomentumLLControl_Hypo_decay}, we obtain
\begin{equation}\label{eq:UmLf2_Hypo_decay}
\begin{aligned}
& \lvv \nabla^s U \rvv_{L^2_x({\Omega})}^2 
+ \left\lvv \sqrt{\frac{\iota}{2-\iota}} \, U \right\rvv_{L^2_x(\partial{\Omega})}^2 
= -   \int_{{\Omega}} (\nabla_x \varrho + \nabla_x \cdot M_q[g] ) \cdot U \, \dx\\
&\quad 
= \int_{{\Omega}} \varrho I_3 : \nabla U \, \dx
+ \int_{{\Omega}} M_q[g] : \nabla U \, \dx  - \int_{\partial {\Omega}}  \varrho \, n(x) \cdot U \, \d\sigma_{\! x} 
-  \int_{\partial {\Omega}} M_q[g] n(x) \cdot U  \, \d\sigma_{\! x} \\
&\quad 
=  \int_{{\Omega}} \varrho I_3 : \nabla^s U \, \dx
+ \int_{{\Omega}} \left(  \sqrt{\frac{2}{d}} \EE I_3 + M_q[g^\perp] \right) : \nabla^s U \, \dx 
-  \int_{\partial {\Omega}}  M_q[g] \,  n(x) \cdot U \, \d\sigma_{\! x} ,
\end{aligned}
\end{equation}
where we have performed an integration by parts to obtain the second line, we have used the fact that $U \cdot n(x) = 0$ since $U \in \UUU$ and \eqref{eq:H1Momentum_Hypo_decay} in the last one.

We now handle the boundary term in the last equation. Using Lemma~\ref{lem:BoundaryTerms_Hypo_decay} we have that, for any $x \in \partial{\Omega}$, it follows that
$$
\begin{aligned}
\int_{\partial {\Omega}} M_q[g] n_x \cdot U \d\sigma_x 
&=   \int_{{\Sigma}}  v_i v_j g \, n_j(x) U_i   \, \dv\d\sigma_x
-  \int_{{\Sigma}}  g n_i(x)  U_i \, \dv\d\sigma_x = \int_{\R^3}  g  (v \cdot U)  (n_x\cdot v) \, \dv\d\sigma_x  \\
&= \iota(x) \int_{{\Sigma_+}}    \DDD_1^\perp \gamma_+ g (v \cdot U)   (n_x\cdot v)_+ \, \dv\d\sigma_x  +\int_{{\Sigma_+}}   ( v - \VV_x v ) \cdot U  \DDD_1 \gamma _+ g  \, (n_x\cdot v)_+ \, \dv\d\sigma_x \\
&\quad 
+ \int_{{\Sigma_+}}  ( v - \VV_x v ) \cdot U (1-\iota(x))  \DDD_1^\perp \gamma_+ g  \, (n_x\cdot v)_+ \, \dv\d\sigma_x   \\
&\quad 
-\int_{\partial {\Omega}} \iota(x) (\widetilde{\gamma_+g}) \left( \int_{\R^3}  (v\cdot U) \left[ \MMM_\TTT(x,v) - \MMM_1\right]  (n_x\cdot v)_- \dv \right) \, \dv\d\sigma_x \\
&\quad
=: m_1 + m_2 + m_3 + m_4 , 
\end{aligned}
$$
where we have used again that $U \cdot n(x) = 0$ on the first line. 
We first observe that repeating the exact same computations as the proof of \cite[Lemma 3.5]{MR4581432} we have that
\be\label{eq:control_mi_123}
\lv m_1 + m_2 + m_3\rv \lesssim \left\lvv  \sqrt{\iota\over 2-\iota} U  \right\rvv_{L^2(\partial{\Omega})} \left\lvv  \sqrt{\iota(2-\iota)}  \DDD_1^\perp \gamma_+ f \right\rvv_{\partial {\HH}_+}
\ee

To control the remaining terms we use \eqref{eq:Controlpsi} together with the Cauchy-Schwarz inequality and we deduce that 
\be\label{eq:control_m4}
\lv m_4 \rv \lesssim \theta_0 \, \left\lvv \sqrt{\iota\over 2-\iota} U \right\rvv_{L^2(\partial {\Omega})}  \left\lvv  \sqrt{\iota(2-\iota)} (\widetilde{\gamma_+g}) \right\rvv_{L^2(\partial {\Omega})} +  \theta_0 \,  \left\lvv \sqrt{\iota\over 2-\iota} U \right\rvv_{L^2(\partial {\Omega})} .
\ee
Putting together \eqref{eq:control_mi_123} and \eqref{eq:control_m4}, and using the Young inequality  we deduce that there is a constant $C>0$ such that 
\begin{multline*}
\left\lv \int_{\partial {\Omega}} M_q[g] (n_x \cdot U) \d\sigma_x  \right\rv \leq \frac 12 \left\lvv \sqrt{\iota\over 2-\iota} U \right\rvv_{L^2(\partial {\Omega})}^2  + C \left\lvv  \sqrt{\iota(2-\iota)}  \DDD_1^\perp \gamma_+ g \right\rvv_{\partial {\HH}_+}^2 \\
+  (\theta_0)^2 \, C \left\lvv  \sqrt{\iota(2-\iota)} (\widetilde{\gamma_+g}) \right\rvv_{L^2(\partial {\Omega})}^2  ,
\end{multline*}
We conclude the proof by repeating exactly the arguments from \cite[Lemma 3.5]{MR4581432} to control the interior terms of \eqref{eq:UmLf2_Hypo_decay} and putting everything together.
\end{proof}

Using this, and arguing in the spirit of \cite[Lemma 3.6]{MR4581432}, we further have the following lemma.
\begin{lem}\label{lem:MomentumCoercivity_Hypo_decay}
There are constants $\kappa_2,C_2>0$ such that
\begin{multline*}
 \langle -\grad_x U[\mu], M_q[\LL g]\rangle_{L^2({\Omega})} + \langle -\grad_x U[\mu[\LL g]], M_q[ g]\rangle_{L^2({\Omega})} \geq 
 	\kappa_2 \lVert \mu\rVert^2_{L^2({\Omega})}  \\
 	- C_2  \, \lVert g^{\perp}\rVert_{ {\HH}}\lVert \varrho \rVert_{L^2({\Omega})} 
	 - C_2 \,  \lVert\EE\rVert_{L^2({\Omega})}\lVert \varrho\rVert_{L^2({\Omega})} 
 	- C_2 \, \lVert\EE\rVert^2_{L^2({\Omega})}
         - C_2\, \lVert g^{\perp}\rVert^2_{ {\HH}} \\
         - C_2 \, \left\lVert \sqrt{\iota(2-\iota)} \DDD_1^\perp \gamma_+g\right\rVert^2_{\partial {\HH}_+} 
         - (\theta_0)^2 \, C_2  \left\lVert  \sqrt{\iota(2-\iota)} (\widetilde {\gamma_+ g}) \right\rVert^2_{L^2(\partial {\Omega})}  .
\end{multline*}
\end{lem}

\begin{proof}
We define
$$
E_1 = \langle -\grad_x U[\mu], M_q[\LL g]\rangle_{L^2({\Omega})} \quad \text{ and } \quad E_2 = \langle -\grad_x U[\mu[\LL g]], M_q[ g]\rangle_{L^2({\Omega})}.
$$

Using the Cauchy-Schwarz inequality, \eqref{eq:H1Momentum_Hypo_decay}, and the fact that $\lvv M_q [g^\perp] \rvv_{L^2({\Omega})} \lesssim \lvv g^\perp\rvv_{\HH_0}$, we deduce that
\bean
\lvert E_2 \rvert &\lesssim& \lVert \grad^s_x U[\mu] \rVert_{L^2({\Omega})} \left( \lVert \EE \rVert_{L^2({\Omega})} + \lVert g^\perp \rVert_{\HH} \right) \\
&\lesssim &  \left(  \lVert \varrho \rVert_{L^2({\Omega})} +  \lVert \EE\rVert_{L^2({\Omega})}  + \lVert g^\perp\rVert_{ {\HH}} + \left\lVert \sqrt{\iota(2-\iota)} \DDD_1^\perp \gamma_+ g \right\rVert_{\partial {\HH}_+} \right.\\
&&  \qquad \qquad \qquad \qquad  \left. +  (\theta_0)^2 \, \left\lvv   \sqrt{\iota (2-\iota)}   \left( \widetilde{ \gamma_+ g} \right) \right\rvv_{L^2(\partial {\Omega})}   \right)   \left( \lVert \EE \rVert_{L^2({\Omega})} + \lVert g^\perp \rVert_{\HH} \right) .
\eean
On the other side, using the decomposition \eqref{eq:DecompositionMicroMacro_Hypo_decay}, we have that 
$$
M_q[\LL g ] =  M_q[-v\cdot \grad_x g] + M_q [ \CC g^\perp],
$$
thus $E_1 = E_{1,1} + E_{1,2}$ with 
$$
E_{1,1} :=  \left\la   (\grad^s_x U[\mu])_{ij} ,  \partial_{x_k} \int_{\R^3} q_{ij} (v) \, v_k  \, g \,\dv \right\ra   \quad \text{ and } \quad  E_{1, 2} :=  \left\la  - \grad^s_x U[\mu] ,  \int_{\R^3} q (v)  \CC g^\perp \,\dv \right\ra ,
$$
where we have used the convention of summation over repeated indices in the definition of $E_{1,1}$. 

Furthermore, using \eqref{eq:MicroCoercivity} we have that
$$
\int_{\R^3} q(v) \CC g^\perp\dv = \la g^\perp, \CC(\MM q)\ra_{\HH_0} \lesssim \lvv g^\perp \rvv_{\HH_0},
$$
thus
$$
\lvert E_{1,2} \rvert \lesssim   \lVert \grad^s_x U[\mu] \rVert_{L^2({\Omega})} \lVert g^\perp \rVert_{\HH} \lesssim   \lVert \mu \rVert_{L^2({\Omega})} \lVert g^\perp \rVert_{\HH}
$$
where we have used the regularity result given by Theorem~\ref{thm:LameRegularity} to obtain the second inequality.

\smallskip
To control $E_{1,1}$ we perform integration by parts so we obtain that $E_{1,1} = E_{1, 3} + E_{1, 4}$ where
$$
E_{1, 3} :=  \left \la -\partial_{x_j} (\grad^s_x U[\mu])_{ij},   \int_{\R^3} q_{ij} (v) \, v_k  \, g \,\dv \right\ra \quad \text{and}\quad   E_{1,4} :=  \int_{{\Sigma}} \grad^s_x U[\mu] :   q (v)\,  \gamma g \,  (n_x\cdot  v) . 
$$
By arguing exactly as in the proof of \cite[Lemma 3.6]{MR4581432} to control the interior term $E_{1,3}$, we have that 
$$
E_{1, 3}\geq  \lVert \mu \rVert^2_{L^2({\Omega})} - C \lVert f^\perp \rVert^2_{\HH},  
$$
for some constant $C>0$. 
Regarding the boundary term $E_{1,4}$, we use Lemma~\ref{lem:BoundaryTerms} and we have that
\bean
E_{1,4} &=&  \int_{{\Sigma_+}} \grad^s_x U[\mu] : q(v)  \, \iota(x) \DDD_1^\perp \gamma_+ g (n_x\cdot v)_+\dv \d\sigma_x\\
&& + \int_{{\Sigma_+}}  \grad^s_x U[\mu] :  \left[ q(v) - q(\VV_x v) \right] (1-\iota(x)) \DDD_1^\perp \gamma_+ g (n_x\cdot v)_+\dv \d\sigma_x\\
&& + \int_{{\Sigma_+}} \grad^s_x U[\mu] :  \left[ q(v) - q(\VV_x v) \right] \DDD_1 \gamma_+ g (n_x\cdot v)_+\dv \d\sigma_x \\
&& - \int_{\partial {\Omega}} \iota(x)  (\widetilde{\gamma_+ g}) (\grad^s_x U[\mu])_{ij} \left( \int_{\R^3}   q_{ij}(v)\,  \left[\MMM_\TTT(x,v) - \MMM_1(v)\right] (n_x\cdot v)_-\dv \right) \d\sigma_x \\
&=& E_{1,5} + E_{1,6} + E_{1, 7} + E_{1, 8} .
\eean
We note that, by arguing exactly as in the proof of \cite[Lemma 3.6]{MR4581432}, we have that 
$$
\lv E_{1,5} + E_{1,6} + E_{1, 7}\rv \lesssim  \lVert \mu \rVert_{L^2({\Omega})} \lVert \iota \DDD_1^\perp g  \rVert_{\partial {\HH}_+}.
$$
To control the remaining terms we use \eqref{eq:PertTempBdy}, and we observe that
$$
\left \lv \int_{\R^3} q_{ij}(v)  \left[\MMM_\TTT(x,v) - \MMM_1(v)\right]  \, (n_x \cdot v)_- \dv \right\rv \lesssim  \theta_0 \qquad \forall i,j= \llbracket 1,3 \rrbracket, 
$$
thus we further have, by using the Cauchy-Schwarz inequality, that
$$
\lv E_{1,8}  \rv \lesssim   \theta_0  \lvv\mu\rvv_{L^2_x(\partial{\Omega})}  \lvv   \iota (\widetilde{\gamma_+ g}) \rvv_{L^2(\partial {\Omega})} ,
$$
where we have used the regularity result from Theorem~\ref{thm:LameRegularity} to obtain the second inequality. 

We conclude by putting everything together and using the Young inequality together with the fact that $\iota \leq \sqrt{\iota/(2-\iota)}$. 
\end{proof}

\subsubsection{Mass functional}\label{ssec:MassFunctional_Hypo_decay}
We devote this subsection to construct a functional in order to control the mass component of the macroscopic part $\Pi g$. 

We define the operator 
\be\label{eq:Mass_Hypo_decay}
\varrho[h] := \int_{\R^3} h \, \dv , 
\ee
so that, in particular, $\varrho = \varrho[g]$. 

We also consider $u_{\rm N}[\varrho]$ the unique variational solution to the Poisson equation \eqref{eq:PoissonEquation} associated to $\xi = \varrho$ and complemented with the Neumann boundary conditions \ref{item:P2}, which is well-defined since $\lla \varrho \rra_{{\Omega}}=0$, and is given by Theorem~\ref{theo:PoissonRegularity}.

\begin{lem}\label{lem:ineqHypoMass_Hypo_decay}
There holds 
\be\label{eq:MassLLControl_Hypo_decay}
 \varrho[\LL g] = -   \Div_x \mu  ,
\ee
and 
\be\label{eq:H1mass_Hypo_decay}
 \lVert u_{\rm N}[\varrho[\LL g]]\rVert_{H^1({\Omega})} \lesssim   \lVert \mu\rVert_{L^2({\Omega})} .
 \ee 
\end{lem}

\begin{proof}
We note that, from the very definition of $\LL$ and the decomposition \eqref{eq:DecompositionMicroMacro_Hypo_decay}, it follows that $\LL g = - v\cdot \nabla_x f +  \CC g^\perp$, hence we have that
$$
\varrho[\LL g] 
= \varrho[-  v \cdot \nabla_x g]  + \varrho[ \CC g^\perp]
= -  \nabla_x \cdot \int_{\R^3} v g \, \dv  
$$
since $\varrho[\CC g^\perp] = \la g^\perp, \CC 1\ra_{L^2({\Omega})} =0 $, and this yields \eqref{eq:MassLLControl_Hypo_decay}.

We consider now $u := u_{\mathrm{N}}[\varrho[\LL g]]$ as the unique variational solution to the Poisson \eqref{eq:PoissonEquation} with the Neumann boundary condition \ref{item:P2} associated to $\xi = \varrho[\LL g]$, which is given by Theorem~\ref{theo:PoissonRegularity}. 

From the variational formulation \eqref{eq:PoissonVariationalSln} with $v=u_N$ we have that
\bean
\| \nabla_x u \|_{L^2_x({\Omega})}^2 
& =& -   \int_{{\Omega}} (\nabla_x \cdot \mu) u \, \dx \\
&=&   \int_{{\Omega}}  \mu \cdot \nabla_x u \, \dx 
-   \int_{\partial{\Omega}}  \mu \cdot n(x) \, u \, \d\sigma_{\! x}
=   \int_{{\Omega}}  \mu \cdot \nabla_x u \, \dx  
\eean
where we have used integration by parts and the fact that $\mu(x) \cdot n(x) = 0$ in the last line. 
We conclude \eqref{eq:H1mass_Hypo_decay} by using the Cauchy-Schwarz inequality, the Young inequality and the regularity results from Theorem~\ref{theo:PoissonRegularity}. 
\end{proof}

Using this, and by arguing in the spirit of  \cite[Lemma 3.8]{MR4581432}, we have the following result. 
\begin{lem}\label{lem:MassCoercivity_Hypo_decay}
There are constants $\kappa_3,C_3>0$ such that
    \begin{multline*}
          \langle -\grad_x u_{\rm N}[\varrho], \mu[\LL g]\rangle_{L^2({\Omega})} + \langle -\grad_x u_{\rm N} [\varrho[\LL g]], \mu\rangle_{L^2({\Omega})}\\
        \quad \geq \kappa_3  \lVert \varrho\rVert^2_{L^2({\Omega})} 
        - C_3 \,  \lVert \mu\rVert^2_{L^2({\Omega})}
        - C_3 \,  \lVert\EE\rVert^2_{L^2({\Omega})} 
        - C_3 \,  \lVert g^{\perp}\rVert^2_{ {\HH}}  \\
	- C_3 \,  \left\lVert \sqrt{\iota(2-\iota)} \DDD_1^\perp \gamma_+ g\right\rVert^2_{\partial{\HH}_+} 
	- (\theta_0)^2 \, C_3 \left\lVert  \sqrt{\iota(2-\iota)} (\widetilde{\gamma_+ g}) \right\rVert^2_{L^2(\partial{\Omega})}    .
    \end{multline*}   
\end{lem}

\begin{proof}
We define
$$
E_1 := \la -\grad_x u_{\rm N}[\varrho] , \mu[\LL g]\ra_{L^2({\Omega})} \quad \text{ and }\quad E_2 := \la -\grad_x u_{\rm N}[\varrho[\LL g]], \mu\ra_{L^2({\Omega})}.
$$

On the one hand, we immediately compute by using the Cauchy-Schwarz inequality
$$
\lvert E_2\rvert  \leq  \left\lvv  \grad_x u_{\rm N}[\varrho[\LL g]] \right\rvv_{L^2({\Omega})} \lVert \mu\rVert_{L^2({\Omega})}  \lesssim  \lVert \mu\rVert^2_{L^2({\Omega})},
$$
where we have used \eqref{eq:H1mass_Hypo_decay} to obtain the second inequality above. On the other hand, to control the remaining term $E_2$, we use integration by parts and there yields $E_1 = E_{1,1} + E_{1,2}$ with
$$
E_{1,1} := - \left\la \partial_{x_i}\partial_{x_j} u_{\rm N}[\varrho], \int_{\R^3} v_iv_j g \dv  \right\ra_{L^2({\Omega})} \quad \text{and}  \quad
E_{1,2} :=  \int_{{\Sigma}} (\grad_x u_{\rm N}[\varrho] \cdot v) \, \gamma g \, (n_x\cdot v) \dv\d\sigma_x,
$$
where we have used the convention of sum over repeated indices in the definition of $E_{1,1}$. 

By arguing now exactly as in the proof of \cite[Theorem 3.8]{MR4581432}, we have that there is a constant $C>0$ such that
$$
E_{1,1}\geq \frac 12  \lVert \varrho \rVert^2_{L^2({\Omega})} -  C \lVert \EE \rVert^2_{L^2({\Omega})} - C  \lVert g^\perp \rVert^2_{{\HH}}. 
$$
Furthermore, by using Lemma \ref{lem:BoundaryTerms} in the boundary term $E_{1,2}$, we have that 
\bean
E_{1,2} &=& \int_{{\Sigma_+}} \grad_x u_{\rm N}[\varrho] \cdot v \, \iota(x) \, \DDD_1^\perp \gamma_+ g (n_x\cdot v)_+\dv \d\sigma_x\\
&& +  \int_{{\Sigma_+}}  \grad_x u_{\rm N}[\varrho] \cdot  \left[ v - \VV_x v \right] (1-\iota(x)) \DDD_1^\perp \gamma_+ g (n_x\cdot v)_+\dv\d\sigma_x \\
&& + \int_{{\Sigma_+}} \grad_x u_{\rm N}[\varrho] \cdot  \left[ v - \VV_x v \right] \DDD_1 \gamma_+ g (n_x\cdot v)_+\dv\d\sigma_x \\
&& -  \int_{{\Sigma_-}}   \grad_x u_{\rm N}[\varrho] \cdot  v\,  \iota(x) \, \left[\MMM_\TTT (x,v) - \MMM_1(v) \right] (\widetilde{\gamma_+ g}) \, (n_x\cdot v)_-\dv \d\sigma_x\\
&=:& E_{1,3} + E_{1,4} + E_{1,5}+ E_{1,6} .
\eean
By repeating exactly the computations of \cite[Theorem 3.8]{MR4581432} we have that
$$
\lv E_{1,3} + E_{1,4} + E_{1,5}\rv \lesssim  \lvv \varrho \rvv_{L^2(  {\Omega})} \lvv \iota \DDD_1^\perp \gamma_+ g \rvv_{\partial {\HH}_+}.
$$
Moreover, using \eqref{eq:PertTempBdy}, we further have that
$$
\lv E_{1,6} \rv  \lesssim \theta_0 \, \lvv \varrho\rvv_{L^2({\Omega})} \lvv  \iota (\widetilde{\gamma_+ g}) \rvv_{L^2({\Omega})} .
$$
where we have used the regularity gain from Theorem~\ref{theo:PoissonRegularity}. We conclude by using Young's inequality and the fact that $\iota \leq \sqrt{\iota/(2-\iota)}$.
\end{proof}

\subsubsection{Proof of Theorem \ref{theo:Pert_hypo_Hypo_decay}}\label{ssec:conclusion_Hypo_decay}
We define a scalar product on ${\HH}$ 
\be\label{eq:hypocNorm_Hypo_decay}
\begin{aligned}
\lla h_1, h_2\rra &:= \langle h_1,h_2 \rangle_{ {\HH}}\\
&\qquad +\eta_1  \langle-\grad_x u[\EE[h_1]], M_p[h_2]\rangle_{L^2({\Omega})} + \eta_1  \langle-\grad_x u[\EE[h_2]], M_p[h_1]\rangle_{ L^2({\Omega})}\\
&\qquad +\eta_2  \langle-\grad^s_x U[\mu[h_1]], M_q[h_2]\rangle_{L^2({\Omega})} + \eta_2 \langle-\grad^s_x U[\mu[h_2]], M_q[h_1]\rangle_{ L^2({\Omega})}\\
&\qquad +\eta_3  \langle-\grad_x u_{\rm N}[\varrho[h_1]], \mu[h_2]\rangle_{L^2({\Omega})} + \eta_3  \langle-\grad_x u_{\rm N}[\varrho[h_2]], \mu[h_1]\rangle_{ L^2({\Omega})},
\end{aligned}
\ee
for some parameters $0 \ll \eta_3 \ll \eta_2 \ll \eta_1 \ll 1$ to be chosen later, and where we recall that the functionals $M_p$ and $M_q$ are defined respectively in \eqref{eq:def-Mp_Hypo_decay} and \eqref{eq:def-Mq_Hypo_decay}; $u[\EE[h_j]]$ is the solution of the Poisson equation~\eqref{eq:PoissonEquation} with $\xi = \EE[h_j]$ and boundary condition \ref{item:P1}; $U[\mu[h_j]]$ is the solution to the elliptic system \eqref{eq:LameSystem} with data $\Xi = \mu[h_j]$; $u_{N}[\varrho[h_j]]$ is the solution to the Poisson equation with homogeneous Neumann boundary condition~\ref{item:P2} and with data $\xi = \varrho[g]$; for any $j=1,2$. 

We define next the norm associated to the above scalar product as
\be\label{def:HypoNorm_Hypo_decay}
\lvvv g \rvvv := \displaystyle \sqrt{\lla g, g\rra},
\ee
and we observe that by using the Cauchy-Schwarz inequality, the Young inequality, and the regularizing estimates from Theorem \ref{theo:PoissonRegularity} and \ref{thm:LameRegularity} together with \eqref{eq:NormDecompositionMicroMacro_Hypo_decay} and the fact that 
\be\label{eq:ControlMpMq_Hypo_decay}
\lvv M_p[g] \rvv_{L^2({\Omega})} + \lvv M_q[g] \rvv_{L^2({\Omega})} \lesssim \lvv g\rvv_{{\HH}} 
\ee
yield that there is a constant $c>0$, such that
$$
  \lvv g\rvv_{\HH}^2 \left( 1- 2c\, (\eta_1 + \eta_2 + \eta_3) \right)   \leq  \lvvv g \rvvv^2 \leq   \lvv g\rvv_{\HH}^2 \left( 1 + 2c\, (\eta_1 + \eta_2 + \eta_3) \right) . 
$$
Then by choosing $\eta_1, \eta_2, \eta_3 \in (0, (12c)^{-1})$ we have the equivalency of norms
\begin{equation}\label{eq:NormEquivalence_Hypo_decay}
    \lVert g\rVert_{ {\HH}}\lesssim \lvvv g\rvvv \lesssim \lVert g\rVert_{ {\HH}}.
\end{equation}

Let now $g$ satisfy the assumptions of Theorem~\ref{theo:Pert_hypo_Hypo_decay}. Recalling that we have denoted $\varrho=\varrho[g]$, $\mu=\mu[g]$ and $\EE = \EE[g]$, noting that $\sqrt{\iota(2-\iota)} \ge \iota$, and gathering Lemmas~\ref{lem:MicroCoercivity_Hypo_decay},~\ref{lem:EnergyCoercivity_Hypo_decay},~\ref{lem:MomentumCoercivity_Hypo_decay} and~\ref{lem:MassCoercivity_Hypo_decay}, one has 
$$
\begin{aligned}
\la \! \la - \LL g , g \ra \! \ra 
&\ge  \kappa_0  \lVert g^\perp \rVert^2_{\HH} + \frac {1}4 \left \lVert \sqrt{\iota (2-\iota)}\DDD_1^\perp \gamma_+ g \right\rVert^2_{\partial\HH_+}  -  \theta_0\, C \left\lVert   (\iota)^{1/2}   \left( \widetilde {\gamma_+ g}\right) \right\rVert^2_{L^2(\partial\Omega)}  \\
&\quad 
+\eta_1 \, \Bigg( \kappa_1 \lVert\EE\rVert^2_{L^2({\Omega})} - C  \lVert \mu\rVert_{L^2({\Omega})} \lVert g^{\perp}\rVert_{ \HH} - C   \left\lVert \sqrt{\iota(2-\iota)} \DDD_1^\perp \gamma_+g\right\rVert^2_{\partial\HH_+}
\\&\qquad \qquad \qquad \qquad 
 - C  \lVert g^{\perp}\rVert^2_{ \HH}  - (\theta_0)^2  \, C   \left\lVert  \sqrt{\iota(2-\iota)}  \, \left( \widetilde{\gamma_+ g}\right) \right\rVert^2_{L^2(\partial {\Omega})}    \Bigg)\\
&\quad 
+\eta_2 \, \Bigg( \kappa_2 \lVert \mu\rVert^2_{L^2({\Omega})} 
 	- C  \, \lVert g^{\perp}\rVert_{ {\HH}}\lVert \varrho \rVert_{L^2({\Omega})} 
	 - C \,  \lVert\EE\rVert_{L^2({\Omega})}\lVert \varrho\rVert_{L^2({\Omega})}  - C \, \lVert\EE\rVert^2_{L^2({\Omega})}
\\&\qquad \qquad
         - C\, \lVert g^{\perp}\rVert^2_{ {\HH}} 
         - C \,\left\lVert \sqrt{\iota(2-\iota)} \DDD_1^\perp \gamma_+g\right\rVert^2_{\partial {\HH}_+} 
         -  (\theta_0)^2 \, C \left\lVert  \sqrt{\iota(2-\iota)} (\widetilde {\gamma_+ g}) \right\rVert^2_{L^2(\partial {\Omega})}  \Bigg) \\
&\quad 
+\eta_3 \,  \Bigg( \kappa_3 \lVert \varrho\rVert^2_{L^2({\Omega})} 
        - C \,  \lVert \mu\rVert^2_{L^2({\Omega})}
        - C \, \lVert\EE\rVert^2_{L^2({\Omega})} 
        - C \,  \lVert g^{\perp}\rVert^2_{ {\HH}}  \\
&\qquad \qquad
- C \,  \left\lVert \sqrt{\iota(2-\iota)} \DDD_1^\perp \gamma_+g\right\rVert^2_{\partial{\HH}_+} 
	- (\theta_0)^2\, C \left\lVert  \sqrt{\iota(2-\iota)} (\widetilde{\gamma_+ g}) \right\rVert^2_{L^2(\partial{\Omega})}   \Bigg),
\end{aligned}
$$
for some constant $C>0$. 
Following the same ideas as in the proof of \cite[Theorem 1.1]{MR4581432} we obtain
$$
\begin{aligned}
\la \! \la - \LL g , g \ra \! \ra 
&\ge  \left(\frac{\kappa_0}{2} - \eta_1 C - \eta_2 C - \eta_3 C\right) \| g^\perp \|_{{\HH}}^2  +\left(\frac{\eta_1 \kappa_1}{2} - \eta_2 C - \eta_3 C \right)\| \EE \|_{L^2_x({\Omega})}^2  \\
&\quad 
+ \left( \eta_2 \kappa_2 - \eta_1^2 C - \eta_3 C \right)\| \mu \|_{L^2_x({\Omega})}^2  
+ \left( \eta_3 \kappa_3  - \eta_2^2 C - \frac{\eta_2^2}{\eta_1}C\right)\| \varrho \|_{L^2_x({\Omega})}^2 
\\
&\quad
+ \left(\frac12 - \eta_1 C - \eta_2 C - \eta_3 C  \right) \left\lvv \sqrt{\iota(2-\iota)} \DDD_1^\perp \gamma_+ g \right\rvv_{\partial {\HH}_+}^2  \\
&\quad 
 -C \left\lVert  \theta_0 \,  (\iota)^{1/2}  \left( \widetilde {\gamma_+ g}\right) \right\rVert^2_{L^2(\partial{\Omega})}  -C(\eta_1 + \eta_2 + \eta_3) \left\lVert  \theta_0 \,  \sqrt{\iota(2-\iota)}  \left( \widetilde {\gamma_+ g}\right) \right\rVert^2_{L^2(\partial{\Omega})}.
\end{aligned}
$$
We now use that $2-\iota \leq 2$, and we choose $\eta_1 := \eta$, $\eta_2 := \eta^{\frac{3}{2}}$, $\eta_3 := \eta^{\frac{7}{4}}$, so we deduce that
$$
\begin{aligned}
\la \! \la - \LL g , g \ra \! \ra 
&\ge \left(\frac{\kappa_0}{2} - \eta C \right) \| g^\perp \|_{{\HH}}^2  
+ \left(\frac12 - \eta C   \right) \left\lvv  \sqrt{\iota(2-\iota)} \DDD_1^\perp \gamma_+ g \right\rvv_{\partial {\HH}_+}^2 +\eta \left( \frac{\kappa_1}{2} - \eta^{\frac{1}{2}} C \right)\| \EE \|_{L^2_x({\Omega})}^2 \\
&\quad    
+ \eta^{\frac{3}{2}} \left(  \kappa_2 - \eta^{\frac{1}{4}} C  \right)\| \mu \|_{L^2_x({\Omega})}^2  
+ \eta^{\frac{7}{4}} \left(  \kappa_3  - \eta^{\frac{1}{4}} C\right)\| \varrho \|_{L^2_x({\Omega})}^2  
- 7 C \left\lVert  \theta_0 \,  (\iota)^{1/2}  \left( \widetilde {\gamma_+ g}\right) \right\rVert^2_{L^2(\partial{\Omega})} 
\end{aligned}
$$
Choosing then $0 < \eta \ll1$ small enough, we further get that
$$
\begin{aligned}
\la \! \la - \LL g , g \ra \! \ra 
&\ge \kappa \left( \| g^\perp \|_{{\HH}}^2 
+\| \varrho \|_{L^2_x({\Omega})}^2 
+\| \mu \|_{L^2_x({\Omega})}^2 
+\| \EE \|_{L^2_x({\Omega})}^2  \right)   \\
&\qquad \qquad  \qquad 
-\kappa' \left\lVert  \theta_0 \,  (\iota)^{1/2}  \left( \widetilde {\gamma_+ g}\right) \right\rVert^2_{L^2(\partial{\Omega})} + \kappa'' \left\lvv \sqrt{\iota(2-\iota)} \DDD_1^\perp \gamma_+ g \right\rvv_{\partial {\HH}_+}^2, 
\end{aligned}
$$
for some constants $\kappa,\kappa', \kappa'' >0$.
We conclude the proof of Theorem~\ref{theo:Pert_hypo_Hypo_decay} by using \eqref{eq:NormDecompositionMicroMacro_Hypo_decay}
and the norm equivalency \eqref{eq:NormEquivalence_Hypo_decay}. 
\qed

\subsection{Interior a priori estimates}\label{ssec:Hypo_Apriori_Hypo_decay}

This subsection is devoted to the obtention of a priori estimates of the solution of Equation~\eqref{eq:Hypo_decay} 

\begin{prop}\label{prop:L2_Apriori_Hypo_decay}
Assume that either Assumption \ref{item:H1} or Assumption \ref{item:H2} holds
There is $\kappa \in \R$ such that for every $g$ solution of Equation~\eqref{eq:Hypo_decay} there holds
\be\label{eq:A_priori_bound_HH_Hypo_decay}
\lvv g_t \rvv_{\HH} \lesssim e^{\kappa \,  t} \lvv g_0\rvv_{\HH},
\ee
for every $t\geq 0$.
\end{prop}

\begin{rem}
The proof follows the same lines as that of \cite[Proposition 6.4]{BE_iso}, and we adapt it in order to consider non-isothermal diffusive boundary condition. 
\end{rem}

\begin{proof}[Proof of Proposition~\ref{prop:L2_Apriori_Hypo_decay}]
We first observe that \cite[Theorem 7.2.4]{MR1307620} implies that
\be\label{eq:Cercignani_CK_Hypo_decay}
\lVert K h \rVert_{{\HH}} \leq C_K \lVert  h \rVert_{{\HH}} ,
\ee
for some constant $C_K>0$.

We divide then the proof into two different cases: first we obtain \eqref{eq:A_priori_bound_HH_Hypo_decay} for smooth domains, and afterwards we repeat and adapt those computations for the cylindrical setting.

\medskip\noindent
\emph{Case 1. (Smooth domains---Assumption~\ref{item:H1})}
We define the cutoff function $\chi_R(v) := \chi(|v|/R)$ for every $R>0$, with $\chi \in C^2(\R_+)$, $\mathbf{1}_{[0,1]} \le \chi \le \mathbf{1}_{[0,2]}$, and $\chi^c_R : 1 - \chi_R$, and inspired by the weight functions introduced in \cite{CMKFP24}, and subsequently used in \cite{CM_Landau_domain, CGMM24}, we define the modified weight function
\be\label{def:muA_Hypo_decay}
 \mu_A^2(x,v) := \MM_{\TTT}^{-1} \chi_A + \MM^{-1} (1-\chi_A)
 \quad \text{ and } \quad   \mu_0^2 (x,v) :=   \left( 1 + \frac{ 1  }{ 2}  \frac{ n_x \cdot   v}{\la v \ra^2} \right) \mu_A^2,
\ee
where with $A > 1$ to be chosen later
. Moreover, it is worth emphasizing that
\be\label{eq:mu_mu0_Hypo_decay}
 c_A^{-1} \MM^{-1/2} \le \frac12\mu_A \le  \mu_0  \le \frac32 \mu_A \le c_A \MM^{-1/2}, 
\ee
for some $c_A  \in (0,\infty)$. 
Arguing then at a formal level, we have that if $g$ is a solution of Equation~\eqref{eq:Hypo_decay} it follows that
\be\label{eq:Well_posedness_HH_1_Hypo_decay}
\begin{aligned}
\frac 12 \frac d{dt} \int_{\OO} g_t^2\mu_0^2  &= \int_{\OO} g_t \left( - v\cdot \grad_x g_t + K g_t -\nu g_t  \right) \mu_0^2 \\
&=  - \frac {1}2 \int_{\Sigma} \gamma g_t^2 \,  (n_x\cdot v) \mu_0^2 + \frac {1}2 \int_{\OO} g_t^2 (v\cdot \grad_x \mu_0^2)  
- \int_{\OO} \nu g_t^2 \mu_0^2 + \int_{\OO} g_t (K g_t) \mu_0^2,
\end{aligned}
\ee
where we have used integration by parts, and we define 
\bean
I_1 &:=& - \frac {1}2 \int_{\Sigma} \gamma g_t^2 \,  (n_x\cdot v) \mu_0^2,  \\
I_2 &:=& \frac {1}2 \int_{\OO} g_t^2 (v\cdot \grad_x \mu_0^2)  - \int_{\OO} \nu g_t^2 \mu_0^2 +{1} \int_{\OO} g_t (Kg_t) \mu_0^2.
\eean
We divide now the proof for the Case 1 into three steps.

\medskip\noindent
\emph{Step 1. (Control of the interior terms)} We first observe that
$$
\lvv \mu_0^2 \MM \rvv_{L^\infty(\OO)} \lesssim 1,
$$ 
and using this, \eqref{eq:Cercignani_CK_Hypo_decay}, and the Cauchy-Schwarz inequality we have that 
$$
\beal
-\int_{\OO} \nu g_t^2 \mu_0^2 + \int_{\OO} g_t (Kg_t) \mu_0^2 &\leq  -\nu_0 \lvv g_t\rvv_{L^2_{\mu_0}(\OO)}^2 +C_K  \lvv g_t\rvv_{L^2_{\mu_0}(\OO)} \lvv g_t\rvv_{\HH} \,  \lvv \mu_0\, \MM^{1/2}\rvv_{L^\infty(\OO)} \\
&\lesssim \lvv g_t\rvv_{L^2_{\mu_0}(\OO)}^2 ,
\eeal
$$
where we have used \eqref{eq:mu_mu0_Hypo_decay} to obtain the last line. We then compute
$$
\frac 12 \int_{\OO} g_t^2 (v\cdot \grad_x \mu_0^2)  = \frac 12 \left \lvv \left( v\cdot \grad_x \mu_0^2 \right) \mu_0^{-2} \right\rvv_{L^\infty( \Omega)} \,   \lvv g_t\rvv_{L^2_{\mu_0}(\OO)}^2 ,
$$
and we note that
$$
\left( v\cdot \grad_x \mu_0^2 \right) \mu_0^{-2}  = \frac 12 \left( {v\cdot \grad_x (n_x \cdot v) \over \la v\ra^2} \right) \mu_A^2 \mu_0^{-2}  + \left( 1 + \frac{ 1  }{ 2}  \frac{ n_x \cdot   v}{\la v \ra^2} \right) \chi_A (v) (v\cdot \grad_x \MM_{\TTT} )\MM_{\TTT} ^{-2}  \mu_0^{-2} .
$$
This last expression, together with \eqref{eq:mu_mu0_Hypo_decay}, the compact support of the cutoff function $\chi_A$, and the regularity assumptions on the normal vector $n_x$, imply that 
$$
\frac 12 \int_{\OO} g_t^2 (v\cdot \grad_x \mu_0^2)  \lesssim  \lvv g_t\rvv_{L^2_{\mu_0}(\OO)}^2 .
$$
Altogether, we conclude this first step by gathering the above estimates and noting that we have obtained 
$$
I_2 \lesssim  \lvv g_t\rvv_{L^2_{\mu_0}(\OO)}^2,
$$
for every $A\geq 1$.

\medskip\noindent
\emph{Step 2. (Control of the boundary terms)} We observe that 
$$
I_1 = - \frac {1}2 \int_{\Sigma} \gamma g_t^2 \,  (n_x\cdot v) \mu_A^2  - \frac {1}2 \int_{\Sigma} \gamma g_t^2 \,   {(n_x \cdot v)^2\over \la v\ra^{2} }  \mu_A^2 =: I_{1,1} + I_{1,2},
$$
and we estimate each of the above integrals separately. 
On the one hand, using the boundary conditions of Equation~\eqref{eq:LPBE}, we have that
$$\beal
I_{1,1} &=  - \frac {1}2 \int_{\Sigma_+} \gamma_+ g_t^2 \,  \mu_A^2\, (n_x\cdot v)_+  +  \frac {1}2 \int_{\Sigma_-}  (\RRR_\TTT\gamma_+ g_t ) ^2 \, \mu_A^2 \,  (n_x\cdot v)_- \\
 &\leq - \frac {1}2 \int_{\Sigma_+} \gamma_+ g_t^2 \, \mu_A^2 \, (n_x\cdot v)_+  + \frac {1}2  \int_{\Sigma_-}  (1-\iota) (\SSS\gamma_+ g_t) ^2 \,  \mu_A^2\, (n_x\cdot v)_-  \\
 &\quad \qquad + \frac {1}2   \int_{\Sigma_-}  \iota \left(\MMM_{\TTT} \widetilde{\gamma_+ g_t} \right) ^2 \,  \mu_A^2\, (n_x\cdot v)_-  \\
 &\leq  -\frac {1}2 \int_{\Sigma_+} \iota (\gamma_+ g_t)^2 \,\mu_A^2\,   (n_x\cdot v)_+ + \frac {1}2 \int_{\partial \Omega} \iota \left( \widetilde{\gamma_+ g_t}\right)^2 \int_{\R^3} \MMM_{\TTT}^2 \mu_A^2 (n_x \cdot v)_+ ,
\eeal$$
where we have used successively used a convexity inequality with the function $x\mapsto \lv x\rv^2$ and the change of variables $v\mapsto \VV_x v$ together with the fact that $ \lv \VV_x v \rv = \lv v\rv $ and $|n(x) \cdot \VV_x v| = |n(x) \cdot v|$. 

Applying now the Cauchy-Schwarz inequality we have that
$$\beal
\left( \widetilde{\gamma_+ g}\right)^2 = \left( \int_{\R^3} \gamma_+ g \, (n_x \cdot v)_+ \right) ^2 \leq \left( \int_{\R^3} (\gamma_+ g)^2 \mu_A^2 (n_x \cdot v)_+  \right)  \left( \int_{\R^3} \mu_A^{-2} (n_x\cdot v)_+ \right).
\eeal $$
Using this and the Young inequality in the above expression of $I_{1,1}$ we deduce that 
$$
I_{1,1} \leq -\frac {1}2 \int_{\partial \Omega} \iota \left( \widetilde{\gamma_+ g}\right)^2  \II_{A,1}(x)  ,
$$
where we have defined
\be\label{def:IIA_1_Hypo_decay}
\II_{A, 1} (x) := \left( \int_{\R^3} \mu_A^{-2} (n_x\cdot v)_+ \right)^{-1} -  \int_{\R^3} \MMM_\TTT^2 \mu_A^2 (n_x \cdot v)_+,
\ee
and we note that as $A\to +\infty$ there holds
\be\label{eq:LimIIA1_Hypo_decay}
\II_{1, A}(x) \to \frac 12 \left[ \left( \widetilde{\MMM_\TTT} \right)^{-1} + \widetilde{\MMM_\TTT} \right]= \frac 12  \left(1-1\right) =0,
\ee
uniformly in $x$.
On the other hand, we have that 
$$\beal
I_{1,2}& = - \frac {1}2 \int_{\Sigma_+} (\gamma_+ g_t)^2 \,   {(n_x \cdot v)_+^2\over \la v\ra^{2} }  \mu_A^2 - \frac {1}2 \int_{\Sigma_-} (\gamma_- g_t)^2 \,   {(n_x \cdot v)_-^2\over \la v\ra^{2} }  \mu_A^2  \\
& \leq - \frac {1}2 \int_{\Sigma_+} (\gamma_+ g_t)^2 \,   {(n_x \cdot v)_+^2\over \la v\ra^{2} }  \mu_A^2 \leq  -\frac {1}2 \int_{\partial \Omega} \iota \left( \widetilde{\gamma_+ g} \right)^2 \II_{A, 2} (x),
\eeal$$
where we have used the Cauchy-Schwarz inequality and the fact that $\iota (x) \in [0,1]$ for every $x\in \partial \Omega$ to obtain the last inequality. We have also set
\be\label{def:IIA2_Hypo_decay}
\II_{A, 2}(x) := \left( \int_{\R^3} \la v\ra^2 \mu_A^{-2}\right)^{-1} \underset{A\to +\infty}{\longrightarrow} \int_{\R^3} \la v\ra^2 \MMM_\TTT   \in (0, +\infty),
\ee
uniformly for $x\in \partial \Omega$, and we note that the bounds on the last limit are due to our condition \eqref{eq:Condition_T_theta} on $\TTT$.

Gathering the above estimates, and choosing $A > 1$ large enough such that $\II_{A,1}(x) + \II_{A, 2} (x) \geq 0$ uniformly for $x\in \partial \Omega$, we deduce that
$$
I_1 \leq  -\frac {1}2 \int_{\partial \Omega} \iota \left( \widetilde{\gamma_+ g} \right)^2 \left[ \II_{A,1}(x) + \II_{A, 2} (x)  \right]  \leq 0 .
$$

\medskip\noindent 
\emph{Step 3. (Conclussion)} By putting together the estimates for $I_1$ and $I_2$ obtained during Steps 1 and 2, into \eqref{eq:Well_posedness_HH_1_Hypo_decay}, we have that
$$
\frac 12 \frac d{dt} \int_{\OO} g_t^2\mu_0^2  \lesssim \lvv g_t\rvv_{L^2_{\mu_0}(\OO)}^2 .
$$
We conclude the proof for Case 1 (under Assumption~\ref{item:RH1}) by using the Grönwall lemma and \eqref{eq:mu_mu0_Hypo_decay}.

\medskip\noindent
\emph{Case 2. (Cylindrical domains---Assumption~\ref{item:H2})}
We write any element $x\in \R^3$ by its components as $x = (x_1, x_2, x_3)$, and we consider the vector field $n_1 : \R^3\to \R^3$ defined by
\be\label{def:n1_Hypo_decay}
n_1(x_1, x_2, x_3) := {(x_1, 0, 0) \over L },
\ee
We note that, by its very definition, $n_1$ is a smooth vector field and that $n_1(x) = n(x)$ for every $x\in \Lambda_1\cup \Lambda_2$. Moreover, there also holds $\lv n_1 (x) \rv \leq 1$ for every $x\in \bar\Omega$, and
\be\label{eq:N1normalOrtogonality_Hypo_decay}
n_x \cdot n_1(x) = 0 \qquad \forall x\in \Lambda_3.
\ee
For every $A> 1$, we define then the modified weight function
\be\label{def:mu1_Hypo_decay}
\mu_1^2 (x,v) :=   \left( 1 + \frac{ 1  }{ 2}  \frac{ n_1(x) \cdot   v}{\la v \ra^2} \right) \mu_A^2,
\ee
where we recall that $\mu_A$ has been defined in \eqref{def:muA_Hypo_decay}. Moreover, we observe that there holds
\be\label{eq:mu_mu1_Hypo_decay}
 c_A^{-1} \MM^{-1/2} \le \frac12\mu_A \le  \mu_1  \le \frac32 \mu_A \le c_A \MM^{-1/2},
\ee
for every $A> 1$.
We then have that 
\be\label{eq:Well_posedness_HH_2_Hypo_decay}
\begin{aligned}
\frac 12 \frac d{dt} \int_{\OO} g_t^2\mu_1^2  &= \int_{\OO} g_t \left( - v\cdot \grad_x g_t + Kg_t - \nu g_t  \right) \mu_1^2 \\
&=  - \frac {1}2 \int_{\Sigma} \gamma g_t^2 \,  (n_x\cdot v) \mu_1^2 + \frac {1}2 \int_{\OO} g_t^2 (v\cdot \grad_x \mu_1^2) 
  -\int_{\OO} \nu g_t^2 \mu_1^2 + \int_{\OO} g_t (Kg_t) \mu_1^2  ,
\end{aligned}
\ee
where we have used integration by parts, and we have defined 
\bean
I_1 &:=& - \frac {1}2 \int_{\Sigma} \gamma g_t^2 \,  (n_x\cdot v) \mu_1^2,  \\
I_2 &:=& \frac {1}2 \int_{\OO} g_t^2 (v\cdot \grad_x \mu_1^2)  -  \int_{\OO} \nu g_t^2 \mu_1^2 +  \int_{\OO} g_t (Kg_t) \mu_1^2,
\eean
and we compute each of the above terms separately. 
On the one hand, by repeating exactly the same arguments as in the Step 1 of the Case 1, we deduce that 
$$
I_2 \lesssim \lvv g_t\rvv_{L^2_{\mu_1}(\OO)}^2,
$$
for every $A> 1$.
On the other, to control the boundary term we have that 
$$
I_1 = - \frac {1}2 \int_{\Sigma} \gamma g_t^2 \,  (n_x\cdot v) \mu_A^2  - \frac {1}2 \int_{\Sigma} \gamma g_t^2 \,   {(n_1(x) \cdot v) (n_x \cdot v)\over \la v\ra^{2} }  \mu_A^2 =: I_{1,1} + I_{1,2},
$$
and we estimate each of the above integrals separately. 
First, using the boundary conditions of Equation~\eqref{eq:Hypo_decay} with the choice of $\iota $ given by Assumption~\ref{item:H2}, we have that
$$\beal
I_{1,1} &=  - \frac {1}2 \int_{\Sigma_+} \gamma_+ g_t^2 \,  \mu_A^2\, (n_x\cdot v)_+  +  \frac {1}2 \int_{\Sigma_-}  (\RRR_\TTT\gamma_+ g_t) ^2 \, \mu_A^2 \,  (n_x\cdot v)_- \\
 &= - \frac {1}2 \int_{\Lambda_3}\int_{\R^3} \gamma_+ g_t^2 \, \mu_A^2\,  (n_x\cdot v)_+   + \frac {1}2  \int_{\Lambda_3}\int_{\R^3}   (\SSS\gamma_+ g_t) ^2 \,  \mu_A^2\, (n_x\cdot v)_-  \\
 &\quad \qquad  - \frac 12 \int_{\Lambda_1\cup\Lambda_2}\int_{\R^3} \gamma_+ g_t^2 \,  \mu_A^2 \, (n_x\cdot v)_+   + \frac {1}2   \int_{\Lambda_1\cup\Lambda_2}   \left(\MMM_\TTT \widetilde{\gamma_+ g_t} \right) ^2 \,  \mu_A^2\, (n_x\cdot v)_-  \\
 &=   - \frac {1}2 \int_{\Lambda_1\cup\Lambda_2} \int_{\R^3} (\gamma_+ g_t)^2 \, \mu_A^2\,  (n_x\cdot v)_+  + \frac {1}2 \int_{\Lambda_1\cup\Lambda_2} \left( \widetilde{\gamma_+ g}\right)^2 \int_{\R^3} \MMM_\TTT^2 \mu_A^2 (n_x \cdot v)_+ 
\eeal$$
where we have used successively used the change of variables $v\mapsto \VV_x v$ together with the fact that $ \lv \VV_x v \rv = \lv v\rv $ and $|n(x) \cdot \VV_x v| = |n(x) \cdot v|$.

Using the Cauchy-Schwarz inequality and the Young inequality exactly as in the control of $I_{1,1}$ in the Step 2 of the Case 1 we deduce that there is a constant $C>0$ such that 
$$
I_{1,1} \leq -\frac {1}2 \int_{\Lambda_1\cup \Lambda_2} \left( \widetilde{\gamma_+ g}\right)^2  \II_{A,1}(x) ,
$$
where we recall that $\II_{A,1}$ is defined in \eqref{def:IIA_1_Hypo_decay} and it satisfies $\II_{A,1} \to 0$ as $A\to \infty$, uniformly in $x$ (as computed in \eqref{eq:LimIIA1_Hypo_decay}).

Secondly, to control $I_{1,2}$ we compute
$$
I_{1,2} = - \frac {1}2 \int_{\Lambda_1 \cup \Lambda_2} \int_{\R^3} \gamma g_t^2 \,   {(n_x \cdot v)^2\over \la v\ra^{2} }  \mu_A^2  - \frac {1}2 \int_{\Lambda_3 } \int_{\R^3} \gamma g_t^2 \,   {(n_1(x) \cdot v) (n_x \cdot v)\over \la v\ra^{2} }  \mu_A^2.
$$
We observe now that, using the boundary condition of Equation~\eqref{eq:Hypo_decay} with the choice of $\iota$ given by Assumption~\ref{item:H2}, we have that
$$\beal
- \frac {1}2 \int_{\Lambda_3 } \int_{\R^3} (\gamma g_t)^2 \,   {(n_1(x) \cdot v) (n_x \cdot v)\over \la v\ra^{2} }  \mu_A^2 &= - \frac {1}2 \int_{\Lambda_3 } \int_{\R^3} (\gamma_+ g_t)^2 \,   {(n_1(x) \cdot v) (n_x \cdot v)_+\over \la v\ra^{2} }  \mu_A^2  \\
&\quad +  \frac {1}2 \int_{\Lambda_3 } \int_{\R^3} (\SSS \gamma_+ g_t)^2 \,   {(n_1(x) \cdot v) (n_x \cdot v)_-\over \la v\ra^{2} }  \mu_A^2.
\eeal$$
Using then the change of variables $v\mapsto \VV_x v$ together with the fact that $ \lv \VV_x v \rv = \lv v\rv $, $|n(x) \cdot \VV_x v| = |n(x) \cdot v|$ and $n_1(x) \cdot \VV_x v = n_1(x) \cdot v$ due to \eqref{eq:N1normalOrtogonality_Hypo_decay}, we deduce that 
$$
- \frac {1}2 \int_{\Lambda_3 } \int_{\R^3} (\gamma g_t)^2 \,   {(n_1(x) \cdot v) (n_x \cdot v)\over \la v\ra^{2} }  \mu_A^2= 0. 
$$
Thus, using the Cauchy-Schwarz inequality as in the Step 2 of Case 1, we immediately deduce that 
$$
I_{1,2} \leq - \frac {1}2 \int_{\Lambda_1 \cup \Lambda_2} \int_{\R^3} (\gamma_+ g_t)^2 \,   {(n_x \cdot v)_+^2\over \la v\ra^{2} }  \mu_A^2 \leq  -\frac {1}2 \int_{\Lambda_1\cup \Lambda_2} \left( \widetilde{\gamma_+ g} \right)^2 \II_{A, 2} (x),
$$
where $\II_{A,2}$ is defined in \eqref{def:IIA2_Hypo_decay}. 

Gathering the above estimates, and choosing $A> 1$ large enough such that $\II_{A,1}(x) + \II_{A, 2}  \geq 0$ uniformly for $x\in \partial \Omega$, we deduce that
$$
I_1 \leq  -\frac {1}2 \int_{\partial \Omega} \left( \widetilde{\gamma_+ g} \right)^2 \left[ \II_{A,1}(x) + \II_{A, 2} (x)\right] \leq 0,
$$
We conclude the proof for Case 2 (under Assumption~\ref{item:H2}) by putting the estimates for $I_1$ and $I_2$ into \eqref{eq:Well_posedness_HH_2_Hypo_decay}, using the Grönwall lemma and \eqref{eq:mu_mu0_Hypo_decay} as in the conclusion for the Case~1. 
\end{proof}

\subsection{Boundary a priori estimates}\label{ssec:Hypo_AprioriBdy_Hypo_decay}
In this subsection we provide a priori estimates for the trace of the solutions of Equation~\eqref{eq:Hypo_decay}. 
We define the set of singular points in the cylinder as
\be\label{def:SSSS_Hypo_decay}
\SSSS:=(\overline{\Lambda_1} \cap \overline{\Lambda_3}) \cup (\overline{\Lambda_2} \cap \overline{\Lambda_3}) = \left\{(x_1, x_2, x_3)\in \R^3,\, x_1 =\pm L, \, (x_2)^2+(x_3)^2 =  \RRRR^2 \right\}.
\ee

We define $\delta_\SSSS$ as a smooth $C^2$ function coinciding, in a small neighborhood of $\SSSS$, with the distance function to this compact set. We note that regularity of $\delta_\SSSS$ is justified since $\SSSS$ is the disjoint union of compact sub-manifolds of $\R^3$ making the distance function to such set a smooth function in a small neighborhood around it, see, for instance, \cite{MR749908, MR737190, MR614221}. 

We then define the distance--like function
\be\label{def:zetaSSSS_Hypo_decay}
\zeta_\SSSS(x) := \left\{\begin{array}{cc}
1 & \text{ if Assumption~\ref{item:H1} holds}\\
\\
\displaystyle {(\delta_\SSSS (x))^2 \over 1+ (\delta_\SSSS (x)) ^2} & \text{ if Assumption~\ref{item:H2} holds},
\end{array}\right.
\ee
and we note that due to the discussion above $\zeta_\SSSS \in C^1 (\R^3)$. Following now exactly the same ideas as in the proof of \cite[Proposition 6.5]{BE_iso}, we have then the following a priori estimate for the trace of the solutions of Equation~\eqref{eq:LPBE}.

 \begin{prop}\label{prop:L2_AprioriBdy_Hypo_decay}
Assume that either Assumption \ref{item:H1} or Assumption \ref{item:H2} holds
There is $\kappa \in \R$ such that for every $g$ solution of Equation~\eqref{eq:Hypo_decay} there holds
\be\label{eq:A_priori_bound_Bdy_Hypo_decay}
\int_0^t \int_{\Sigma}  (\gamma g_s)^2 \, (n_x\cdot v)^2 \, \zeta_{\SSSS}(x) \, \la v\ra^{-2} \, \MM^{-1} \,  \dv\d\sigma_x\ds  \lesssim e^{\kappa \,  t} \lvv g_0\rvv_{\HH}^2 ,
\ee
for every $t\geq 0$.
\end{prop}

\begin{rem}\label{rem:L2_AprioriBdy_Hypo_decay}
As observed on \cite[Remark 6.6]{BE_iso}, on cylindrical domains (under Assumption~\ref{item:H2}) the estimates for the boundary are more degenerate than for smooth domains. The additional term $\zeta_{\SSSS}$, making at all possible the control of the trace, serves to---in a sense---smooth out the normal vector when approaching the singular set $\SSSS$.  
\end{rem}

\begin{proof}[Proof of Proposition~\ref{prop:L2_AprioriBdy_Hypo_decay}]
As noted above, the proof is nothing but the exact repetition of the ideas of \cite[Proposition 6.5]{BE_iso}, using Proposition~\ref{prop:L2_Apriori_Hypo_decay}, thus we skip it. 
\end{proof}

Furthermore, we now prove an additional a priori estimate for the trace in cylindrical domains.

 \begin{prop}\label{prop:L2_AprioriBdyLambda12_Hypo_decay}
Assume that Assumption \ref{item:H2} holds.
There is $\kappa \in \R$ such that for every $g$ solution of Equation~\eqref{eq:Hypo_decay} there holds
\be\label{eq:A_priori_bound_BdyLambda12_Hypo_decay}
\int_0^t \int_{\Lambda_1\cup \Lambda_2} \int_{\R^3} (\gamma g_s)^2 \, (n_x\cdot v)^2 \,  \la v\ra^{-2} \, \MM^{-1} \,  \dv\d\sigma_x\ds  \lesssim  e^{\kappa  t} \lvv g_0\rvv_{\HH}^2 , \qquad \forall t\geq 0.
\ee
\end{prop}

\begin{proof}
We recall the smooth vector field $n_1:\R^3 \to \R^3$ defined by \eqref{def:n1_Hypo_decay} and satisfying that $n_1(x) = n(x)$ for every $x\in \Lambda_1\cup \Lambda_2$, $\lv n_1 (x) \rv \leq 1$ for every $x\in \bar\Omega$, and the orthogonality property \eqref{eq:N1normalOrtogonality_Hypo_decay}.
We then compute 
\be\label{eq:BdyApriori2}
\begin{aligned}
\frac 12 \frac \d{\dt} \int_{\OO} (g_t)^2  \MM^{-1} \la v\ra^{-2} (n_1(x) \cdot v)  &= \int_{\OO} g_t \left( -v\cdot \grad_x g_t + \CC g_t  \right) \MM^{-1} \la v\ra^{-2} (n_1 (x) \cdot v)\\
& \leq - \frac {1}2 \int_{\OO} v\cdot \grad_x (g_t^2) \MM^{-1} (n_1(x)  \cdot v) \la v\ra^{-2}  + C \,  \lvv g_t\rvv_{\HH}^2,
\end{aligned}
\ee
for some constant $C>0$, and we note that we have used integration by parts, the fact  that $\lv \la v\ra^{-2} (n_1(x)\cdot v) \rv\leq 1$, the Cauchy-Schwarz inequality, and \eqref{eq:Cercignani_CK_Hypo_decay} to obtain the last line. 
Using again integration by parts we now have that 
$$
\begin{aligned}
& -\int_{\OO} v\cdot \grad_x (g_t^2) \MM^{-1} (n_1(x) \cdot v) \la v\ra^{-2} \dv\dx \\
=& - \int_{\Sigma} g_t^2 \MM^{-1} (n_1(x) \cdot v) (n_x\cdot v)  \la v\ra^{-2} \dv\d\sigma_x + \int_{\OO} g_t^2 \MM^{-1} \, (v\cdot \grad_x (n_1(x)  \cdot v))\,  \la v\ra^{-2}\dv\dx  \\
\leq&  - \int_{\Sigma} g_t^2 \MM^{-1} (n_1(x) \cdot v) (n_x \cdot v) \la v\ra^{-2}\dv\d\sigma_x  + C \lvv g_t \rvv_{\HH}^2,
\end{aligned}
$$
for some constant $C>0$, and we note that we have used the regularity of the vector field $n_1$ to deduce the last line

We denote the above boundary integral as $I_0$ and we observe that
$$
I_0 = - \int_{\Lambda_1 \cup \Lambda_2} \int_{\R^3} \gamma g_t^2 \,   {(n_x \cdot v)^2\over \la v\ra^{2} }  \MM^{-1}  -  \int_{\Lambda_3 } \int_{\R^3} \gamma g_t^2 \,   {(n_1(x) \cdot v) (n_x \cdot v)\over \la v\ra^{2} }  \MM^{-1} =: I_1+ I_2.
$$
Using then the boundary conditions of Equation~\eqref{eq:Hypo_decay} with the choice of $\iota $ given by Assumption~\ref{item:H2}, we have that
$$
\beal
I_2 &= -  \int_{\Lambda_3 } \int_{\R^3} \gamma_+ g_t^2 \,   {(n_1(x) \cdot v) (n_x \cdot v)_+\over \la v\ra^{2} }  \MM^{-1}  + \int_{\Lambda_3 } \int_{\R^3} (\SSS \gamma_+ g_t)^2 \,   {(n_1(x) \cdot v) (n_x \cdot v)_-\over \la v\ra^{2} }  \MM^{-1}\\
&=-  \int_{\Lambda_3 } \int_{\R^3} \gamma_+ g_t^2 \,   {(n_1(x) \cdot v) (n_x \cdot v)_+\over \la v\ra^{2} }  \MM^{-1}  +  \int_{\Lambda_3 } \int_{\R^3} \gamma_+ g_t^2 \,   {(n_1(x) \cdot v) (n_x \cdot v)_+\over \la v\ra^{2} }  \MM^{-1} =0
\eeal$$
where we have used the change of variables $v\mapsto \VV_x v$ together with the fact that $ \lv \VV_x v \rv = \lv v\rv $, $|n(x) \cdot \VV_x v| = |n(x) \cdot v|$ and $n_1(x) \cdot \VV_x v = n_1(x) \cdot v$ due to \eqref{eq:N1normalOrtogonality_Hypo_decay} to deduce that the last line vanishes.

Gathering the above estimates and putting everything together into \eqref{eq:BdyApriori2} we have that
$$
\frac 12\frac \d{\dt} \int_{\OO} (g_t)^2 \MM^{-1} \la v\ra^{-2} (n_1(x) \cdot v)  \leq I_1 + C \lvv g_t\rvv_{\HH}^2,
$$
for some constant $C>0$. 
Integrating then in time we deduce that
$$
\int_0^t \int_{\Lambda_1\cup \Lambda_2} \int_{\R^3} (\gamma f_s)^2 \MM^{-1} (n_x \cdot v)^2 \la v\ra^{-2}\dv\d\sigma_x \ds \lesssim \lvv g_0 \lvv_{\HH}^2 + \lvv g_t\rvv_\HH^2 + \int_0^t \lvv g_s \rvv_{\HH}^2 \ds, 
$$
and we conclude by using \eqref{eq:A_priori_bound_HH_Hypo_decay}. 
\end{proof}

\subsection{Well-posedness}\label{ssec:Hypo_L2WellPosedness_Hypo_decay}
We introduce the boundary measures
$$\d \xi^1_\omega:=\omega (n_x\cdot v)\dv\d\sigma_x ,\quad, \d\xi_1^1 = \d\xi^1_{\MM^{-1}} \quad  \text{ and } \quad \d\xi_2:= \MM^{-1}  (n_x\cdot v)^2  \, \zeta_{\SSSS} (x) \, \la v\ra^{-2} \, \dv\d\sigma_x,
$$
where $\d\sigma_x$ stand for the Lebesgue measure on the boundary set $\partial \Omega$.
Secondly, we consider the space of renormalizing  functions $C_{pw, *}^1(\R)$ as the space of $C^1$ piecewise functions with finite limits at $\pm\infty$, and such that $s\mapsto \la s\ra\beta'(s)$ is bounded in $\R$.

We devote then the rest of this subsection to prove the following well-posedness result. 

\begin{theo}\label{theo:ExistenceL2_Hypo_decay}
Assume that either Assumption \ref{item:H1} or Assumption \ref{item:H2} holds and let $g_0 \in  {\HH}$. There is $g\in C(\R_+, {\HH})$ with an associated trace function $\gamma g\in L^2 (\Gamma;\,   \d\xi_2 \dt )$, 
 unique global solution to Equation~\eqref{eq:Hypo_decay} in the following weak sense: for any $\varphi \in \DD_{\SSSS}(\bar \UU)$ there holds 
\begin{multline}\label{eq:renormalizedFormulation_L2SSSS_Hypo_decay}
\int_{\OO} g(t,\cdot) \, \varphi(t,\cdot) \, \dx\dv-\int_0^t \int_{\OO^\eps} (Kg) \varphi  + g \left( \partial_t \varphi  - v\cdot \grad_x \varphi -\nu \, \varphi \right)   \dx\dv\ds \\
= \int_{\OO} g_0(\cdot ) \varphi(0,\cdot ) \dx\dv
+ \int_0^t \int_{\Sigma} \gamma g\,  \varphi\,  (n_x\cdot v) \d\sigma_x \dv ,
\end{multline}
for every $t\geq 0$, and where we have defined the set of test functions
\be\label{def:DD_SSSS}
\DD_{\SSSS} (\bar \UU)  := \left\{ \varphi \in C^\infty_c (\bar \UU); \, \text{ such that } \varphi = \zeta_{\SSSS} \phi, \text{ for some } \phi \in C^\infty_c (\bar \UU) \right\},
\ee
where we recall that $\zeta_{\SSSS}$ has been defined in \eqref{def:zetaSSSS_Hypo_decay}.
\end{theo}

\begin{rem}\label{rem:ExistenceL2_Hypo_decay}
In particular, Theorem~\ref{theo:ExistenceL2_Hypo_decay} implies the existence of a strongly continuous semigroup $S_{\LL} :{\HH}\to {\HH}$ associated to the solutions of Equation~\eqref{eq:Hypo_decay}.
\end{rem}

\begin{rem}
It is worth noting that, in the case of cylindrical domains (i.e. under Assumption~\ref{item:H2}), the functional space where the trace function is well defined is more singular than in the case of smooth domains. This is reminiscent of our comment from Remark~\ref{rem:L2_AprioriBdy_Hypo_decay}.
\end{rem}

\begin{rem}
The problem of well-posedness for transport equations with non-local terms in bounded domains has been deeply addressed in the literature, see, for instance,  \cite{sanchez2024kreinrutmantheorem, Bardos70, Beal87, MR1022305, MR2072842, Mischler2000, Mischler2010, MR1132764}.
In particular, this proof follows very similar lines as the ideas presented during the proof of \cite[Theorem 6.11]{BE_iso}, and we adapt it in order to treat non-isothermal boundary conditions.
\end{rem}

\begin{proof}[Proof of Theorem~\ref{theo:ExistenceL2_Hypo_decay}]
\medskip\noindent
We divide the proof into three steps.

\medskip\noindent
\emph{Step 1. (Auxiliary problem with inflow boundary conditions)} We consider $\gggg\in L^2(\Gamma;\dt \d\xi_1^1)$, and we study the following evolution equation
\begin{equation}
	\left\{\begin{array}{rlll}
		 \partial_t g&=&  \LL g   &\text{ in }\UU\\
		\gamma_-g&=& \gggg  &\text{ on }\Gamma_{-}\\
		g_{t=0}&=&g_0 &\text{ in }\OO.
	\end{array}\right.\label{eq:TransportL2K_Inflow_Hypo_decay}
\end{equation}
A direct application of \cite[Proposition 8.16]{sanchez2024kreinrutmantheorem} yields the existence of $g\in C(\R_+, {\HH})$, with a trace $\gamma g\in L^2(\Gamma; \d\xi^1_1 \dt)$, unique renormalized solution to Equation~\eqref{eq:TransportL2K_Inflow_Hypo_decay}, i.e. for every test function $\varphi \in \DD(\bar \UU)$ there holds
\begin{multline}\label{eq:renormalizedFormulation_HH_Inflow_Hypo_decay}
\int_{\OO} \beta(g)(t,\cdot) \, \varphi(t,\cdot) -\int_0^t \int_{\OO}  \beta'(g)\varphi \, Kg  + \beta(g) \left( \partial_t \varphi  -  v\cdot \grad_x \varphi  \right) - \nu \beta'(g) g \\
 - \int_0^t \int_{\Sigma_+} \gamma_+ \beta(g)\,  \varphi\,  (n_x\cdot v)_+ 
= \int_{\OO} \beta(g_0)(\cdot ) \varphi(0,\cdot ) 
- \int_0^t \int_{\Sigma_-}  \beta(\gggg) \,  \varphi\,  (n_x\cdot v)_- ,
\end{multline}
for every $t\geq 0$, and every renormalizing function $\beta\in C_{pw, *}^1(\R)$.

\medskip\noindent
\emph{Step 2. (Banach fixed point for modified Maxwell boundary conditions)} We take an arbitrary $T>0$ to be fixed later, a constant $\alpha \in (0,1)$, and we consider a function $h\in C(\R_+, {\HH})$, with a trace $\gamma h\in L^2(\Gamma; \d\xi_1^1 \dt)$
. We study now the following kinetic equation
\be
	\left\{\begin{array}{rlll}
		 \partial_t g &=& \LL g  &\text{ in }\UU_T := (0,T]\times \OO\\
		\gamma_- g&=& \alpha \, \RRR_\TTT \gamma_+ h   &\text{ on }\Gamma_{T, -} := (0,T]\times \Sigma_-\\
		g_{t=0}&=&g_0 &\text{ in }\OO.
	\end{array}\right.\label{eq:TransportL2K_k_Hypo_decay}
\ee
We define the weight function
\be\label{def:mu_MM}
\mu :=\left\{\begin{array}{ll}
\mu_0 & \text{if Assumption~\ref{item:H1} holds}\\
\mu_1 & \text{if Assumption~\ref{item:H2} holds},
\end{array}\right.
\ee
where we recall that $\mu_0$ and $\mu_1$ are defined in \eqref{def:muA_Hypo_decay} and \eqref{def:mu1_Hypo_decay} respectively. Furthermore, \eqref{eq:mu_mu0_Hypo_decay} and \eqref{eq:mu_mu1_Hypo_decay} imply that there is a constant $c>0$ such that
\be\label{eq:MM_mu_Hypo_decay}
c^{-1} \, \MM^{-1}\leq \mu^2 \leq c\, \MM^{-1}.
\ee
We note then that the computations performed during the proof of Proposition~\ref{prop:L2_Apriori_Hypo_decay} imply that
$$
\lvv \RRR_\TTT \gamma_+ h \rvv_{L^2(\Sigma_-, \d\xi^1_{\mu^2})} \leq \lvv \gamma_+ h \rvv_{L^2(\Sigma_+, \d\xi^1_{\mu^2})} .
$$
This, together with \eqref{eq:MM_mu_Hypo_decay}, further implies that 
$$
\lvv \RRR_\TTT \gamma_+ h \rvv_{L^2(\Sigma_-, \d\xi^1_{1})} \lesssim \lvv \gamma_+ h \rvv_{L^2(\Sigma_+, \d\xi^1_{1})} .
$$
Therefore, Step 1 implies the existence of $g\in C(\R_+, {\HH})$, with trace $\gamma g\in L^2(\Gamma; \d\xi_1^1\dt)$, unique renormalized solution of Equation~\eqref{eq:TransportL2K_k_Hypo_decay}, in the sense of \eqref{eq:renormalizedFormulation_HH_Inflow_Hypo_decay}. 

Furthermore, using the renormalized formulation \eqref{eq:renormalizedFormulation_HH_Inflow_Hypo_decay} with the test function $\varphi = \mu_0^2$ in the case of smooth domains, $\varphi = \mu_1^2$ in the case of cylindrical domains, and the renormalizing function $\beta(s) =  \beta_M(s) = M\wedge s^2$, for any $M>1$, we argue in a similar way as in the proof of Proposition~\ref{prop:L2_Apriori_Hypo_decay}, and using the integral version of the Grönwall lemma, and taking $M\to \infty$, we obtain that there is $\kappa \in \R$ such that
\begin{multline}\label{eq:Banach_energy_estimate_Hypo_decay}
\lvv g_t\rvv_{L^2_{\mu_0}(\OO)}^2 + \int_0^t e^{2\kappa   (t-s)}  \lvv \gamma_+ g_s\rvv^2_{L^2(\Sigma_+;  \d\xi^1_{\mu_0^2})}   \ds  \leq e^{2\kappa t} \lvv g_0\rvv_{L^2_{\mu_0}(\OO)}^2 \\
+ \alpha \, \int_0^t e^{2\kappa (t-s)}  \lvv \gamma_+ h_s \rvv^2_{L^2(\Sigma_+; \d\xi^1_{\mu_0^2} )}   \ds ,
\end{multline}
for every $t\in [0,T]$. This implies that the mapping $h \mapsto g$ is $\alpha$-Lipschitz for the norm defined by 
$$
\sup_{t \in [0,T]} \left\{ \| g_t \|^2_{L^2_{\mu_0}(\OO)}  e^{- 2\kappa  t }+  \int_0^t  \| \gamma_+ g_s \|^2_{L^2(\Sigma_+; \d\xi^1_{\mu_0^2})}  \, e^{- 2\kappa   s} \, \d s \right\}.
$$
The application of the Banach fixed point theorem, together with \eqref{eq:MM_mu_Hypo_decay} again, imply then the existence of $g\in C([0,T], {\HH})$, with a trace $\gamma g\in L^2([0,T]\times \Sigma; \d\xi^1_1 \dt)$, unique renormalized solution of the evolution equation
\beqn
	\left\{\begin{array}{rlll}
		 \partial_t g &=&  \LL  g &\text{ in }\UU_T\\
		\gamma_- g&=& \alpha \, \RRR_\TTT \gamma_+ g   &\text{ on }\Gamma_{T, -}\\
		g_{t=0}&=&g_0 &\text{ in }\OO, 
	\end{array}\right.
\eeqn
in the sense that for every $\varphi \in \DD(\bar \UU)$, and every $\beta \in C_{pw, *}^1$ we have
\begin{multline}\label{eq:renormalizedFormulation_HH_alpha_Hypo_decay}
\int_{\OO} \beta(g)(t,\cdot) \, \varphi(t,\cdot) -\int_0^t \int_{\OO}  \beta'(g)\varphi Kf  + \beta(g) \left( \partial_t \varphi  - v\cdot \grad_x \varphi  \right) -  \nu \beta'(g) g \\
 - \int_0^t \int_{\Sigma} \gamma_+ \beta(g)\,  \varphi\,  (n_x\cdot v)_+ 
= \int_{\OO} \beta(g_0)(\cdot ) \varphi(0,\cdot ) ,
\end{multline}
for every $t\in [0,T]$.

\medskip\noindent
\textit{Step 3.} For a sequence $\alpha_k \in (0,1)$ such that $\alpha_k \nearrow 1$, we consider the sequence $(g_k)$, obtained by using the Step~2, as the solution to the problem
\begin{equation}\label{eq:linear_gak_Hypo_decay}
\left\{
\begin{array}{rcll}
 \partial_t g_k &=& \LL  g_k   &\text{ in } \UU_T \\
 \gamma_{-} g_k   &=&  \alpha_k \, \RRR_\TTT \gamma_{+} g_k   &\text{ on }  \Sigma_{T, -} \\
 g_{k, t=0} &=& g_0   &\text{ in }   \OO.
\end{array}
\right.
\end{equation}
We consider again the renormalizing functions $\beta(s) =  \beta_M(s) = M\wedge s^2$, for any $M>1$, and the test function $\varphi = \mu^2$, where we recall that $\mu$ is defined in \eqref{def:mu_MM}. 
Using these choices in the renormalized formulation \eqref{eq:renormalizedFormulation_HH_alpha_Hypo_decay}, and arguing as in the proof of Proposition~\ref{prop:L2_Apriori_Hypo_decay}, passing to the limit as $M\to \infty$, using the integral version of the Grönwall lemma, and the equivalency \eqref{eq:MM_mu_Hypo_decay} between the weight functions, we deduce that $g_k$ satisfies 
\be\label{eq:linear_gak_bdd_Hypo_decay}
\lvv g_{kt} \rvv_{\HH}  \leq e^{\kappa t} \lvv g_0\rvv_{\HH}  , 
\ee
for any $t \in [0,T]$ and any $k \ge 1$.

Taking again $\beta (s) = \beta_M(s) = M\wedge s^2$, for any $M>1$, and $\varphi$ as the weight functions considered during the proof of Proposition~\ref{prop:L2_AprioriBdy_Hypo_decay}, we argue as before using the renormalized formulation and we deduce the conclusion of Proposition~\ref{prop:L2_AprioriBdy_Hypo_decay}. Namely, there holds
$$
\int_{\Gamma_T} (\gamma g_k)^2 \d\xi_2  \, \d t  
\lesssim  \| g_0 \|^2_{{\HH}} e^{\kappa  T} . 
$$

From the above estimates, we deduce that, up to the extraction of a subsequence, there exist $g \in L^2([0,T]; {\HH}) \cap L^\infty([0,T]; {\HH})$ and $\mathfrak{g}_\pm \in L^2(\Gamma_{T, \pm}; \d\xi_2 \dt)$ such that 
$$
g_k \wto g \hbox{ weakly in } \ L^2(0,T; {\HH})  \cap L^\infty(0,T; {\HH}), 
\quad
\gamma_\pm g_k \wto \mathfrak{g}_\pm \hbox{ weakly in } \  L^2(\Gamma_{T, \pm}; \d\xi_2 \d t). 
$$

Since $\langle v \rangle  \MM^{1/2} \in L^2(\R^3)$, 
 we have that $L^2(\Gamma_T;  \d\xi_2\dt) \subset L^1(\Gamma_T; \zeta_{\SSSS} (x) \, (n_x\cdot v)  \dv\d\sigma_x\dt)$. 
Moreover, from the very definition of the boundary condition of Equation~\eqref{eq:Hypo_decay}, we have that 
\be\label{eq:KolmogorovWPL12_Hypo_decay}
 \| \RRR \|_{L^1(\Sigma_+; \zeta_{\SSSS}(x)\,  (n_x\cdot v)  \dv\d\sigma_x) \to L^1(\Sigma_-; \zeta_{\SSSS}(x)\,  (n_x\cdot v)  \dv\d\sigma_x)} \le 1.
\ee
Altogether this implies that $\RRR(\gamma_+ g_{k})  \wto \RRR(\mathfrak{g}_+)$   weakly in $L^1(\Gamma_-; \zeta_{\SSSS} (x)\,  (n_x\cdot v)  \dv\d\sigma_x)$. 

Fur\-ther\-mo\-re, arguing as in the proof of \cite[Theorem~4.4]{Mischler2010} we deduce that $g$ admits a trace in the sense stablished by \cite[Lemma~6.8]{BE_iso}. Moreover, using \cite[Lemma~6.8--(T1)]{BE_iso} we have that $\gamma_\pm g_k \wto \gamma_\pm g$ weakly in $L^1 (\Gamma_{T, \pm}; \zeta_{\SSSS} \lv n_x \cdot v \rv \d\sigma_x \dv \dt)$. 
Using both convergences in the boundary condition $\gamma_- g_k = \RRR(\gamma_+ g_k)$, and the unicity of the limit, we obtain that $\gamma_- g = \RRR(\gamma_+ g)$ in the distributional sense.

We may thus pass to the limit in the weak formulation of Equation~\eqref{eq:linear_gak_Hypo_decay}, obtained from the Step 2, with a test function $\varphi \in \DD_\SSSS(\bar\UU)$ and we obtain that $g \in C([0,T]; {\HH})$ is a weak solution to Equation~\eqref{eq:Hypo_decay} in the sense of \eqref{eq:renormalizedFormulation_L2SSSS_Hypo_decay}. 
Moreover, passing to the limit in \eqref{eq:linear_gak_bdd_Hypo_decay}, we also have that Proposition~\ref{prop:L2_Apriori_Hypo_decay} holds. This and the linearity give the uniqueness of the solution to Equation~\eqref{eq:Hypo_decay}, and repeating this argument in the time intervals $[nT, (n+1)T]$ for every $n\in \N$ we conclude the existence and uniqueness of a global weak solution. 
\end{proof}

\subsection{Proof of Theorem~\ref{theo:Hypo_decay}} \label{ssec:Hypo_ProofMain_Hypo_decay}
The existence result is an immediate consequence of Theorem~\ref{theo:ExistenceL2_Hypo_decay}. We now consider $g_0\in \mathrm{Dom}(\LL)$, then the existence of a strongly continuous semigroup, given by Theorem~\ref{theo:ExistenceL2_Hypo_decay} and Remark~\ref{rem:ExistenceL2_Hypo_decay}, implies that we may apply Theorem~\ref{theo:Pert_hypo_Hypo_decay}, and together with the Grönwall lemma we deduce that there are constants $\kappa, C>0$ such that 
\beqn
\lvvv g_t\rvvv^2  \leq e^{-\kappa t} \lvvv g_0 \rvvv^2 \\
+  \theta_0\, C \int_0^t e^{-\kappa (t-s)} \lvv (\iota)^{1/2} (\widetilde{\gamma_+ g_s})\rvv_{L^2(\partial \Omega)}^2 \ds  ,
\eeqn
where we recall that the hypocoercivity norm $\lvvv \cdot \rvvv$ has been defined in \eqref{def:HypoNorm_Hypo_decay}.
Using then the norm equivalency \eqref{eq:NormEquivalence_Hypo_decay}, the above inequality yields
\beqn
\lvv g_t \rvv_\HH^2  \leq C' e^{-\kappa t} \lvv g_0 \rvv_\HH^2 \\
+  \theta_0\, C' \int_0^t e^{-\kappa (t-s)} \lvv (\iota)^{1/2} (\widetilde{\gamma_+ g_s})\rvv_{L^2(\partial \Omega)}^2 \ds  ,
\eeqn
for some constant $C'>0$.

Using now the Cauchy-Schwarz inequality together with Proposition~\ref{prop:L2_AprioriBdy_Hypo_decay} in the case of smooth domains we have that
\beqn
\int_{\partial \Omega} \iota (\widetilde{\gamma_+ g_s})^2 \d\sigma_x \lesssim \int_{\Sigma_+}  (\gamma_+ g_s)^2 \la v\ra^{-2} \MM^{-1} \, (n_x\cdot v)^2 \dv \d\sigma_x  \lesssim e^{\eta s} \lvv g_0 \rvv_\HH^2 , 
\eeqn
for some $\eta>0$.
Similarly, for the case of cylindrical domains, the use Proposition~\ref{prop:L2_AprioriBdyLambda12_Hypo_decay} yields
\beqn
\int_{\partial \Omega} \iota (\widetilde{\gamma_+ g_s})^2 \d\sigma_x = \int_{\Lambda_1\cup \Lambda_2}  (\widetilde{\gamma_+ g_s})^2 \d\sigma_x \lesssim \int_{\Lambda_1\cup \Lambda_2}  (\gamma_+ g_s)^2 \la v\ra^{-2} \MM^{-1} \, (n_x\cdot v)^2 \dv \d\sigma_x  \lesssim e^{\eta s} \lvv g_0 \rvv_\HH^2.
\eeqn
Putting the above estimates together we deduce that 
\beqn
\lvv g_t \rvv_\HH^2  \leq C' e^{-\kappa t} \lvv g_0 \rvv_\HH^2 \\
+  \theta_0\, {C' \over \eta} e^{\eta t} \lvv g_0 \rvv_\HH^2,
\eeqn
then a standard density argument implies the validity of the above inequality for any $g_0\in \HH$, and the associated family of solutions $g_t$, given by the well-posedness result of Theorem~\ref{theo:ExistenceL2_Hypo_decay}. 

\medskip
We now choose $T>0$ large enough such that $Ce^{-\kappa T/2} \leq 1/4$, and $\theta_\star >0 $ small enough such that 
$$
\theta_\star \, {C'\over \eta}  \, e^{\eta T} = \frac 14 e^{-\kappa T/2}.
$$
Then, for every $\theta_0\in (0,\theta_\star)$, it follows that
\be\label{eq:Hypo_decay_T}
\lvv g_T \rvv_\HH^2  \leq \frac 12 e^{-\kappa T/2} \lvv g_0 \rvv_\HH^2,
\ee
and 
\be\label{eq:Hypo_decay_t}
\lvv g_t \rvv_\HH^2  \lesssim e^{-\kappa t/2} \lvv g_0 \rvv_\HH^2 \qquad \forall t\in [0,T].
\ee
Take then $\tau >0$, there is $n\in \Z_+$ and $\tau_0 \in [0,T]$, such that $\tau = nT+ \tau_0$, thus the repeated use of \eqref{eq:Hypo_decay_T} together with \eqref{eq:Hypo_decay_t} imply the exponential decay \eqref{eq:HypoEquiv_Hypo_decay}.
\qed

\section{Weighted $L^2$ estimates}\label{sec:Hypo_Pert}

During this section, we extend the hypocoercivity results from Section~\ref{sec:Hypo_decay} to establish perturbed $L^2$ estimates on the solutions of Equation~\eqref{eq:LPBE}. 

We note that the main difference is the presence of the \emph{inflow}-type term $\psi$ at the boundary. However, since $\lla \psi\rra_{\Sigma_-} = 0$, we are able to repeat the computations from Section~\ref{sec:Hypo_decay} and, treating this term as a perturbation, we establish the following $L^2$ estimate.

\begin{theo}\label{theo:L2Decay_pert}
Assume that either Assumption \ref{item:H1} or Assumption \ref{item:H2} holds. 
There exists a unique global weak solution to Equation~\eqref{eq:LPBE}. Moreover, there are constructive constants $\kappa>0$ and $C\geq 1$ such that 
\be
    \lVert f_t \rVert_{ \HH}  \lesssim_C e^{-\kappa t}\lVert f_0\rVert_{ \HH} 
    +    \vartheta_0^{1/2} \eps^{11/2} \left( \int_0^t e^{- 2\kappa (t-s)} \left\lVert  (\iota)^{1/2} (\widetilde{\gamma_+ f_s} ) \right\rVert^2_{L^2(\partial \Omega)} \ds\right)^{1/2} +  \vartheta_0^{1/2} \eps^{11/2}  , \label{eq:Hypo}
\ee
for every $t\geq0$.
Furthermore, there is a norm $\lvvv \cdot \rvvv_\eps$ equivalent to the usual norm of $\HH$ uniformly in $\eps$, i.e. there is a constant $c>0$ independent of $\eps$ such that
\be\label{eq:ClassicHypoEquivalence}
c^{-1} \lvv f\rvv_{\HH} \leq \lvvv f\rvvv_\eps \leq c \lvv f\rvv_{\HH},
\ee
for which there holds
\begin{multline}
 \lvvv f_t \rvvv_\eps  \leq e^{-\kappa t}\lvvv f_0\rvvv_\eps  +  C^\star\,   \vartheta_0^{1/2} \eps^{11/2} \left( \int_0^t e^{-2\kappa (t-s)} \left\lVert  (\iota)^{1/2} (\widetilde{\gamma_+ f_s} ) \right\rVert^2_{L^2(\partial \Omega)} \ds\right)^{1/2}\\
 +  \vartheta_0^{1/2} \eps^{11/2}\, C^\star ,  \label{eq:HypoEquiv}
\end{multline}
for every $t\geq 0$, and some constant $C^\star>0$.
\end{theo}

\begin{rem}
In contrast with Theorem~\ref{theo:Hypo_decay}, Theorem~\ref{theo:L2Decay_pert} presents two main differences. On the one hand, we observe the presence of the last term of right-hand side \eqref{eq:Hypo} and \eqref{eq:HypoEquiv}, which does not depend on $f$, and it is reminiscent from the fact that Equation~\eqref{eq:LPBE} is not homogenous.

On the other, we have decided not to obtain an expression in the spirit of \eqref{eq:HypoEquiv_Hypo_decay}, since this would demand a bound on $\vartheta_0$, analog to the role of $\theta_\star$ on Theorem~\ref{theo:Hypo_decay}, to depend on $\eps$. Although on our final results (Theorems~\ref{theo:NESS} and \ref{theo:Main}) we indeed see a dependence on $\eps$ for the upper bounds of $\vartheta_0$, we would like to postpone the imposition of such a dependency until it is absolutely necessary, which could be meaningful to understand for future works on hydrodynamic limits.
\end{rem}

This section is structured similarly as Section~\ref{sec:Hypo_decay}, and the proof follow exactly the same ideas, with the main difference being the presence of the inflow-type term $\psi$ at the boundary.

\subsection{Perturbed hypocoercivity}\label{ssec:Hypo_Perturbed}
Throughout the rest of this subsection we consider $f\in \HH$ solution of Equation~\eqref{eq:LPBE} and the computations are done in the sense of a priori estimates. 

We observe first that, since $\lla \psi\rra_{\Sigma_-}=0$, then, repeating the computations performed in \eqref{eq:ConservationMassBdy} and using \eqref{eq:ConservationLawsCCC}, we have conservation of mass for solutions of Equation~\eqref{eq:LPBE}, i.e. for all $t\geq 0$ we have $\lla f_t \rra_\OO = \lla f_0\rra_\OO =0$.

\smallskip
Throughout the rest of this section we define $\vartheta' = \eps^{11} \vartheta$, $\vartheta_0' := \eps^{11} \vartheta_0$, and we devote the rest of this subsection to prove the following \emph{perturbed} $L^2$ estimate in the spirit of Theorem~\ref{theo:Pert_hypo_Hypo_decay}. 

\begin{theo}\label{theo:Pert_hypo}
Assume that either Assumption \ref{item:H1} or Assumption \ref{item:H2} holds. There exists a scalar product $\la\!\la \cdot , \cdot \ra \! \ra_\eps$ on the space $\HH$ so that the associated norm $\Nt \cdot \Nt_\eps$ is equivalent to the usual norm $\| \cdot \|_{\HH}$, and for which the linear operator $\LLL^\eps$ satisfies the following estimate: there are positive constants $\kappa, \kappa'  >0$ such that 
\begin{equation}\label{eq:pert_coercivity-LLL} 
\la \! \la - \LLL^\eps f , f \ra\!\ra_\eps \ge \kappa  \Nt f \Nt^2_\eps -  \kappa'  \vartheta_0' \left\lVert  (\iota)^{1/2} (\widetilde{\gamma_+ f} ) \right\rVert^2_{L^2(\partial \Omega)} -  \kappa' \vartheta_0'  ,
\end{equation}
for any $f \in \mathrm{Dom}(\LLL^\eps)$ satisfying the non-isothermal boundary condition of Equation~\eqref{eq:LPBE}.
\end{theo}

The rest of this subsection is devoted to constructing the scalar product $\lla \cdot, \cdot\rra_\eps$ generating the norm $\lvvv \cdot \rvvv_\eps$, and it follows in the same spirit and structure as Subsection~\ref{ssec:Hypo_Perturbed_Hypo_decay}.

\subsubsection{Microscopic coercivity for the full linearized operator $\LLL^\eps$} \label{ssec:Microscopic_part}

We now extend the results from Lemma~\ref{lem:MicroCoercivity_Hypo_decay} to problems presenting the boundary conditions exhibited in Equation~\eqref{eq:LPBE}.

\begin{lem}\label{lem:MicroCoercivity}
There is a constant $C>0$ such that
$$
\la -\LLL^\eps f, f \ra_{\HH} \geq \kappa_0 \eps^{-2} \lVert f^\perp \rVert^2_{\HH} + \frac {\eps^{-1}}4 \left \lVert \sqrt{\iota (2-\iota)}\DDD_1^\perp \gamma_+ f \right\rVert^2_{\partial\HH_+}  -  \vartheta_0'\, C \left\lVert   (\iota)^{1/2}   \left( \widetilde {\gamma_+ f}\right) \right\rVert^2_{L^2(\partial\Omega)} -  \vartheta_0' \, C ,
$$
where we recall that $\DDD_1 h := \MMM_1 \widetilde h$, $\DDD_1^\perp := Id - \DDD_1$, and $\partial \HH_+ = L^2(\Sigma_+,\, \MM^{-1} (v) \, \lvert n(x)\cdot v\rvert\,  \, \dv \d\sigma_x)$. 
\end{lem}

\begin{proof}
Using the Stokes theorem we compute
\beqn
 \la \LLL^\eps f, f \ra_{\HH}  = -\frac {\eps^{-1}}{2}  \int_{\OO} v\cdot \grad _x (f^2)  \MM^{-1} +   \eps^{-2}  \la \CCC f, f\ra_{\HH}  \leq  -\frac {\eps^{-1}}{2}  \int_{\Sigma} (\gamma f)^2 \MM^{-1} (n_x\cdot v)  - \kappa_0 \eps^{-2}  \|  f^\perp \|_{\HH}^2,
\eeqn
where we have used the microscopic coercivity property \eqref{eq:MicroCoercivity} to obtain the last inequality. 
It is left then to control the first term of the above inequality.

We observe then that the boundary condition of Equation~\eqref{eq:LPBE} implies that 
\bean
\TT := -\int_{{\Sigma}} (\gamma f)^2 \MM^{-1} (n_x\cdot v) &=&- \int_{{\Sigma_+}} (\gamma_+ f)^2 \MM^{-1} (n_x\cdot v)_+ + \int_{{\Sigma_-}} (\gamma_- f)^2 \MM^{-1} (n_x\cdot v) _- \\
&=& - \int_{{\Sigma_+}} (\gamma_+ f)^2 \MM^{-1} (n_x\cdot v)_+ \\
&&+ \int_{{\Sigma_-}}  \left( (1-\iota) \SSS\gamma_+ f +\iota \DDD_\Theta \gamma _+ f + \iota \psi  \right)^2 \MM^{-1} (n_x\cdot v) _-. 
\eean
Applying the change of variables $v\mapsto \VV_x v $ and using that $\DDD_\Theta \gamma_+ f (x,\VV_x v) = \DDD_\Theta \gamma_+ f(x,v)$, $\psi (x,\VV_xv) = \psi (x,v)$, and $|n(x) \cdot \VV_x v| = |n(x) \cdot v|$,  we have that
\bean
\TT &=&-\int_{{\Sigma_+}} (\gamma_+ f)^2 \MM^{-1} (n_x\cdot v)_+ + \int_{{\Sigma_+}}  \left( (1-\iota) \gamma_+ f +\iota \DDD_\Theta \gamma _+ f + \iota \psi \right)^2 \MM^{-1} (n_x\cdot v) _+ \\
&\leq & \int_{{\Sigma_+}} (\gamma_+ f)^2 \MM^{-1} (n_x\cdot v)_+ + \int_{{\Sigma_+}}  \left( (1-\iota) (\gamma_+ f)^2 +\iota (\DDD_\Theta \gamma _+ f + \psi)^2  \right) \MM^{-1} (n_x\cdot v) _+ \\
&\leq & \int_{{\Sigma_+}} -\iota (\gamma_+ f)^2 \MM^{-1} (n_x\cdot v)_+ + \int_{{\Sigma_+}} \iota (\DDD_\Theta \gamma _+ f + \psi)^2  \MM^{-1} (n_x\cdot v) _+ =: \TT_1 + \TT_2,
\eean
where we have used a convexity inequality for the function $x\mapsto x^2$ to obtain the second inequality.  
Using the decomposition from \eqref{eq:HypoBoundDecomposition_Hypo_decay}, and arguing as in the proof of Lemma~\ref{lem:MicroCoercivity_Hypo_decay} we have that
\be\label{eq:Boundary_Estimate_TT1}
\TT_1= -\int_{{\Sigma_+} } \iota \left( \DDD_1 \gamma_+ f \right)^2 \MM^{-1} (n_x\cdot v)_+ -\int_{{\Sigma_+} } \iota \left( \DDD_1^\perp \gamma_+ f \right)^2 \MM^{-1} (n_x\cdot v)_+ .
\ee
Furthermore, expanding the term $\TT_2$ and using the very definition of the diffusive boundary condition, we also observe that
$$
\begin{aligned}
\TT_2 &=  \int_{\Sigma_+} \iota (\DDD_1 \gamma _+ f)^2  \MM^{-1} (n_x\cdot v) _+  + \int_{\Sigma_+} \iota \left( \MMM_{\Theta}^2 - \MMM_1^2 \right) (\widetilde{\gamma _+ f})^2  \MM^{-1} (n_x\cdot v) _+ \\
& \quad \quad + 2\int_{\Sigma_+} \iota \psi \MMM_{\Theta} (\widetilde{\gamma_+f}) \MM^{-1} (n_x\cdot v)_+ + \int_{\Sigma_+} \psi^2 \MM^{-1} (n_x\cdot v)_+ =:\TT_{2,0} +  \TT_{2,1} + \TT_{2,2} + \TT_{2,3},
\end{aligned} 
$$
and we control each of the terms independently. 
We recall the function $P$ defined in \eqref{eq:PolyHypo_decay}, and we observe that
$$
 \int_{\R^3}  \left( \MMM_{\Theta}^2 - \MMM_1^2 \right) \MM^{-1} (n_x \cdot v)_+ \dv =  \sqrt{2\pi}  \left( {1\over ({\Theta})^2 (2-{\Theta})^2  - 1} \right) =P(\Theta) .
$$
Moreover, using the decomposition ${\Theta} = 1+ \eps \vartheta'$ we have that
\beqn\label{eq:BdyEstPTheta}
P({\Theta}) = \sqrt{2\pi}  \left( {1\over (1+\eps\vartheta')^2  (1-\eps\vartheta')^2}  - 1\right)  \lesssim    1-  (1-(\eps\vartheta')^2)^2  \lesssim (\eps\vartheta')^2 \leq \eps^{2} \vartheta_0'^2,
\eeqn
where we have used that $\lvv \vartheta'\rvv_{L^\infty(\partial {\Omega})}\leq \vartheta_0'$ with $\vartheta_0' \in (0,1/8)$.
Therefore, we deduce that 
$$
\TT_{2,1} \lesssim  \int_{\partial\Omega} \iota \,  \vartheta_0'^2 \,  \eps^{2}\,  (\widetilde{\gamma _+ f})^2.
$$
On the other hand, using \eqref{eq:Controlpsi} we deduce that 
$$
\TT_{2,2} \lesssim \int_{\partial \Omega} \iota \, (\widetilde{\gamma_+ f}) \, \vartheta_0' \, \eps \lesssim   \vartheta_0' \, \eps + \eps \, \vartheta_0'  \int_{\partial \Omega} (\iota)^2 \,  (\widetilde{\gamma_+ f})^2 ,
$$
where we have used the Young inequality to obtain the second inequality. Furthermore, using again \eqref{eq:Controlpsi} we also have that
$$
\TT_{2,3} \lesssim \vartheta_0' \eps.
$$
Altogether we have obtained that there is a constant $C>0$such that 
\be\label{eq:Boundary_Estimate_TT2}
\begin{aligned}
\TT_2 &\leq   \int_{{\Sigma_+}} \iota (\DDD_1 \gamma _+ f)^2  \MM^{-1} (n_x\cdot v) _+  + C \, \vartheta_0' \eps \left\lVert  (\iota)^{1/2}   \left( \widetilde {\gamma_+ f}\right) \right\rVert^2_{L^2(\partial{\Omega})} + C\,  \vartheta_0' \, \eps .
\end{aligned}
\ee
Finally, recalling that $\TT \leq \TT_1 + \TT_2$, and using \eqref{eq:Boundary_Estimate_TT1} and \eqref{eq:Boundary_Estimate_TT2}, we have that
\bean
\TT &\leq&  -\int_{{\Sigma_+} } \iota \left( \DDD_1^\perp \gamma_+ f \right)^2 \MM^{-1} (n_x\cdot v)_+ + C \, \vartheta_0' \, \eps \left\lVert  (\iota)^{1/2}   \left( \widetilde {\gamma_+ f}\right) \right\rVert^2_{L^2(\partial{\Omega})} + C\,  \vartheta_0' \, \eps \\
&\leq &  - \frac 12 \int_{{\Sigma_+} } \iota (2-\iota) \left( \DDD_1^\perp \gamma_+ f \right)^2 \MM^{-1} (n_x\cdot v)_+ + C \, \vartheta_0' \, \eps \left\lVert  (\iota)^{1/2}   \left( \widetilde {\gamma_+ f}\right) \right\rVert^2_{L^2(\partial{\Omega})} + C\,  \vartheta_0' \, \eps,
\eean
where we have employed the inequality $\iota \geq \iota(2-\iota)/2$ to obtain the second line. This concludes the proof.
\end{proof}

\subsubsection{Boundary terms}\label{ssec:BoundaryTerms}
We now have the following lemma, in the spirit of Lemma~\ref{lem:BoundaryTerms_Hypo_decay}, for the solutions of Equation~\eqref{eq:LPBE}.  

\begin{lem} \label{lem:BoundaryTerms}
Consider a function $\phi:\R^3\to \R$. For any $x\in \partial {\Omega}$ there holds
$$
\begin{aligned}
\int_{\R^3} \phi (v) \, \gamma f \, (n_x\cdot v)\dv &= \int_{\R^3} \phi(v) \, \iota(x) \, \DDD_1^\perp \gamma_+ f \,  (n_x\cdot v)_+\dv  \\
&+ \int_{\R^3} \left[ \phi(v) - \phi(\VV_x v) \right] (1-\iota (x) ) \, \DDD_1^\perp \gamma_+ f \,  (n_x\cdot v)_+\dv \\
& + \int_{\R^3}  \left[ \phi(v) - \phi(\VV_x v) \right] \DDD_1 \gamma_+ f \,  (n_x\cdot v)_+\dv \\
& - \left( \widetilde{\gamma_+ f} \right) \,  \sqrt{2\pi} \, \iota \left( \int_{\R^3}  \phi( v)   \, \psi(x,v) \, (n_x\cdot v)_- \dv \right)  - \iota \int_{\R^3}  \phi(v)   \psi \, (n(x) \cdot v)_-\, \dv.
\end{aligned}
$$
\end{lem}

\begin{proof}[Proof of Lemma~\ref{lem:BoundaryTerms}]
The proof follows the main ideas and computations of \cite[Lemma 3.2]{MR4581432}, by taking into account the arguments from the proof of Lemma~\ref{lem:BoundaryTerms_Hypo_decay}, and the presence of the additional term $\psi$. 
Moreover, in contrast with Lemma~\ref{lem:BoundaryTerms_Hypo_decay}, we have rewritten the right-hand side using the fact that, from its very definition, it follows that $\MMM_\Theta-\MMM_1 = \sqrt{2\pi} \, \psi $.
\end{proof}

\subsubsection{Energy, momentum and mass functionals} 
We devote this subsection to construct the functionals in order to control the energy, momentum and mass components of the macroscopic part $\Pi f$. 
We recall the energy, momentum and mass operators defined in \eqref{eq:Energy_Hypo_decay}, \eqref{eq:Momentum_Hypo_decay}, and \eqref{eq:Mass_Hypo_decay} respectively. Abusing notation, we define throughout this subsection $\EE := \EE[f]$, $\mu := \mu[f]$, and $\varrho := \varrho[f]$. 
We also recall the definition of the functionals $M_p$ and $M_q$ given by \eqref{eq:def-Mp_Hypo_decay} and \eqref{eq:def-Mq_Hypo_decay}, respectively. 

Moreover, we define $u[\EE]$ as the solution of the Poisson equation~\eqref{eq:PoissonEquation} with $\xi = \EE$ and boundary condition \ref{item:P1}; $U[\mu]$ is the solution to the elliptic system \eqref{eq:LameSystem} with data $\Xi = \mu$; $u_{N}[\varrho]$ is the solution to the Poisson equation with homogeneous Neumann boundary condition~\ref{item:P2} and with data $\xi = \varrho$. We also note that the existence of the previous objects is given by Theorems~\ref{theo:PoissonRegularity} and \ref{thm:LameRegularity}, see, for instance, Sub-subsections \ref{ssec:EnergyFunctional_Hypo_decay}, \ref{ssec:MomentumFunctional_Hypo_decay} and \ref{ssec:MassFunctional_Hypo_decay}.

\smallskip
Arguing then as in Section~\ref{sec:Hypo_decay}, we obtain the following lemmas.
\begin{lem}\label{lem:EnergyCoercivity}
There are constants $\kappa_1, C_1>0$ such that 
    \begin{multline*}
  \langle -\grad_x u[\EE], M_p [\LLL^\eps f]\rangle_{L^2({\Omega})} + \langle -\grad_x u[\EE[\LLL^\eps f]], M_p [ f]\rangle_{L^2({\Omega})}\\
       \geq \kappa_1 \eps^{-1} \lVert\EE\rVert^2_{L^2({\Omega})} - C_1 \eps^{-1}  \lVert \mu\rVert_{L^2({\Omega})} \lVert f^{\perp}\rVert_{ \HH} - C_1 \eps^{-1}   \left\lVert \sqrt{\iota(2-\iota)} \DDD_1^\perp \gamma_+f\right\rVert^2_{\partial\HH_+} \\
       - C_1 \eps^{-3}  \lVert f^{\perp}\rVert^2_{ \HH}  - C_1 \vartheta_0'^2  \eps  \left\lVert  \sqrt{\iota(2-\iota)}  \, \left( \widetilde{\gamma_+ f}\right) \right\rVert^2_{L^2(\partial {\Omega})}  - C_1\vartheta_0'^2 \eps  .
    \end{multline*}
\end{lem}
\begin{proof} 
The proof follows that of Lemma~\ref{lem:EnergyCoercivity_Hypo_decay}, using instead Lemma~\ref{lem:BoundaryTerms}, and controlling the term involving $\psi$ using \eqref{eq:Controlpsi}.
\end{proof}

\begin{lem}\label{lem:MomentumCoercivity}
There are constants $\kappa_2,C_2>0$ such that
\begin{multline*}
 \langle -\grad_x U[\mu], M_q[\LLL^\eps f]\rangle_{L^2({\Omega})} + \langle -\grad_x U[\mu[\LLL^\eps f]], M_q[ f]\rangle_{L^2({\Omega})} \geq 
 	\eps^{-1}\kappa_2 \lVert \mu\rVert^2_{L^2({\Omega})}  \\
 	- C_2  \, \eps^{-1}\lVert f^{\perp}\rVert_{ {\HH}}\lVert \varrho \rVert_{L^2({\Omega})} 
	 - C_2 \, \eps^{-1} \lVert\EE\rVert_{L^2({\Omega})}\lVert \varrho\rVert_{L^2({\Omega})} 
 	- C_2 \, \eps^{-1} \lVert\EE\rVert^2_{L^2({\Omega})}
         - C_2\, \eps^{-3} \lVert f^{\perp}\rVert^2_{ {\HH}} \\
         - C_2 \, \eps^{-1} \left\lVert \sqrt{\iota(2-\iota)} \DDD_1^\perp \gamma_+f\right\rVert^2_{\partial {\HH}_+} 
         - C_2 \, \eps \vartheta_0'^2 \left\lVert  \sqrt{\iota(2-\iota)} (\widetilde {\gamma_+ f}) \right\rVert^2_{L^2(\partial {\Omega})}  + C_2\, \vartheta_0'^2 \eps.
\end{multline*}
\end{lem}
\begin{proof} 
The proof follows that of Lemma~\ref{lem:MomentumCoercivity_Hypo_decay}, using instead Lemma~\ref{lem:BoundaryTerms}, and controlling the term involving $\psi$ using \eqref{eq:Controlpsi}.
\end{proof}

\begin{lem}\label{lem:MassCoercivity}
There are constants $\kappa_3,C_3>0$ such that
    \begin{multline*}
          \langle -\grad_x u_{\rm N}[\varrho], \mu[\LLL^\eps f]\rangle_{L^2({\Omega})} + \langle -\grad_x u_{\rm N} [\varrho[\LLL^\eps f]], \mu\rangle_{L^2({\Omega})}\\
        \quad \geq \kappa_3 \eps^{-1} \lVert \varrho\rVert^2_{L^2({\Omega})} 
        - C_3 \, \eps^{-1} \lVert \mu\rVert^2_{L^2({\Omega})}
        - C_3 \,  \eps^{-1}\lVert\EE\rVert^2_{L^2({\Omega})} 
        - C_3 \, \eps^{-1} \lVert f^{\perp}\rVert^2_{ {\HH^\eps}}  \\
	- C_3 \,  \eps^{-1}\left\lVert \sqrt{\iota(2-\iota)} \DDD_1^\perp \gamma_+f\right\rVert^2_{\partial{\HH}_+} 
	- C_3 \, \eps \, \vartheta_0'^2 \left\lVert  \sqrt{\iota(2-\iota)} (\widetilde{\gamma_+ f}) \right\rVert^2_{L^2(\partial{\Omega})}  
	- C_3 \, \eps \, \vartheta_0'^2    .
    \end{multline*}   
\end{lem}
\begin{proof} 
The proof follows that of Lemma~\ref{lem:MassCoercivity_Hypo_decay}, using instead Lemma~\ref{lem:BoundaryTerms}, and controlling the term involving $\psi$ using \eqref{eq:Controlpsi}.
\end{proof}

\subsubsection{Proof of Theorem \ref{theo:Pert_hypo}}\label{ssec:conclusion}
We define a scalar product on ${\HH}$
\be\label{eq:hypocNorm}
\begin{aligned}
\lla h,g\rra_\eps &:= \langle h,g\rangle_{ {\HH}}\\
&\qquad +\eta_1 \eps \langle-\grad_x u[\EE[h]], M_p[g]\rangle_{L^2({\Omega})} + \eta_1 \eps \langle-\grad_x u[\EE[g]], M_p[h]\rangle_{ L^2({\Omega})}\\
&\qquad +\eta_2 \eps  \langle-\grad^s_x U[\mu[h]], M_q[g]\rangle_{L^2({\Omega})} + \eta_2 \eps \langle-\grad^s_x U[\mu[g]], M_q[h]\rangle_{ L^2({\Omega})}\\
&\qquad +\eta_3 \eps  \langle-\grad_x u_{\rm N}[\varrho[h]], \mu[g]\rangle_{L^2({\Omega})} + \eta_3\eps  \langle-\grad_x u_{\rm N}[\varrho[g]], \mu[h]\rangle_{ L^2({\Omega})},
\end{aligned}
\ee
for some parameters $0 \ll \eta_3 \ll \eta_2 \ll \eta_1 \ll 1$ to be chosen later, and where we recall that the moments $M_p$ and $M_q$ are defined respectively in \eqref{eq:def-Mp_Hypo_decay} and \eqref{eq:def-Mq_Hypo_decay}; $u[\EE[f]]$ is the solution of the Poisson equation~\eqref{eq:PoissonEquation} with $\xi = \EE[f]$ and boundary condition \ref{item:P1}; $U[\mu[f]]$ is the solution to the elliptic system \eqref{eq:LameSystem} with data $\Xi = \mu[f]$; $u_{N}[\varrho[f]]$ is the solution to the Poisson equation with homogeneous Neumann boundary condition~\ref{item:P2} and with data $\xi = \varrho[f]$, and similarly for the terms depending on $g$. 

We define next the norm associated to the above scalar product 
\be\label{def:HypoNorm}
\lvvv f \rvvv_\eps := \displaystyle \sqrt{\lla f,f\rra_\eps},
\ee
and, arguing as in the proof of Theorem~\ref{theo:Pert_hypo_Hypo_decay}, we deduce that there is a constant $c>0$, such that
$$
  \lvv f\rvv_{\HH}^2 \left( 1- 2c\, \eps(\eta_1 + \eta_2 + \eta_3) \right)   \leq  \lvvv f \rvvv_\eps^2 \leq   \lvv f\rvv_{\HH}^2 \left( 1 + 2c\, \eps(\eta_1 + \eta_2 + \eta_3) \right) . 
$$
By choosing then $\eta_1, \eta_2, \eta_3 \in (0, (12c)^{-1})$, and the fact that $\eps\leq 1$, there yields the equivalency of norms
\begin{equation}\label{eq:NormEquivalence}
    \lVert f\rVert_{ {\HH}}\lesssim \lvvv f\rvvv_\eps \lesssim \lVert f\rVert_{ {\HH}}.
\end{equation}

Let now $f$ satisfy the assumptions of Theorem~\ref{theo:Pert_hypo}. Recalling that, during this section, we denote $\varrho=\varrho[f]$, $\mu=\mu[f]$ and $\EE = \EE[f]$, noting that $\sqrt{\iota(2-\iota)} \ge \iota$, gathering Lemmas~\ref{lem:MicroCoercivity},~\ref{lem:EnergyCoercivity},~\ref{lem:MomentumCoercivity} and~\ref{lem:MassCoercivity}, and arguing as in the proof of Theorem~\ref{theo:Pert_hypo_Hypo_decay} one  has
$$
\begin{aligned}
\la \! \la - \LLL^\eps f , f \ra \! \ra_\eps 
&\ge  \left(\frac{\kappa_0}{2} - \eta_1 C - \eta_2 C - \eta_3 C\right) \| f^\perp \|_{{\HH}}^2  +\left(\frac{\eta_1 \kappa_1}{2} - \eta_2 C - \eta_3 C \right)\| \EE \|_{L^2_x({\Omega})}^2  \\
&\quad 
+ \left( \eta_2 \kappa_2 - \eta_1^2 C - \eta_3 C \right)\| \mu \|_{L^2_x({\Omega})}^2  
+ \left( \eta_3 \kappa_3  - \eta_2^2 C - \frac{\eta_2^2}{\eta_1}C\right)\| \varrho \|_{L^2_x({\Omega})}^2 
\\
&\quad
+ \left(\frac12 - \eta_1 C - \eta_2 C - \eta_3 C  \right) \left\lvv \sqrt{\iota(2-\iota)} \DDD_1^\perp \gamma_+ f \right\rvv_{\partial {\HH}_+}^2  - C (1+ \eta _1+ \eta_2 + \eta_3) \vartheta_0' \\
&\quad 
 -C \left\lVert  \vartheta_0' \,  (\iota)^{1/2}  \left( \widetilde {\gamma_+ f}\right) \right\rVert^2_{L^2(\partial{\Omega})}  -C(\eta_1 + \eta_2 + \eta_3) \left\lVert  \vartheta_0' \,  \sqrt{\iota(2-\iota)}  \left( \widetilde {\gamma_+ f}\right) \right\rVert^2_{L^2(\partial{\Omega})}.
\end{aligned}
$$
Arguing again as in the proof of Theorem~\ref{theo:Pert_hypo_Hypo_decay} we choose $\eta_1 := \eta$, $\eta_2 := \eta^{\frac{3}{2}}$, $\eta_3 := \eta^{\frac{7}{4}}$, with $0 < \eta \ll1$ small enough, and we deduce that
$$
\begin{aligned}
\la \! \la - \LLL^\eps f , f \ra \! \ra_\eps
&\ge \kappa \left( \| f^\perp \|_{{\HH}}^2 
+\| \varrho \|_{L^2_x({\Omega})}^2 
+\| \mu \|_{L^2_x({\Omega})}^2 
+\| \EE \|_{L^2_x({\Omega})}^2  \right)   \\
&\quad 
-\kappa' \left\lVert  \vartheta_0' \,  (\iota)^{1/2}  \left( \widetilde {\gamma_+ f}\right) \right\rVert^2_{L^2(\partial{\Omega})} - \kappa' \vartheta_0' + \kappa'' \left\lvv \sqrt{\iota(2-\iota)} \DDD_1^\perp f_{+} \right\rvv_{\partial {\HH}_+}^2, 
\end{aligned}
$$
for some constants $\kappa,\kappa', \kappa'' >0$.
We conclude the proof of Theorem~\ref{theo:Pert_hypo} by using \eqref{eq:NormDecompositionMicroMacro_Hypo_decay}
and the norm equivalency \eqref{eq:NormEquivalence}. 
\qed

\subsection{Interior a priori estimates}\label{ssec:Hypo_Apriori}

This subsection is devoted to the obtention of a priori estimates of the solution of Equation~\eqref{eq:LPBE} 

\begin{prop}\label{prop:L2_Apriori}
Assume that either Assumption \ref{item:H1} or Assumption \ref{item:H2} holds
There is $\kappa \in \R$ such that for every $f$ solution of Equation~\eqref{eq:LPBE} there holds
\be\label{eq:A_priori_bound_HH}
\lvv f_t \rvv_{\HH} \lesssim e^{\kappa \, \eps^{-2} \, t} \lvv f_0\rvv_{\HH} + \left(e^{\kappa \, \eps^{-2} \, t}-1\right) \, \eps^{23/2}\, \vartheta_0,
\ee
for every $t\geq 0$.
\end{prop}

\begin{rem}
The proof follows the same lines as that of Proposition~\ref{prop:L2_Apriori_Hypo_decay}, and we adapt it in order to consider the presence of the inflow-type term $\psi$. 
\end{rem}

\begin{proof}[Proof of Proposition~\ref{prop:L2_Apriori}]
We divide the proof in two different cases: first we obtain \eqref{eq:A_priori_bound_HH} for smooth domains, and then we repeat and adapt those computations for the setting of the cylinder. 

\medskip\noindent
\emph{Case 1. (Smooth domains---Assumption~\ref{item:H1})}
We recall the modified weight function $\mu_0$ defined in \eqref{def:muA_Hypo_decay} satisfying \eqref{eq:mu_mu0_Hypo_decay}. Arguing then as in the Case 1 of the proof of Proposition~\ref{prop:L2_Apriori_Hypo_decay}, we have that if $f$ is a solution of Equation~\eqref{eq:LPBE} there holds
\be\label{eq:Well_posedness_HH_1}
\begin{aligned}
\frac 12 \frac d{dt} \int_{\OO} f_t^2\mu_0^2  
&
=  - \frac {\eps^{-1}}2 \int_{\Sigma} \gamma f_t^2 \,  (n_x\cdot v) \mu_0^2 + \frac {\eps^{-1}}2 \int_{\OO} f_t^2 (v\cdot \grad_x \mu_0^2)  
\\&\qquad \qquad \qquad \qquad 
- \eps^{-2}\int_{\OO} \nu f_t^2 \mu_0^2 + \eps^{-2}\int_{\OO} f_t (Kf_t) \mu_0^2 
=: I_1+I_2,
\end{aligned}
\ee
where we have used integration by parts, and we have defined 
\bean
I_1 &:=& - \frac {\eps^{-1}}2 \int_{\Sigma} \gamma f_t^2 \,  (n_x\cdot v) \mu_0^2,  \\
I_2 &:=& \frac {\eps^{-1}}2 \int_{\OO} f_t^2 (v\cdot \grad_x \mu_0^2)  -{\eps^{-2}} \int_{\OO} \nu f_t^2 \mu_0^2 +{\eps^{-2}} \int_{\OO} f_t (Kf_t) \mu_0^2.
\eean

\medskip
Arguing as in the proof of Step 1 of Case 1 of the proof of Proposition~\ref{prop:L2_Apriori_Hypo_decay}, we have that 
$$
I_2 \lesssim \eps^{-2}  \lvv f_t\rvv_{L^2_{\mu_0}(\OO)}^2,
$$
for every $A\geq 1$.

\medskip
For the boundary terms, we observe that 
$$
I_1 = - \frac {\eps^{-1}}2 \int_{\Sigma} \gamma f_t^2 \,  (n_x\cdot v) \mu_A^2  - \frac {\eps^{-1}}2 \int_{\Sigma} \gamma f_t^2 \,   {(n_x \cdot v)^2\over \la v\ra^{2} }  \mu_A^2 =: I_{1,1} + I_{1,2},
$$
and we estimate each of the above integrals separately. 
On the one hand, using the boundary conditions of Equation~\eqref{eq:LPBE}, we have that
$$\beal
I_{1,1} &=  - \frac {\eps^{-1}}2 \int_{\Sigma_+} \gamma_+ f_t^2 \,  \mu_A^2\, (n_x\cdot v)_+  +  \frac {\eps^{-1}}2 \int_{\Sigma_-}  (\RRR\gamma_+ f_t + \iota \psi) ^2 \, \mu_A^2 \,  (n_x\cdot v)_- \\
 &\leq - \frac {\eps^{-1}}2 \int_{\Sigma_+} \gamma_+ f_t^2 \, \mu_A^2 \, (n_x\cdot v)_+  + \frac {\eps^{-1}}2  \int_{\Sigma_-}  (1-\iota) (\SSS\gamma_+ f_t) ^2 \,  \mu_A^2\, (n_x\cdot v)_-  \\
 &\quad \qquad + \frac {\eps^{-1}}2   \int_{\Sigma_-}  \iota \left(\MMM_{\Theta} \widetilde{\gamma_+ f_t} + \psi \right) ^2 \,  \mu_A^2\, (n_x\cdot v)_-  \\
 &\leq  -\frac {\eps^{-1}}2 \int_{\Sigma_+} \iota (\gamma_+ f_t)^2 \,\mu_A^2\,   (n_x\cdot v)_+ + \frac {\eps^{-1}}2 \int_{\partial \Omega} \iota \left( \widetilde{\gamma_+ f}\right)^2 \int_{\R^3} \MMM_{\Theta}^2 \mu_A^2 (n_x \cdot v)_+ \\
 &\qquad \qquad  + {\eps^{-1}} \int_{\partial \Omega} \iota \left(\widetilde{\gamma_+ f_t}\right)  \int_{\R^3}  \MMM_\Theta  \, \psi \, \mu_A^2 \, (n_x\cdot v)_+ + \frac {\eps^{-1}}2 \int_{\Sigma_+} \iota \, \psi^2\, \mu_A^2 \, (n_x\cdot v)_+,
\eeal$$
where we have successively used a convexity inequality with the function $x\mapsto \lv x\rv^2$ and the change of variables $v\mapsto \VV_x v$ together with the fact that $ \lv \VV_x v \rv = \lv v\rv $ and $|n(x) \cdot \VV_x v| = |n(x) \cdot v|$. 

Applying now the Cauchy-Schwarz inequality we have that
$$\beal
\left( \widetilde{\gamma_+ f}\right)^2 = \left( \int_{\R^3} \gamma_+ f (n_x \cdot v)_+ \right) ^2 \leq \left( \int_{\R^3} (\gamma_+ f)^2 \mu_A^2 (n_x \cdot v)_+  \right)  \left( \int_{\R^3} \mu_A^{-2} (n_x\cdot v)_+ \right).
\eeal $$
Moreover, using the Young inequality we also have that 
$$
\left(\widetilde{\gamma_+ f_t}\right)  \int_{\R^3}  \MMM_\Theta  \, \psi \, \mu_A^2 \, (n_x\cdot v)_+ \leq \frac \alpha2 \left(\widetilde{\gamma_+ f_t}\right) ^2 + {1\over 2\alpha } \left(\int_{\R^3}  \MMM_\Theta  \, \psi \, \mu_A^2 \, (n_x\cdot v)_+  \right)^2,
$$
Using now \eqref{eq:Controlpsi} we further have that 
$$\beal
\left(  \int_{\R^3}  \MMM_\Theta  \, \psi \, \mu_A^2 \, (n_x\cdot v)_+\right)^2 + \int_{\Sigma_+} \iota \, \psi^2\, \mu_A^2 \, (n_x\cdot v)_+ \lesssim \vartheta_0^2 \,\eps^{24}.
\eeal $$
Applying then the Young inequality in the above expression of $I_{1,1}$, and using the above inequalities, we deduce that for any $\alpha >0$ we have that
$$
I_{1,1} \leq {C \, \eps^{23} \over 2\alpha } \, \vartheta_0^2-\frac {\eps^{-1}}2 \int_{\partial \Omega} \iota \left( \widetilde{\gamma_+ f}\right)^2  \left(\II_{A,1}(x) -  \alpha\right) ,
$$
for some constant $C>0$, and where we have defined
\be\label{def:IIA_1}
\II_{A, 1} (x) := \left( \int_{\R^3} \mu_A^{-2} (n_x\cdot v)_+ \right)^{-1} -  \int_{\R^3} \MMM_{\Theta}^2 \mu_A^2 (n_x \cdot v)_+,
\ee
and we note that as $A\to +\infty$ there holds
\be\label{eq:LimIIA1}
\II_{1, A}(x) \to \frac 12 \left[ \left( \widetilde{\MMM_{\Theta}} \right)^{-1} + \widetilde{\MMM_{\Theta}} \right]= \frac 12  \left(1-1\right) =0,
\ee
uniformly in $x$.
On the other hand, we have that 
$$\beal
I_{1,2}& = - \frac {\eps^{-1}}2 \int_{\Sigma_+} (\gamma_+ f_t)^2 \,   {(n_x \cdot v)_+^2\over \la v\ra^{2} }  \mu_A^2 - \frac {\eps^{-1}}2 \int_{\Sigma_-} (\gamma_- f_t)^2 \,   {(n_x \cdot v)_-^2\over \la v\ra^{2} }  \mu_A^2  \\
& \leq - \frac {\eps^{-1}}2 \int_{\Sigma_+} (\gamma_+ f_t)^2 \,   {(n_x \cdot v)_+^2\over \la v\ra^{2} }  \mu_A^2 \leq  -\frac {\eps^{-1}}2 \int_{\partial \Omega} \iota^{\eps} \left( \widetilde{\gamma_+ f} \right)^2 \II_{A, 2} (x),
\eeal$$
where we have used the Cauchy-Schwarz inequality and the fact that $\iota (x) \in [0,1]$ for every $x\in \partial \Omega$ to obtain the last inequality. We have also set
\be\label{def:IIA2}
\II_{A, 2}(x) := \left( \int_{\R^3} \la v\ra^2 \mu_A^{-2}\right)^{-1} \underset{A\to +\infty}{\longrightarrow} \int_{\R^3} \la v\ra^2 \MMM_{\Theta}   \in (0, +\infty),
\ee
uniformly for $x\in \partial \Omega$, and we note that the bounds on the last limit are due to our condition \eqref{eq:ConditionsTheta} on $\Theta$.

Gathering the above estimates, and choosing the parameters $A > 1$ large enough, and $\alpha >0$ small enough, such that $\II_{A,1}(x) + \II_{A, 2} (x) - \alpha  \geq 0$ uniformly for $x\in \partial \Omega$, we deduce that
$$
I_1 = I_{1,1} + I_{1,2} \leq {C \, \eps^{23} \over 2\alpha } \vartheta_0^2 -\frac {\eps^{-1}}2 \int_{\partial \Omega} \iota^{\eps} \left( \widetilde{\gamma_+ f} \right)^2 \left[ \II_{A,1}(x) + \II_{A, 2} (x)-\alpha  \right]  \leq C \, \eps^{23}\, \vartheta_0^2 ,
$$
for some constant $C>0$.

\medskip
By putting together the estimates for $I_1$ and $I_2$ obtained during Steps 1 and 2, into \eqref{eq:Well_posedness_HH_1}, we have that
$$
\frac 12 \frac d{dt} \int_{\OO} f_t^2\mu_0^2  \lesssim \eps^{-2} \lvv f_t\rvv_{L^2_{\mu_0}(\OO)}^2 + C\, \eps^{23} \, \vartheta_0^2 .
$$
We conclude the proof for Case 1 (under Assumption~\ref{item:RH1}) by using the Grönwall lemma and \eqref{eq:mu_mu0_Hypo_decay}.

\medskip\noindent
\emph{Case 2. (Cylindrical domains---Assumption~\ref{item:H2})}
The proof for cylindrical domains follows by using the arguments from the Case of the proof of Proposition~\ref{prop:L2_Apriori_Hypo_decay} together with the arguments to handle the presence of $\psi$ on the boundary conditions used during the Case 1 of this proof. 
\end{proof}

\subsection{Boundary a priori estimates}\label{ssec:Hypo_AprioriBdy}
In this subsection we provide a priori estimates for the trace of the solutions of Equation~\eqref{eq:LPBE}, in the spirit of the estimates given in Subsection~\ref{ssec:Hypo_AprioriBdy_Hypo_decay}.

 \begin{prop}\label{prop:L2_AprioriBdy}
Assume that either Assumption \ref{item:H1} or Assumption \ref{item:H2} holds
There is $\kappa \in \R$ such that for every $f$ solution of Equation~\eqref{eq:LPBE} there holds
\be\label{eq:A_priori_bound_Bdy}
\int_0^t \int_{\Sigma}  (\gamma f_s)^2 \, (n_x\cdot v)^2 \, \zeta_{\SSSS}(x) \, \la v\ra^{-2} \, \MM^{-1} \,  \dv\d\sigma_x\ds  \lesssim \eps \, e^{\kappa \, \eps^{-2} \, t} \lvv f_0\rvv_{\HH}^2 +   \vartheta_0^2 \eps^{24} e^{\kappa \, \eps^{-2} \, t},
\ee
for every $t\geq 0$.
\end{prop}

\begin{proof}[Proof of Proposition~\ref{prop:L2_AprioriBdy}]
The proof is nothing but the exact repetition of the ideas of \cite[Proposition 6.5]{BE_iso}, using Proposition~\ref{prop:L2_Apriori}, thus we skip it. 
\end{proof}

Furthermore, we now prove an additional a priori estimate for the trace in the spirit of Proposition~\ref{prop:L2_AprioriBdyLambda12_Hypo_decay}.

 \begin{prop}\label{prop:L2_AprioriBdyLambda12}
Assume that Assumption \ref{item:H2} holds
There is $\kappa \in \R$ such that for every $f$ solution of Equation~\eqref{eq:LPBE} there holds
\be\label{eq:A_priori_bound_BdyLambda12}
\int_0^t \int_{\Lambda_1\cup \Lambda_2} \int_{\R^3} (\gamma f_s)^2 \, (n_x\cdot v)^2 \,  \la v\ra^{-2} \, \MM^{-1} \,  \dv\d\sigma_x\ds  \lesssim \eps \, e^{\kappa \eps^{-2} t} \lvv f_0\rvv_{\HH}^2 + \vartheta_0^2 \eps^{24} e^{\kappa \, \eps^{-2} \, t},
\ee
for every $t\geq 0$.
\end{prop}

\begin{proof}
The proof is an exact repetition of the proof of Proposition~\ref{prop:L2_AprioriBdyLambda12_Hypo_decay} thus we skip it.
\end{proof}


\subsection{Well-posedness}\label{ssec:Hypo_L2WellPosedness}

We now present a well-posedness result, analog to that from Theorem~\ref{theo:ExistenceL2_Hypo_decay}. 

\begin{theo}\label{theo:ExistenceL2}
Assume that either Assumption \ref{item:H1} or Assumption \ref{item:H2} holds and let $f_0 \in  {\HH}$. There is $f\in C(\R_+, {\HH})$ with an associated trace function $\gamma f\in L^2 (\Gamma;\,   \d\xi_2 \dt )$, 
 unique global solution to Equation~\eqref{eq:LPBE} in the following weak sense: for any $\varphi \in \DD_{\SSSS}(\bar \UU)$ there holds 
\begin{multline}\label{eq:renormalizedFormulation_L2SSSS}
\int_{\OO} f(t,\cdot) \, \varphi(t,\cdot) \dx\dv-\int_0^t \int_{\OO^\eps} \eps^{-2} (Kf) \varphi  + f \left( \partial_t \varphi  - \eps^{-1} v\cdot \grad_x \varphi - \eps^{-2} \nu \, \varphi \right) \dx\dv\ds  \\
= \int_{\OO} f_0(\cdot ) \varphi(0,\cdot ) \dx\dv
+ \eps^{-1}  \int_0^t \int_{\Sigma} \gamma g\,  \varphi\,  (n_x\cdot v)  \d\sigma_x\dv\ds,
\end{multline}
for every $t\geq 0$, and where we recall that $\DD_{\SSSS} (\bar \UU)$ has been defined in \eqref{def:DD_SSSS}.
\end{theo}

\begin{rem}\label{rem:ExistenceL2}
In particular, Theorem~\ref{theo:ExistenceL2} implies the existence of a strongly continuous semigroup $S_{\LLL^\eps} :{\HH}\to {\HH}$ associated to the solutions of Equation~\eqref{eq:LPBE}.
\end{rem}

\begin{proof}
The proof follows exactly the arguments from the proof of Theorem~\ref{theo:ExistenceL2_Hypo_decay}, by further taking into account the presence of the inflow-type term $\psi$ at the boundary during the Step~2 of the proof Theorem~\ref{theo:ExistenceL2_Hypo_decay}. We note though that the ideas and arguments remain in the same spirit using instead the estimates obtained during this section, we thus skip it. 
\end{proof}

\subsection{Proof of Theorem~\ref{theo:L2Decay_pert}} \label{ssec:Hypo_ProofMain}
The proof follows exactly the same density argument used on the proof of Theorem~\ref{theo:Hypo_decay} using  Theorem~\ref{theo:Pert_hypo} together with the trace estimates given by Proposition~\ref{prop:L2_AprioriBdy} for smooth domains, and Proposition~\ref{prop:L2_AprioriBdyLambda12} for cylindrical domains; and the well-posedness result Theorem~\ref{theo:ExistenceL2}. We thus only sketch it, to explicit the dependencies on $\eps$.

Consider $f_0\in \HH$, and $(f_0^n)_{n\in \N }\subset \mathrm{Dom}(\LLL^\eps)$ such that $f_0^n \to f_0$ strongly in $\HH$ as $n\to \infty$. The existence of a strongly continuous semigroup, given by Theorem~\ref{theo:ExistenceL2} and Remark~\ref{rem:ExistenceL2}, implies the existence of a family $(f_t^n)_{t\geq0}\subset \mathrm{Dom}(\LLL^\eps)$ solving Equation~\eqref{eq:LPBE} with initial data $f_0^n$ for every $n\in \N$. Moreover, we may apply Theorem~\ref{theo:Pert_hypo}, and together with the Grönwall lemma we have that there are constants $\kappa, C>0$ such that
\be\label{eq:HypoEquiv_Dom}
\lvvv f_t^n\rvvv_\eps^2  \leq e^{-\kappa t} \lvvv f_0^n \rvvv_\eps^2 \\
+  \vartheta_0'\, C \int_0^t e^{-\kappa (t-s)} \lvv (\iota)^{1/2} (\widetilde{\gamma_+ f_s^n})\rvv_{L^2(\partial \Omega)}^2 \ds  + C\vartheta_0',
\ee
where we recall that the hypocoercivity norm $\lvvv \cdot \rvvv_\eps$ has been defined in \eqref{def:HypoNorm}.
Using then the norm equivalency \eqref{eq:NormEquivalence}, the above inequality yields
\be\label{eq:Hypo_Dom}
\lvv f_t^n \rvv_\HH^2  \leq C e^{-\kappa t} \lvv f_0^n \rvv_\HH^2 \\
+  \vartheta_0' \, C \int_0^t e^{-\kappa (t-s)} \lvv (\iota)^{1/2} (\widetilde{\gamma_+ f_s^n})\rvv_{L^2(\partial \Omega)}^2 \ds + C\vartheta_0' \qquad \forall t\geq 0,
\ee
for some $C'>0$. 
Using now the Cauchy-Schwarz inequality together with Proposition~\ref{prop:L2_AprioriBdy} in the case of smooth domains and Proposition~\ref{prop:L2_AprioriBdyLambda12_Hypo_decay} for cylindrical domains we have that
\be\label{eq:TraceDom}
\int_{\partial \Omega} \iota (\widetilde{\gamma_+ f_s^n})^2 \d\sigma_x \lesssim \int_{\Sigma_+}  (\gamma_+ f_s^n)^2 \la v\ra^{-2} \MM^{-1} \, (n_x\cdot v)^2 \dv \d\sigma_x  \lesssim e^{\eta \, \eps^{-2} \, s} \lvv f_0^n \rvv_\HH^2 , 
\ee
for some $\eta>0$.
Putting together the above estimates we obtain that 
\beqn
\lvv f_t^n \rvv_\HH^2  \leq C' e^{-\kappa t} \lvv f_0^n \rvv_\HH^2 \\
+  \vartheta_0'\, \eps^2 {C' \over \eta}  e^{\eta  \, \eps^{-2} \, t} \lvv f_0^n \rvv_\HH^2 + C'\vartheta_0'.
\eeqn
Let $f\in \HH$ be the solution of Equation~\eqref{eq:LPBE} with initial data $f_0$. Given $T>0$, the above estimate, the fact that strongly in $\HH$ as $n\to \infty$, and the linearity of the problem implies that implies that $f^n\to f$ strongly in $L^\infty((0,T), \HH)$ as $n\to \infty$.

Using the fact that $T>0$ is arbitrary, the above convergence, the trace estimate \eqref{eq:TraceDom}, and Lebesgue's dominated convergence theorem we recover \eqref{eq:Hypo} and \eqref{eq:HypoEquiv} from \eqref{eq:Hypo_Dom} and \eqref{eq:HypoEquiv_Dom} respectively.  
\qed


\section{A priori weighted $L^\infty$ estimates}\label{sec:AprioriLinfty}
We devote this section to  establish the following $L^\infty$ bound on the solutions of Equation~\eqref{eq:PLPBE}. 

\begin{prop}\label{prop:LinftyEstPert} 
Assume that either Assumption \ref{item:H1} or Assumption \ref{item:H2} holds, let $\omega$ be an admissible weight function, let $G:\UU\to \R$ satisfy $\lla G_t\rra_{\OO}=0$ for every $t\geq 0$, and let $f$ be a solution of Equation~\eqref{eq:PLPBE}. There are constructive constants $\eps_{1}, \vartheta_1, C_0, \theta >0$ such that for every $\eps\in (0,\eps_{1})$ and every $\vartheta_0 \in (0, \vartheta_1)$ there holds 
\be
\lVert  f_t \rVert_{L^{\infty}_{\omega}(\bar \OO)} \leq  \varpi_\eps e^{-\theta t} \left(   \lVert f_0\rVert_{L^{\infty}_{\omega}(\OO)} +  \underset{s\in[0,t]}{\sup}\left[ e^{\theta s} \lVert G_s \rVert_{L^{\infty}_{\omega\nu^{-1}}(\OO)}\right]\right)  + C_0 \vartheta_0 ,
\label{eq:LinftyEstPertG}
\ee
for every $t\geq 0$, and some $\varpi_\eps>0$ such that $\varpi_\eps  \to \infty$ as $\eps \to 0$. 
\end{prop}

\begin{rem}
Here and during the sequel (at least when working at the level of a priori estimates) we use the notation
\be\label{def:NormBarOOeps}
\lVert  f_t \rVert_{L^{\infty}_{\omega}(\bar \OO)} := \lVert  f_t \rVert_{L^{\infty}_{\omega}(\OO)} + \lVert \gamma f_t \rVert_{L^{\infty}_{\omega}(\Sigma)}, 
\ee
for any admissible weight function $\omega$, and we recall that $L^{\infty}_{\omega}(\Sigma) :=L^{\infty}(\Sigma, \omega(v) \, \dv\d\sigma_x )$.
\end{rem}

\begin{rem}
The computations leading to the proof of Proposition~\ref{prop:LinftyEstPert} use the stretching method developed in \cite{GuoZhou2024} (see also the previous version \cite{GuoZhou18}), and subsequently applied and generalized in \cite{BE_iso} to cylindrical domains. 
\end{rem}

This section is structured as follows: First we perform an useful change of variables and, in this rescaled framework, we repeat the computations from \cite{BE_iso} to provide a $L^2-L^\infty$ type of control. 
Using this and the results from Section~\ref{sec:Hypo_Pert} we prove then Theorem~\ref{prop:LinftyEstPert} following the ideas used during the proof of \cite[Proposition 3.1]{BE_iso}.

\subsection{Transformation of the problem} \label{ssec:LinRes}
We take $t= \eps^2 \tau$, $x=\eps y $, and we introduce the stretched domains $\Omega^{\varepsilon}:=\{\varepsilon^{-1} x, \, x\in\Omega\}$, $\OO^\eps= \Omega^\eps \times \R^3$, and $\UU^\eps := \R_+ \times \OO^\eps$. 

Moreover, we observe that, since for any $y\in \partial \Omega^\eps$ we have that $x = \eps y\in \partial \Omega$, we may define $\delta^\eps:\R^3 \to \R$, $\delta^\eps(y):= \delta(x)$, and we note that the following holds almost everywhere
\be\label{eq:Equality_normal_eps}
n(x) = -{\grad_x \delta (x)\over \lv \grad_x \delta (x)\rv } = -{\grad_x \delta (\eps y)\over \lv \grad_x \delta (\eps y)\rv } = -{\grad_y [\delta (\eps y)]\over \lv \grad_y [\delta (\eps y)] \rv } = -{\grad_y \delta^\eps (y)\over \lv \grad \delta^\eps (y)\rv } =: n(y) = n_y,
\ee
which is nothing but saying that the normal vector on a rescaled point of the boundary set $\partial\Omega^\eps$ coincides with the one of the corresponding point on the original boundary set $\partial\Omega$. 
We define then $\Sigma^\eps:= \partial \Omega^\eps\times \R^3$ and we define accordingly the sets 
$$
\Sigma^\eps_\pm:= \{(y,v)\in \Sigma^\eps, \, \pm n_y \cdot v >0\}, \quad \Sigma_0^\eps:= \{(y,v)\in \Sigma^\eps, \,  n_y \cdot v =0\},
$$
and $\Gamma_\pm^\eps:= (0,\infty) \times \Sigma^\eps_\pm$. 

\smallskip
Let then $f$ be a solution of Equation~\eqref{eq:PLPBE}, we introduce the functions
$$ 
h^\eps (\tau, y, v) :=  f(\eps^2 \tau , \eps y ,v)  = f(t,x,v) , 
$$ 
so that $h^\eps$ satisfies the eqaution
\be
	\left\{\begin{array}{rlll}
		\partial_{\tau} h^\eps &=& -v\cdot \grad_y h^\eps -\nu h^\eps + K h^\eps +   G^\eps &\text{ in }\UU^\eps\\
		\gamma_- h^\eps &=& (1-\iota^\eps) \SSS \gamma_+h^\eps + \iota^\eps  \DDD_{\Theta^\eps} \gamma_+h^\eps  + \iota^\eps \psi^\eps&\text{ on }\Gamma_{-}^\eps\\
		 h^\eps_{\tau=0}(y,v)&=&   f_0 (\eps y, v)  &\text{ in }\OO^\eps,
	\end{array}\right.\label{eq:RPLPBE}
\ee
where we have defined $G^\eps$, $\iota^\eps$, $\Theta^\eps$, and $\psi^\eps$ as follows
$$ 
\phi^\eps = \phi^\eps (\tau, y, v) := \phi(\eps^2 \tau , \eps y ,v)  = \phi(t,x,v).
$$ 
Finally, we translate the geometrical assumptions \ref{item:H1} and \ref{item:H2} into the rescaled setting:
\begin{enumerate}[leftmargin=*]
\item[(RH1)]\label{item:RH1} $\Omega^\eps\subset\R^3$ is an open $C^2$ domain, and $\delta^\eps \in C^{2}(\R^3, \R) \cap W^{3, \infty}(\R^3, \R)$. Moreover, $\iota^\eps\in C(\partial\Omega)$ and such that for every $y\in \partial \Omega^\eps$, $\iota^\eps(y) \in [\iota_0, 1]$ with $\iota_0 \in (0,1]$.

\item[(RH2)]\label{item:RH2} $\Omega^\eps =  (-L^\eps, L^\eps) \times \Omega_0^\eps$, with $L^\eps:= \eps^{-1} L$ and $\Omega_0^\eps:= \eps^{-1} \Omega_0$, i.e. is the 2-dimensional ball of radius $\eps^{-1}\RRRR$ centered at the origin. We also define 
\begin{equation*}
    \Lambda_1^\eps:= \{-L^\eps\}\times\Omega_0^\eps,\quad
    \Lambda_2^\eps:= \{L^\eps\}\times\Omega_0^\eps, \quad
    \Lambda_3^\eps:= (-L^\eps,L^\eps)\times \partial \Omega_0^\eps ,\label{eq:cylinderDefinition}
\end{equation*}
and $\Lambda^\eps:= \Lambda_1^\eps\cup \Lambda_2^\eps\cup \Lambda_3^\eps$. Moreover, we take $\iota= \Ind_{\Lambda_1^\eps \cup \Lambda_2^\eps}$, and we define the rescaled singular set at the boundary
\be\label{def:SSSS_eps}
\begin{aligned}
\SSSS^\eps&:=(\overline{\Lambda_1^\eps} \cap \overline{\Lambda^\eps_3}) \cup (\overline{\Lambda_2^\eps} \cap \overline{\Lambda^\eps_3}).
\end{aligned}
\ee
\end{enumerate}

We devote the rest of this section to prove an equivalent version of Proposition~\ref{prop:LinftyEstPert} for this new framework.
\begin{prop}\label{prop:LinftyEstPertRescaled} 
Assume that either Assumption \ref{item:RH1} or Assumption \ref{item:RH2} holds, let $\omega$ be an admissible weight function, let $G^\eps :\UU^\eps\to \R$ satisfying $\lla G_\tau^\eps \rra_{\OO}=0$ for every $\tau\geq 0$, and let $h^\eps$ be a solution of Equation~\eqref{eq:RPLPBE}. There are constructive constants $\eps_{1}, \vartheta_1, \theta >0$ such that for every $\eps\in (0,\eps_{1})$ and every $\vartheta_0 \in (0, \vartheta_1)$ there holds 
\beqn
\lVert  h_\tau^\eps \rVert_{L^{\infty}_{\omega}(\bar \OO^\eps)} \leq  \varpi_\eps e^{-\theta \eps^{2} \tau } \left(   \lVert h_0^\eps \rVert_{L^{\infty}_{\omega}(\OO^\eps)} +  \underset{s\in[0,\tau]}{\sup}\left[ e^{\theta \eps^2 s} \lVert G_s^\eps \rVert_{L^{\infty}_{\omega\nu^{-1}}(\OO^\eps )}\right]\right)  + C_0 \vartheta_0 ,
\label{eq:LinftyEstPert}
\eeqn
for every $t\geq 0$, and some universal constant $C_0 >0$, and $C_\eps>0$ such that $C_\eps  \to \infty$ as $\eps \to 0$. 
\end{prop}

\subsection{$L^2$ estimate in the rescaled framework}
Following the ideas developed during Section~\ref{sec:Hypo_Pert} we obtain the following $L^2$ estimate in the spirit of Theorem~\ref{theo:L2Decay_pert}.
\begin{theo}\label{theo:L2Decay_pert_rescaled}
Assume that either Assumption \ref{item:RH1} or Assumption \ref{item:RH2} holds, and consider a function $G^\eps:\UU^\eps\to \R$ satisfying $\lla G^\eps_\tau\rra_{\OO}=0$ for every $\tau\geq 0$ . There are constructive constants $\kappa>0$ and $C\geq 1$ such that for any $h^\eps$ solution of Equation~\eqref{eq:RPLPBE} there holds
\begin{multline}
    \lVert h^\eps_\tau \rVert_{ \HH^\eps}  \lesssim_C e^{-\kappa \eps^2 \tau }\lVert h_0^\eps \rVert_{ \HH^\eps} 
    +     \vartheta_0^{1/2}   \eps^{13/2} \left( \int_0^\tau e^{-2\kappa \eps^2  (\tau-s)} \left\lVert  (\iota)^{1/2} (\widetilde{\gamma_+ h^\eps_s} ) \right\rVert^2_{L^2(\partial \Omega^\eps)}  \ds \right)^{1/2} \\
    +\eps^{-1} \left( \int_0^\tau e^{-2\kappa \eps^2 (\tau-s)} \lvv G^\eps_s\rvv_{\HH^\eps}^2 \ds \right)^{1/2} +  \eps^4 \vartheta_0^{1/2}    , \label{eq:HypoRescaled}
\end{multline}
for every $\tau\geq 0$. Furthermore, there is a norm $\lvvv \cdot \rvvv_\eps$ equivalent to the usual norm of $\HH^\eps$ uniformly in $\eps$, i.e. there is a constant $c>0$ independent of $\eps$ such that
\be\label{eq:ClassicHypoEquivalenceRescaled}
c^{-1} \lvv h \rvv_{\HH^\eps} \leq \lvvv h\rvvv_\eps \leq c \lvv h \rvv_{\HH^\eps},
\ee
for which there holds
\begin{multline}
 \lvvv h^\eps_\tau  \rvvv_\eps \leq e^{-\kappa t}\lvvv h^\eps_0 \rvvv_\eps
+  C^\star\,   \vartheta_0^{1/2} \eps^{13/2} \left(  \int_0^\tau e^{-2\kappa \eps^2  (\tau-s)} \left\lVert  (\iota)^{1/2} (\widetilde{\gamma_+ h^\eps_s} ) \right\rVert^2_{L^2(\partial \Omega^\eps )} \ds \right)^{1/2} \\
    +\eps^{-1} C^\star\left( \int_0^\tau e^{-2\kappa \eps^2 (\tau-s)} \lvv G^\eps_s\rvv_{\HH^\eps}^2 \ds \right)^{1/2}  + \vartheta_0^{1/2}  \, \eps^4 \, C^\star ,  \label{eq:HypoEquivRescaled}
\end{multline}
for every $\tau\geq 0$, and some constant $C^\star>0$.
\end{theo}

\begin{proof}
The proof follows the exact same ideas and computations as those leading to the proof of Theorem~\ref{theo:L2Decay_pert}, and using the Young inequality to control the terms involving $G^\eps$, we thus skip it.
\end{proof}

\subsection{Auxiliary problem in finite time and backwards trajectories}\label{ssec:Trajectories}
Here and below, during this section, we consider an arbitrary $T>0$, we take $G^\eps$ satisfying $\lla G_\tau^\eps \rra_{\OO^\eps}=0$ for every $\tau\geq 0$, and we study the following evolution equation
\be
	\left\{\begin{array}{rlll}
		\partial_{\eps} h^\eps &=&  -v\cdot \grad_y h^\eps -\nu h^\eps + K h^\eps +  G^\eps &\text{ in }\UU^\eps_T:= (0,T]\times \OO^\eps\\
		\gamma_- h^\eps &=& (1-\iota^\eps) \SSS \gamma_+h^\eps + \iota^\eps  \DDD_{\Theta^\eps} \gamma_+h^\eps  + \iota^\eps \psi^\eps&\text{ on } \Gamma_{-, T}^\eps:= (0,T] \times \Sigma_-^\eps\\
		 h^\eps_{\tau=0}(y,v)&=& h_0^\eps (y,v) :=  f_0 (\eps y, v)  &\text{ in }\OO^\eps.
	\end{array}\right.\label{eq:RPLPBE_T}
\ee	 
We observe now that the characteristics of Equation~\eqref{eq:RPLPBE_T} are given by 
\be
 Y (s; \tau,y,v):= y - v(\tau - s)  \quad \text{ and } \quad V(s; \tau,y,v):= v. \label{eq:carachteristics}
 \ee
Therefore, for any fixed particle with coordinates in $ \UU^\eps$, we can characterize the coordinates of the last collision against the boundary of $\Omega^\eps$. Indeed, let $(\tau_0,y_0,v_0) \in \UU^\eps$ be the coordinates of a particle, we define the time of collision along this trajectory ($\tau_b$), the time of life of the particle prior to such collision ($\tau_1$), and the position $(y_1)$ and velocity ($v_1$) at the boundary during this collision, as follows
\begin{equation}\label{eq:BackwardsTrajectory}
\begin{array}{rcl}
    \tau_b(y_0,v_0) &=& \inf\{s >0;\ Y(-s, 0, y, v)\notin\Omega^\eps\},\\
    \tau_1(\tau_0,y_0,v_0)&=&\tau_0-\tau_b(y_0,v_0),\\
    y_1(\tau_0, y_0,v_0)&=& Y(\tau_1; \tau, y, v) = y -v(t\tau - \tau_1),\\
    v_1(\tau_0, y_0,v_0)&=&\left\{\begin{array}{cl}
         \VV_{y_1}(v_0)& \text{ during specular reflection,} \\
         v_0^*& \text{ during diffuse reflection,}
    \end{array}\right.
\end{array}
\end{equation}
where $v_0^*$ stands for an independent variable.
Furthermore, we define the singular set along these trajectories as
\be\label{def:Sx}
 S_{x}:=\{v\in\R^3;\ \exists \, \tau \in [0,T], \, n(Y(\tau_b(y,v),\tau,y,v))\cdot v=0\}.
\ee

Following then the ideas from \cite[Proposition 3.13]{BE_iso} for smooth domains, and \cite[Proposition 4.11]{BE_iso} for cylindrical domains, we have the following result.

\begin{prop}\label{prop:LinftyEst_rescaled} 
Assume that either Assumption \ref{item:RH1} or Assumption \ref{item:RH2} holds, let $\omega$ be an admissible weight function, and let $h^\eps$ be a solution of Equation~\eqref{eq:RPLPBE_T}. There is $\eps_2 = \eps_2 (T)>0$ such that for every $\eps\in (0,\eps_2)$ there holds 
\begin{multline*}
\lVert  h^\eps_\tau  \rVert_{L^{\infty}_{\omega}(\bar\OO^\eps)} \lesssim_C \la T\ra^p e^{-\nu_0 t} \lVert h^\eps _0\rVert_{L^{\infty}_{\omega}(\OO^\eps)} 
+ T \la T\ra^{p-1} e^{-\nu_0  t }\underset{s\in[0, \tau]}{\sup}\left[ e^{\nu_0 s}\lVert  h^\eps_s \rVert_{\HH^\eps} \right]  \\
+  \la T\ra^p  e^{-\nu_2t}\underset{s\in [0,t]}{\sup}\left[e^{\nu_2 s}\lVert G^\eps_s\rVert_{L^{\infty}_{\omega\nu^{-1}}(\OO^\eps)}\right] + \eps^{12} \, \vartheta_0 .
\end{multline*}
for every $\tau\in [0, T]$, every $\nu_2\in (0,\nu_0)$, some universal constants $C, p>1$, and where we have defined $\HH^\eps := L^2_{\MM^{-1/2}}(\OO^\eps)$. 
\end{prop}

\begin{proof}
The proof follows exactly the arguments and computations from \cite[Proposition 3.13]{BE_iso} for smooth domains, and those of \cite[Proposition 4.11]{BE_iso} for cylindrical domains, together with the use of \eqref{eq:Controlpsi} to control the inflow term $\psi$ at the boundary, thus we skip it. 
\end{proof}

\subsection{Proof of Proposition~\ref{prop:LinftyEstPertRescaled}}
The proof uses the results from Theorem~\ref{theo:L2Decay_pert_rescaled} and Proposition~\ref{prop:LinftyEst_rescaled}, and it follows the ideas and computations from the proof of \cite[Proposition 3.1]{BE_iso}.

\medskip\noindent
\emph{Step 1. (Choice of parameters and a priori estimates)} We choose $T>0$ large enough such that 
$$
C_1 \la T\ra^p e^{-T\nu_0/2 } \leq \frac 12,
$$ 
where $C_1>0$ is given by Proposition~\ref{prop:LinftyEst_rescaled},
and we set $\eps_1^1 = \min(\eps_2(T), 1/2, \sqrt{\nu_0/\kappa})$, where $\eps_2(T)>0$ is given by Proposition~\ref{prop:LinftyEst_rescaled} and we recall that $\kappa>0$ is given by Theorem~\ref{theo:L2Decay_pert_rescaled}.

\smallskip
We recall the hypocoercivity norm $\lvvv \cdot \rvvv_\eps$ given by Theorem~\ref{theo:L2Decay_pert_rescaled}, the equivalency relation \eqref{eq:ClassicHypoEquivalenceRescaled}, and we note that there holds
\be\label{eq:ControlHHinfty} 
\lvv G^\eps\rvv_{\HH^\eps} \lesssim \left( \int_{\OO^\eps} \dx \right)^{1/2} \lvv G^\eps \rvv_{L^\infty_{\omega\nu^{-1}} (\OO^\eps)} \lesssim \eps^{-3/2} \lvv G^\eps \rvv_{L^\infty_{\omega\nu^{-1}} (\OO^\eps)} ,
\ee
and 
\be\label{eq:ControlBoundaryInfty}
\left\lVert  (\iota)^{1/2} (\widetilde{\gamma_+ h^\eps} ) \right\rVert^2_{L^2(\partial \Omega^\eps )} \lesssim \eps^{-3} \lvv h^\eps \rvv_{L^\infty_\omega( \bar \UU_\tau^\eps)}^2,
\ee
thus we deduce from \eqref{eq:HypoEquivRescaled} that there is a constant $C_2>0$ such that
\begin{multline}
 \lvvv h^\eps_\tau  \rvvv_\eps \leq e^{-\kappa \eps^2 \tau}\lvvv h^\eps_0 \rvvv_\eps
+  C_2\,   \vartheta_0^{1/2} \, \eps^5\, \tau^{1/2} e^{- \kappa \eps^2 \tau}\underset{s\in [0,\tau]}{\sup}\left[e^{ \kappa \eps^2 s}  \lvv h^\eps_s\rvv_{L^\infty(\bar \OO^\eps)}\right] \\
    +\eps^{-5/2} \tau^{1/2}\,   C_2\,  e^{- \kappa \eps^2 \tau}\underset{s\in [0,\tau]}{\sup}\left[e^{ \kappa \eps^2 s}\lVert G^\eps_s\rVert_{L^{\infty}_{\omega \nu^{-1}}(\OO^\eps)}\right]  + \eps^4 \, \vartheta_0^{1/2}   C_2 ,\label{eq:HypoPerturbed}
\end{multline}
where $\kappa>0$ is given by Theorem~\ref{theo:L2Decay_pert_rescaled}.

\smallskip
Putting then together the estimate given by Proposition~\ref{prop:LinftyEst_rescaled} and \eqref{eq:HypoPerturbed}, we have that for every $\eps \in (0,\eps_1^1)$, every $\nu_2\in (0,\nu_0)$, and $\kappa_0 \in (0,\kappa^\star)$ there holds 
$$\beal
\lVert  h^\eps_\tau \rVert_{L^{\infty}_{\omega}(\bar \OO^\eps)} \leq& C_1 C_2\,  \eps^5\,  \vartheta_0^{1/2} \, T^{1/2} e^{- \kappa \eps^2 \tau}\underset{s\in [0,\tau]}{\sup}\left[e^{ \kappa \eps^2 s}  \lvv h^\eps_s\rvv_{L^\infty_\omega(\bar \OO^\eps)}\right] 
+ C_1 \la T\ra^p e^{-\nu_0 \tau} \lVert h^\eps_0\rVert_{L^{\infty}_{\omega}( \OO^\eps)}    \\
&+ C_1 \la T\ra^p e^{-\kappa^\star \eps^2 \tau} \lvvv h^\eps_0 \rvvv_\eps +  C_1 \la T\ra^p  e^{-\nu_2 \tau} \underset{s\in [0,\tau]}{\sup}\left[e^{\nu_2 s}\lVert G^\eps_s\rVert_{L^{\infty}_{\omega\nu^{-1}}(\OO^\eps)}\right] \\
&+ C_1C_2 \, \la T\ra^{p+1/2} \, \eps^{-5/2} \, e^{- \kappa_0 \eps^2\tau}\underset{s\in [0,\tau]}{\sup}\left[e^{ \kappa_0 \eps^2 s}\lVert G^\eps_s\rVert_{L^{\infty}_{\omega\nu^{-1}}(\OO^\eps)}\right] + \vartheta_0^{1/2} \eps^4  (C_1+ C_2),
\eeal
$$
where we have used the fact that $\nu_0-\kappa\eps^2 \geq 0$ due to our choice of $\eps_1^1$.

We then set $\theta = \min(\nu_0, \kappa)/8$, $\nu_2= \eps^2\theta$ and $\vartheta_1^1 = (2C_1C_2 T^{1/2})^{-2}$, and we note that these choices are possible since $\eps\leq 1/2$ and $\theta\leq \nu_0/2$, thus $\eps^2\theta\in (0,\nu_0/2)$ and to the fact that $\theta \in (0, \kappa/2)$. 
Absorbing the small contributions on the above estimate we have that
 \begin{multline}\label{eq:LargeTimeTEstimateZeta1/4}
\lVert  h^\eps_T \rVert_{L^{\infty}_{\omega}(\bar \OO^\eps)} \leq  e^{-2 \theta \eps^2 T} \lVert h^\eps_0\rVert_{L^{\infty}_{\omega}(\OO^\eps)} + C_T e^{-2\theta \eps^2 T} \lvvv h^\eps_0\rvvv_\eps  \\
+\eps^{-5/2}  C_T e^{-\theta \eps^2 T} \underset{s\in [0,T]}{\sup}\left[e^{\theta\eps^2 s}\lVert G^\eps_s\rVert_{L^{\infty}_{\omega\nu^{-1}}(\OO^\eps)}\right]  + \vartheta_0^{1/2}  \eps^4 \, C_3,
\end{multline}
and 
\begin{multline}\label{eq:SmallTime_t_EstimateZeta1/4}
\lVert  h^\eps_\tau  \rVert_{L^{\infty}_{\omega}(\bar\OO^\eps)} \leq  C_T  e^{-2 \theta \eps^2 \tau}  \lVert h^\eps_0\rVert_{L^{\infty}_{\omega}(\OO^\eps)} + C_T e^{-2 \theta \eps^2 \tau}  \lvvv h^\eps_0\rvvv_\eps \\
+ \eps^{-5/2}  C_T e^{- \theta \eps^2 \tau} \underset{s\in [0,\tau]}{\sup}\left[e^{  \theta \eps^2 s}\lVert G^\eps_s\rVert_{L^{\infty}_{\omega\nu^{-1}}(\OO^\eps)}\right] + \vartheta_0^{1/2}  \eps^4 \, C_3,
\end{multline}
for some constants $C_T, C_3>0$ and where \eqref{eq:SmallTime_t_EstimateZeta1/4} holds for all $\tau\in [0,T]$. 
Moreover, we set $\vartheta_1^2>0$ such that 
$$
C_T C_2 T^{1/2} (\vartheta_1^2)^{1/2} \eps^5 \leq 1-e^{-\theta \eps^2 T},
$$
uniformly in $\eps \ll 1$. 
We thus deduce that, for every $\vartheta_0\in (0, \min(\vartheta_1^1, \vartheta^1_2))$, and every $\eps\in (0, \eps_1^1)$, combining \eqref{eq:LargeTimeTEstimateZeta1/4} and \eqref{eq:SmallTime_t_EstimateZeta1/4} with \eqref{eq:HypoPerturbed} we further have that
\be
 \lvvv h^\eps_T  \rvvv_\eps \leq e^{- 2\theta \eps^2 T}\lvvv h^\eps_0 \rvvv_\eps
    +\eps^{-5/2}  \tilde C_T \,  e^{- \theta \eps^2 T}\underset{s\in [0,T]}{\sup}\left[e^{ \theta \eps^2 s}\lVert G^\eps_s\rVert_{L^{\infty}_{\omega \nu^{-1}}(\OO^\eps)}\right]  + \eps^4 \, \vartheta_0^{1/2}   C_4 ,\label{eq:HypoPerturbed_Linfty1}
\ee
and 
\be
 \lvvv h^\eps_\tau  \rvvv_\eps \leq \tilde C_T e^{-2\theta \eps^2 \tau }\lvvv h^\eps_0 \rvvv_\eps
    +\eps^{-5/2}  \tilde C_T \,  e^{- \theta \eps^2 \tau}\underset{s\in [0,\tau]}{\sup}\left[e^{ \theta \eps^2 s}\lVert G^\eps_s\rVert_{L^{\infty}_{\omega \nu^{-1}}(\OO^\eps)}\right]  + \eps^4 \, \vartheta_0^{1/2}   C_4 ,\label{eq:HypoPerturbed_Linfty2}
\ee
fo some constants $\tilde C_T>0$ and $C_4>0$, and where \eqref{eq:HypoPerturbed_Linfty2} holds for all $\tau\in [0,T]$.

\medskip\noindent
\emph{Step 2. (Decay estimate)} We now set 
$$
X_\tau := e^{\theta \eps^2 \tau} \lVert  h^\eps_\tau \rVert_{L^{\infty}_{\omega}(\bar \OO^\eps)} , \quad  Y_\tau:=e^{\theta \eps^2 \tau}   \lvvv  h^\eps_\tau \rvvv_\eps,  \quad  \text{ and } \quad \Phi_{\tau_0, \tau_1}:= \underset{s\in [\tau_0,\tau_1]}{\sup}\left[e^{ \theta \eps^2 s}\lVert G^\eps_s\rVert_{L^{\infty}_{\omega\nu^{-1}}(\OO^\eps)}\right].
$$
Translating \eqref{eq:LargeTimeTEstimateZeta1/4} and \eqref{eq:HypoPerturbed_Linfty1} into this new notations we have that
\be\label{eq:Est_XT}
X_T \leq  e^{- \theta \eps^2 T} X_0 + C_T e^{- \theta \eps^2 T} Y_0 +  \eps^{-5/2}  C_T  \Phi_{0,T} + \eps^4 \, \vartheta_0^{1/2} C_3 e^{\theta\eps^2 T},
\ee
and
\be\label{eq:Est_YT}
Y_T \leq e^{- \theta \eps^2 T}  Y_0+ \tilde C_T \, \eps^{-5/2}  T\,  \Phi_{0,T} + \eps^4 \, \vartheta_0^{1/2} C_4 e^{\theta\eps^2 T},
\ee
respectively. Similarly, \eqref{eq:SmallTime_t_EstimateZeta1/4} and \eqref{eq:HypoPerturbed_Linfty2} translates into
\be\label{eq:Est_Xtau}
X_\tau \leq  C_T e^{- \theta \eps^2 \tau} X_0 + C_T e^{- \theta \eps^2 \tau} Y_0 +  \eps^{-5/2}  C_T  \Phi_{0,\tau} + \eps^4\, \vartheta_0^{1/2} C_3 e^{\theta\eps^2 T},
\ee
and
\be\label{eq:Est_Ytau}
Y_\tau \leq e^{- \theta \eps^2 \tau}  Y_0+ \tilde C_T \, \eps^{-5/2}  T\,  \Phi_{0,\tau} + \eps^4 \, \vartheta_0^{1/2} C_4 e^{\theta\eps^2 T},
\ee
respectively.
We now define $\eps_1 = \min(\eps_1^1, (\theta T)^{-1} \log 2)$ so that $e^{\theta\eps^2 T} -1\leq 1$, and we introduce a constant $\beta>0$ defined by 
\be\label{def:betaXY}
\beta =  {1\over 2C_T} \left( e^{\theta\eps^2 T} -1\right),
\ee
so that, due to our choice of $\eps_1$, there holds $\beta C_T\leq 1/2$ and simultaneously
\be\label{def:deltaXY}
\delta :=e^{- \theta \eps^2 T}(1+\beta C_T) =  \frac 12 \left( 1+ e^{-\theta\eps^2T} \right)  < 1, \quad \forall \eps >0.
\ee

Using then \eqref{eq:Est_XT} and \eqref{eq:Est_YT}, we have that
\bear
Z_T:= Y_T + \beta X_T &\leq& (1+ \beta C_T) e^{- \theta \eps^2 T} ( Y_0 + \beta X_0) + \eps^{-5/2} (\beta C_T + \tilde C_T) \Phi_{0,T} + \eps^4 \, \vartheta_0^{1/2} (\beta C_3+ C_4) e^{\theta\eps^2 T} \nonumber \\
&\leq& \delta Z_0 + \eps^{-5/2} \left( \frac 12 + C_2 T\right) \Phi_{0,T} + \eps^4 \, \vartheta_0^{1/2} (\beta C_3+ C_4) e^{\theta\eps^2 T}\label{eq:LargeTimeTEstimateXY},
\eear
where we have used the very definition of $\beta$ to deduce the second line. On the other hand, using \eqref{eq:Est_Xtau}, \eqref{eq:Est_Ytau}, and again our choice of $\beta$, in a similar way as above we further have that
\be\label{eq:SmallTime_t_EstimateXY}
Z_\tau  \leq \left(\frac 12 + \tilde C_T\right)  e^{- \theta \eps^2 \tau} Z_0 + \eps^{-5/2}  \left( \frac 12 + \tilde C_T\right)  \Phi_{0, \tau} + \eps^4 \, \vartheta_0^{1/2} (\beta C_3+ C_4) e^{\theta\eps^2 T} ,
\ee
for every $\tau \in [0,T]$.
Then for any $\bar \tau \in \R$ there is $n\in \N$ such that $\bar \tau\in [nT, (n+1)T)$ and iterating first \eqref{eq:LargeTimeTEstimateXY} we have
\bean
Z_{\bar \tau} &\leq& \delta^n  Z_{\bar \tau - nT} +  \left[ \eps^{-5/2}  \left( \frac 12 + \tilde C_T\right) \Phi_{\bar \tau-nT, \bar \tau}  + \eps^4 \, \vartheta_0^{1/2} (\beta C_3+ C_4) e^{\theta\eps^2 T} \right] \left( \sum_{k=0}^{n-1} \delta^k\right) \\
&\leq & \left( \frac 12 + \tilde C_T\right) \vartheta^n Z_{0} +  \eps^{-5/2}  \left( \frac 12 + \tilde C_T\right) \left( \sum_{k=0}^{n-1} \delta^k\right) \Phi_{0, \bar \tau}   +\eps^4 \, \vartheta_0^{1/2} (\beta C_3+ C_4) e^{\theta\eps^2 T} \left( \sum_{k=0}^{n-1} \delta^k\right) 
\eean
where we have used \eqref{eq:SmallTime_t_EstimateXY} on the second line. 
Using then the above estimate we deduce that there is a constant $\tilde C >0$, independent of $\eps$, such that 
$$
\beta \, e^{\theta \eps^2 t} \lvv f_\tau \rvv_{L^{\infty}_{\omega}(\bar \OO^\eps)} \leq Z_\tau \leq \tilde C Z_0 
+  {\eps^{-5/2}  \over 1- e^{-\theta \eps^2 T}}  \tilde C \underset{s\in [0,\tau ]}{\sup}\left[e^{  \theta \eps^2 s}\lVert G_s\rVert_{L^{\infty}_{\omega\nu^{-1}}(\OO^\eps)}\right] 
+ \tilde C {\eps^4 \over 1- e^{-\theta \eps^2 T} } \, \vartheta_0^{1/2}  ,
$$
for every $\tau >0$.
Finally, we observe that \eqref{eq:ClassicHypoEquivalence} and \eqref{eq:ControlHHinfty} imply together that
$$
Z_0 \lesssim \lvv f_0 \rvv_{L^{\infty}_{\omega}(\OO^\eps)} +  \lvv f_0\rvv_{\HH} \lesssim  \eps^{-3/2} \lvv f_0 \rvv_{L^{\infty}_{\omega}(\OO^\eps)},
$$
Altogether, using the very definition of $\beta$ we deduce that 
\begin{multline*}
 e^{\theta \eps^2 t} \lvv h^\eps_\tau \rvv_{L^{\infty}_{\omega}(\bar \OO^\eps)} \leq { \eps^{-3/2} C'\over e^{\theta\eps^2 T}-1} \lvv h^\eps_0 \rvv_{L^{\infty}_{\omega}(\OO^\eps)} +  {\eps^{-5/2} C' \over \left( e^{\theta\eps^2T} -1 \right) \left( 1-e^{-\theta \eps^2 T} \right) } \underset{s\in [0,t]}{\sup}\left[e^{  \theta \eps^2 s}\lVert G^\eps_s\rVert_{L^{\infty}_{\omega\nu^{-1}}(\OO^\eps)}\right] \\
 + C'{\eps^4 \over \left( e^{\theta\eps^2T} -1 \right) \left( 1-e^{-\theta \eps^2 T} \right) } \vartheta_0^{1/2},
\end{multline*}
for some constant $C'>0$ independent of $\eps$. We conclude by observing that 
$$
{\eps^4 \over \left( e^{\theta\eps^2T} -1 \right) \left( 1-e^{-\theta \eps^2 T} \right) } \lesssim 1,
$$
uniformly in $\eps$, and by taking 
\be\label{def:varpi_eps}
\varpi_\eps : =  {\eps^{-5/2} C' \over \left( e^{\theta\eps^2T} -1 \right) \left( 1-e^{-\theta \eps^2 T} \right) },
\ee
and noting that $\varpi_\eps \to \infty$ as $\eps\to 0$. 
\qed

\section{A priori weighted $L^\infty$ estimates for the steady problem}\label{sec:AprioriLinfty_SS}

In this section we establish a result analogue to Proposition~\ref{prop:LinftyEstPert} for the solutions of Equation~\eqref{eq:PLBE_SS}, steady problem associated to Equation~\eqref{eq:PLPBE}.

\begin{prop}\label{prop:LinftyEstimatePerturbedFiniteT_Summary_SS} 
Assume that either Assumption \ref{item:H1} or Assumption \ref{item:H2} holds, let $\omega$ be an admissible weight function, consider $G:\OO\to \R$ satisfying $\lla G\rra_{\OO}=0$, and let $\FF$ be a solution of Equation~\eqref{eq:PLBE_SS} satisfying $\lla \FF\rra_\OO = 0$. There are constructive constants $\eps_{1},\vartheta_1 >0$, given by Proposition~\ref{prop:LinftyEstPert}, such that for every $\eps\in (0,\eps_1)$ and every $\vartheta_0 \in (0, \vartheta_1)$ there holds 
\beqn
\lvv \FF \rvv_{L^\infty_\omega(\bar \OO)}  \leq  \varpi_\eps  \lVert G \rVert_{L^{\infty}_{\omega\nu^{-1}}(\OO)}  + C_0 \vartheta_0 ,
\label{eq:LinftyPerturbedFiniteTimeDecay_Summary_SS}
\eeqn
where $\varpi_\eps, C_0>0$ are given by Proposition~\ref{prop:LinftyEstPert}. 
\end{prop}

\begin{proof}
We define $g(t,x,v) := \FF(x,v)$ for every $t\geq 0$, and we note that, defined this way, let $G$ is a solution of Equation~\eqref{eq:PLPBE} with $g_0 (x, v)= \FF (x,v)$. 

We then procede with the change of variables performed in Subsection~\ref{ssec:LinRes}, and we define $g^\eps (\tau, y, v) := g(\eps^2 \tau, \eps y, v) = g(t,x,v)$, so that $g^\eps$ is a solution of Equation~\eqref{eq:RPLPBE}.

Repeating then the computations performed during the proof of the Proposition~\ref{prop:LinftyEstPertRescaled} we have have the existence of $T>0$ such that for every $\eps \in (0,\eps_1)$, and every $\vartheta_0 \in (0,  \vartheta_1)$, there holds
 \begin{multline}\label{eq:LargeTimeTEstimateZeta1/4}
\lVert  g^\eps_T \rVert_{L^{\infty}_{\omega}(\bar \OO^\eps)} \leq  e^{-2 \theta \eps^2 T} \lVert g^\eps_0\rVert_{L^{\infty}_{\omega}(\OO^\eps)} + C_T e^{-2\theta \eps^2 T} \lvvv g^\eps_0\rvvv_\eps  \\
+\eps^{-5/2}  C_T e^{-\theta \eps^2 T} \underset{s\in [0,T]}{\sup}\left[e^{\theta\eps^2 s}\lVert G^\eps\rVert_{L^{\infty}_{\omega\nu^{-1}}(\OO^\eps)}\right]  + \vartheta_0^{1/2}  \eps^4 \, C_0,
\end{multline}
and
\be
 \lvvv g^\eps_T  \rvvv_\eps \leq e^{- 2\theta \eps^2 T}\lvvv g^\eps_0 \rvvv_\eps
    +\eps^{-5/2}   \tilde C_T \,  e^{- \theta \eps^2 T}\underset{s\in [0,T]}{\sup}\left[e^{ \theta \eps^2 s}\lVert G^\eps\rVert_{L^{\infty}_{\omega \nu^{-1}}(\OO^\eps)}\right]  + \eps^4 \, \vartheta_0^{1/2}   C_1 ,\label{eq:HypoPerturbed_Linfty1}
\ee
for some contants $C_T, \tilde C_T, C_0, C_1 >0$, and we recall that the hypocoercivity norm $\lvvv \cdot \rvvv_\eps$ is given by Theorem~\ref{theo:L2Decay_pert_rescaled}. 
We then consider $\beta>0$, and $\delta\in (0,1)$, as introduced in \eqref{def:betaXY} and \eqref{def:deltaXY} respectively, and, repeating the computations leading to \eqref{eq:LargeTimeTEstimateXY}, we deduce that
\begin{multline*}
e^{\theta \eps^2 T}   \lvvv  g^\eps_T \rvvv_\eps +  \beta e^{\theta \eps^2 T} \lVert  g^\eps_T \rVert_{L^{\infty}_{\omega}(\bar \OO^\eps)}  
\leq \delta  \left( \lvvv  g_0 \rvvv_\eps +  \beta  \lVert  g_0 \rVert_{L^{\infty}_{\omega}(\bar \OO^\eps)}  \right) \\
\quad \qquad + \eps^{-5/2} \left( \frac 12 + C_2 T\right) \underset{s\in [0,T]}{\sup}\left[e^{ \theta \eps^2 s}\lVert G^\eps\rVert_{L^{\infty}_{\omega \nu^{-1}}(\OO^\eps)}\right] + \eps^4 \, \vartheta_0^{1/2} (\beta C_0+ C_1) e^{\theta\eps^2 T}.
\end{multline*} 
We conclude by absorbing the small contributions on the above inequality, arguing as in the conclusion of Proposition~\ref{prop:LinftyEstPertRescaled}, and coming back to our original variables.
\end{proof}

\section{On the well-posedness of linear kinetic equations in non-isothermal domains}\label{sec:WellPosedness}

In this section we develop a well-posedness theory for evolution and steady linear kinetic equations.

In Subsection~\ref{ssec:WellPosedness_Linfty_ss} we establish the well-posedness of the linear Equation~\eqref{eq:PLBE_SS}, as well as estimates for its solutions in the spirit of Proposition~\ref{prop:LinftyEstPert}.

\subsection{Well-posedness of linear kinetic equations in a weighted $L^\infty$ framework}\label{ssec:WellPosedness_Linfty}
  
In this section we prove the well-posedness of Equation~\eqref{eq:PLPBE} in a weighted $L^\infty$ framework following the main ideas from \cite[Theorem 6.2]{BE_iso}.

  \begin{theo}\label{theo:ExistenceL_infty_linear_transport}
Assume that either Assumption \ref{item:H1} or Assumption \ref{item:H2} holds, let $\omega$ be an admissible weight function, let $f_0 \in L^\infty_{\omega}(\OO)$, and let $G\in L^\infty_{\omega\nu^{-1}}(\UU)$. Also, recall that the parameters $\eps_1, \vartheta_1>0$ are given by Proposition~\ref{prop:LinftyEstPert}.

For every $\eps\in (0,\eps_{1})$ and every $\vartheta_0\in (0,\vartheta_1)$, there is $f\in L^\infty_{\omega} (\UU)$, with an associated trace function $\gamma f\in L^\infty_{\omega}(\Gamma)$, unique global weak solution to Equation~\eqref{eq:PLPBE}, i.e for any $\varphi \in \DD(\bar \UU^\eps)$ there holds 
\begin{multline}\label{eq:renormalizedFormulation_evolLinear}
\int_{\OO} f(t,\cdot) \, \varphi(t,\cdot) \, \dx\dv - \int_0^t \int_{\OO} \eps^{-2} (K f) \, \varphi  + f \left( \eps^{-1} v\cdot \grad_x \varphi - \eps^{-2}\nu \, \varphi \right)  \,  \dv\dx\ds  \\
+ \eps^{-1} \int_0^t \int_{\Sigma_+} \gamma_+ f\,  \varphi\,  (n_x\cdot v)_+ \, \dv\d\sigma_x\ds  - \eps^{-1} \int_0^t \int_{\Sigma_-} \RRR_\Theta \gamma_+ f\,  \varphi\,  (n_x\cdot v)_- \, \dv\d\sigma_x\ds  \\
= \int_{\OO} f_0(\cdot ) \varphi(0,\cdot ) \dx\dv + \eps^{-2} \int_0^t \int_{\OO} G\, \varphi \, \dv\dx +  \eps^{-1} \int_0^t \int_{\Sigma_-} \iota \, \psi \,   \varphi\,  (n_x\cdot v)_- \, \dv\d\sigma_x \ds.
\end{multline}
Furthermore, the results from Proposition~\ref{prop:LinftyEstPert} hold.
\end{theo}

  \begin{proof}
  We split the proof into two steps. 

\medskip\noindent
\emph{Step 1.} We set $f^0=0$, $\alpha \in (0,1)$, and we consider the recurrent sequence of solutions given by the following evolution equation
\begin{equation}
	\left\{\begin{array}{llll}
		 \partial_t f^{k+1} &=&  - \eps^{-1} v\cdot \grad_x f^{k+1} - \eps^{-2} \nu f^{k+1} + \eps^{-2} K f^k + \eps^{-2} G  &\text{ in }\UU\\
		\gamma_-f^{k+1}&=& \alpha\,  \RRR_\Theta \gamma_+ f^k  +\iota\psi &\text{ on }\Gamma_{-}\\
		f^{k+1}_{t=0}&=& f_0 &\text{ in }\OO.
	\end{array}\right.\label{eq:TransportLinftyK}
\end{equation}
Indeed, if we assume that, for some $k$, $f^k\in L^\infty_{\omega}(\UU)$, then \cite[Lemma 3.6 - (K1)]{BE_iso} implies that $K f^k \in L^\infty_{\omega} (\UU)$. Moreover, we observe that is $\gamma f^k \in L^\infty_\omega(\Gamma)$, then a direct computation yields
$$
\lvv \RRR_\Theta \gamma_+ f^k + \iota\psi\rvv_{L^\infty_\omega( \Gamma_-)} \lesssim \lvv \gamma_+ f^k\rvv_{L^\infty_\omega(\Gamma_+)} + 1 <\infty.
$$
Therefore, Step 2 of the proof of \cite[Theorem 6.2]{BE_iso} implies the existence of $f^{k+1} \in L^\infty_{\omega}(\UU)$, with a trace $\gamma f^{k+1} \in L^\infty_{\omega} (\Gamma)$, unique weak solution of Equation~\eqref{eq:TransportLinftyK} in the distributional sense, i.e. there olds the weak formulation
\begin{multline*}
\int_{\OO} f^{k+1}_t \, \varphi(t, \cdot) - \int_0^t \int_{\OO} f^{k+1} \, \left( \eps^{-1}v\cdot \grad_x \varphi - \eps^{-2} \nu(v) \varphi \right)   
+ \int_0^t \int_{\Sigma_+} \eps^{-1} \gamma_+ f^{k+1} \, \varphi \,  (n_x \cdot v)_+  \\
= \int_{\OO} f_0 \, \varphi(0,\cdot) \dv\dx 
+ \int_0^t \int_{\Sigma_-} \eps^{-1}  \left(  \RRR_\Theta \gamma_+ f^k + \iota\psi\right) \, \varphi \, (n_x \cdot v)_-  
+ \int_0^t \int_{\OO} \eps^{-2} (Kf^k+G) \, \varphi  ,
\end{multline*}
for every $\varphi \in \DD(\bar \UU^\eps)$ and $t\in \R_+$. 
Furthermore, we have that $f^{k+1}$ is a mild solution of Equation~\eqref{eq:TransportLinftyK}, i.e. there holds the representation formula
\begin{multline*}
f^{k+1}(t,x,v) = e^{-\nu(v) t/\eps^2 } f_0(x-vt, v) \, \Ind_{t_1\leq 0} + \frac 1{\eps^2} \int_{\max(0, t_1)}^t  e^{-\nu(v) (t-s)/\eps^2} \, (Kf^k+G)(s,X^\eps_s ,v)\, \ds \\
+ e^{- \nu(v)(t-t_1)/\eps^2} \RRR_\Theta \gamma_+f^k(t_1, x_1, v) \, \Ind_{t_1>0} + e^{-\nu(v)(t-t_1)/\eps^2} \psi(x_1, v) \, \Ind_{t_1>0}, 
\end{multline*}
where $X_{s}^\eps = X^\eps_s( t, x, v):=x -\eps^{-1} v(t - s)$, and we have defined, in the spirit of \eqref{eq:BackwardsTrajectory}, $(t_1, x_1, v_1)$ as the coordinates of the last backwards collision starting from the point $(t,x,v)$, given by 
\begin{equation}\label{eq:BackwardsTrajectory_original}
\begin{array}{rcl}
    t_b(x,v) &=& \inf\{s >0;\ x- \eps^{-1}vs \notin\Omega^\eps\},\\[.1cm]
    t_1 = t_1(t,x,v)&=&t_0-t_b(x,v),\\[.1cm]
    x_1 = x_1(t, x,v)&=& X_{t_1}^\eps = x -\eps^{-1} v(t - t_1),\\[.1cm]
    v_1(t, x,v)&=&\left\{\begin{array}{cl}
         \VV_{x_1}(v)& \text{ during specular reflection,} \\
         v^*& \text{ during diffusive reflection,}
    \end{array}\right.
\end{array}
\end{equation}
where $v^*$ stands for an independent variable.

\medskip
We take now $A>0$ to be fixed later and, in the spirit of \eqref{def:muA_Hypo_decay}, we define the modified weight functions
$$
\omega_A = \omega_A (v):=  \MMM^{-1}_\Theta \chi_A(v) + (1-\chi_A(v)) \omega ,
$$
where $\chi_A(v):= \chi(|v|/A)$, for a function $\chi \in C^2(\R_+, \R_+)$, such that $\mathbf{1}_{[0,1]} \le \chi \le \mathbf{1}_{[0,2]}$. Moreover, we note that there is a constant $c_A >0$ such that 
\be\label{eq:equivalency_omega_omega_A}
c_A^{-1} \omega \leq \omega_A \leq c_A \omega.
\ee
Arguing then as in the Step 3 of the proof of \cite[Theorem 6.2]{BE_iso}, we deduce that 
\be\label{eq:LinftyTraceEstimate}
\lvv f^{k+1} - f^k\rvv_{L^\infty_{\omega_A}(\OO)} + \lvv \gamma f^{k+1} - \gamma f^k\rvv_{L^\infty_{\omega_A}(\OO)}  \leq \left( t\, C_0  + \alpha \eta_A\right)\lvv f^{k} - f^{k-1}\rvv_{L^\infty_{\omega_A}(\OO)} ,
\ee
for a constant $\eta_A \geq 1$, satisfying $\eta_A \to 1$ as $A\to \infty$.
We then choose $A>1$ large enough such that $\alpha\eta_A <1$, and $T_0>0$ small enough such that 
$$
t\, C_0  + \alpha \eta_A <1 \qquad \forall t\in [0,T_0].
$$
Therefore, $f^k$ and $\gamma f^k$ are Cauchy sequences in the the Banach spaces $L^\infty_{\omega_A}(\UU_{T_0})$ and $L^\infty_{\omega_A}(\Gamma_{T_0})$ respectively. Therefore, there are functions $f \in  L^\infty_{\omega_A}(\UU_{T_0})$ and $\ffff \in L^\infty_{\omega_A}(\Gamma_{T_0})$ such that, as $k\to \infty$, there holds 
\be\label{eq:ControlUniformK}
f^k \to f \, \text{ strongly in }  L^\infty_{\omega_A}(\UU_{T_0}) \quad \text{ and } \quad \gamma f^k \to \ffff \,  \text{ strongly in }  L^\infty_{\omega_A}(\Gamma_{T_0}). 
\ee

Arguing then again as during the Step 3 of the proof of \cite[Theorem 6.2]{BE_iso}, using \eqref{eq:ControlUniformK} and \eqref{eq:LinftyTraceEstimate}, we may apply \cite[Lemma~6.8--(T2)]{BE_iso} and we have the existence of $\gamma f$, satisfying $\gamma_\pm f = \ffff_\pm$ and such that $\gamma f^k\wto \gamma f$ weakly in $L^1(\Gamma_{T_0}; \, (n_x\cdot v) \d\sigma_x\dv\dt)$.

Altogether, using the above informations, \eqref{eq:KolmogorovWPL12_Hypo_decay}, and \eqref{eq:equivalency_omega_omega_A}, we may pass to the limit in the weak formulation associated with Equation~\eqref{eq:TransportLinftyK} and we deduce that $f$ solves the evolution equation
\begin{equation}
	\left\{\begin{array}{llll}
		 \partial_t f &=&  \LLL^\eps f + \eps^{-2} G  &\text{ in }\UU\\
		\gamma_-f&=& \alpha\,  \RRR_\Theta \gamma_+ f + \iota \psi  &\text{ on }\Gamma_{-}\\
		f _{t=0}&=& f_0 &\text{ in }\OO,
	\end{array}\right.\label{eq:TransportLinftyK_limit}
\end{equation}
in the time interval $[0, T_0]$. Repeating this argument in every time interval $[nT_0, (n+1) T_0]$ for $n\in \N$ yields the existence of a global weak solution $f$. Furthermore, from the representation formula for $f^k$, \cite[Lemma 3.6 - (K1)]{BE_iso}, and the previously established convergence we have that
\begin{multline}\label{eq:TransportRepresentationPsi}
f(t,x,v) = e^{-\nu(v) t/\eps^2 } f_0(x-vt, v) \, \Ind_{t_1\leq 0} + \frac 1{\eps^2} \int_{\max(0, t_1)}^t  e^{-\nu(v) (t-s)/\eps^2} \, Kf(s,X^\eps_s ,v)\, \ds \\
+\frac 1{\eps^2} \int_{\max(0, t_1)}^t  e^{-\nu(v) (t-s)/\eps^2} \, G(s,X^\eps_s ,v)\, \ds \, + \,  e^{- \nu(v)(t-t_1)/\eps^2}\left( \alpha  \RRR_\Theta \gamma_+f(t_1, x_1, v) + \psi(x_1, v)\right) \, \Ind_{t_1>0} , 
\end{multline}
for every $t\geq 0$ and for almost every $(x,v)\in \OO$. 

We can then use the well-posedness from the Step 3 of the proof of Theorem~\ref{theo:ExistenceL2} to rigorously justify the repetition of the arguments leading to the hypocoercivity result from Theorem~\ref{theo:L2Decay_pert} this time applied to Equation~\eqref{eq:TransportLinftyK_limit}. We may then repeat the computations leading to the conclusion of Proposition~\ref{prop:LinftyEstPert}, and we deduce that there is $\theta>0$, such that for every $\eps\in (0,\eps_{1})$, there holds the energy estimate
\be \label{eq:Linfty_ineq_0_Summary_Linfty}
\lvv f_t \rvv_{L^\infty_{\omega} (\OO)} + \lvv \gamma f_t \rvv_{L^\infty_{\omega} (\Sigma)} \leq \varpi_\eps e^{-\theta t} \left( \lvv f_0\rvv_{L^\infty_{\omega}(\OO)} + \underset{s\in [0,t]}{\sup}  \lvv G\rvv_{L^\infty_{\omega\nu^{-1}}(\UU)}  \right),
\ee
for every $t\geq 0$, and we recall that $\varpi_\eps$ is given by Proposition~\ref{prop:LinftyEstPert}. Finally, we note that, since \eqref{eq:Linfty_ineq_0_Summary_Linfty} comes from the computations of Proposition~\ref{prop:LinftyEstPert}, this is an estimate uniform in $\alpha$.

\medskip\noindent
\emph{Step 2.} We now take a sequence $\alpha_k \nearrow 1$, with $k\in \N$, and we consider the sequence $(f_k)_{k\in \N}$ solution to the modified Maxwell reflection boundary condition problem
\begin{equation}\label{eq:linear_gak_Linfty}
\left\{\begin{array}{rrll}
		 \partial_t f_k &=&  \LLL^\eps  f_k + \eps^{-2} G  &\text{ in }\UU\\
		\gamma_- f_k&=& \alpha_k\,  \RRR \gamma_+ f_k + \iota \psi &\text{ on }\Gamma_{-}\\
		 f_{k, t=0} &=& f_0 &\text{ in }\OO,
	\end{array}\right.
\end{equation}
given by the well-posedness result established in Step~1. 
Using the above results and arguing as during the Step 4 of the proof of \cite[Theorem 6.2]{BE_iso} we deduce that there is $f\in L^\infty_{\omega}(\UU) $ with an associated trace $\gamma f\in L^\infty_{\omega}(\Gamma)$ such that  $f_k \wto f$ weakly-$*$ in $L^\infty_{\omega}(\UU)$ and $\gamma f_k \wto \gamma f$ weakly in $L^1(\Gamma; \, (n_x\cdot v) \d\sigma_x\dv\dt)$. Furthermore, $f$ with its associated trace $\gamma f$, solves Equation~\eqref{eq:PLPBE} in the sense of \eqref{eq:renormalizedFormulation_evolLinear}, and passing to the limit in the representation formula for $f_k$ there yields this Duhamel-type formula for $f$. Finally, this implies that we can rigorously justify the computations leading to the proof of Proposition~\ref{prop:LinftyEstPert}, and this concludes the proof.
\end{proof}

 \subsection{Well-posedness of steady linear kinetic equations in a weighted $L^\infty$ framework}\label{ssec:WellPosedness_Linfty_ss}
 We consider the steady problem 
 \be
	\left\{\begin{array}{llll}
		- \LLL^\eps  \FF  &=& \eps^{-2} G&\text{ in }\OO\\
		\gamma_-  \FF&=&\RRR_\Theta \gamma_+ \FF + \iota\psi &\text{ on }\Sigma_{-},
	\end{array}\right.\label{eq:PLRBE_S}
\ee
and we devote this subsection to establish a well-posedness result for Equation~\eqref{eq:PLRBE_S}.

\begin{theo}\label{theo:ExistenceL_infty_linear_transport_ss}
Assume that either Assumption \ref{item:H1} or Assumption \ref{item:H2} holds, let $\omega$ be an admissible weight function, and let $G\in L^\infty_{\omega\nu^{-1}}(\OO)$ such that $\lla G\rra_\OO =0$. Also, recall that $ \lla \psi\rra_{\Sigma_-} = 0$ and the parameters $\eps_1, \vartheta_1>0$ are given by Proposition~\ref{prop:LinftyEstPert}.

For every $\eps\in (0,\eps_{1})$ and $\vartheta_0\in (0,\vartheta_1)$,
there is $ \FF \in L^\infty_{\omega} (\OO)$, with an associated trace function $\gamma \FF \in L^\infty_{\omega}(\Sigma)$, weak solution to Equation~\eqref{eq:PLRBE_S}, i.e. for any $\varphi \in \DD(\bar \OO)$ there holds 
\begin{multline}\label{eq:renormalizedFormulation_ss}
- \int_{\OO} \eps^{-2} (K \FF) \, \varphi  + \FF \left( \eps^{-1} v\cdot \grad_x \varphi - \eps^{-2}\nu \, \varphi \right)  \,  \dv\dx  + \eps^{-1} \int_{\Sigma_+} \gamma_+ \FF\,  \varphi\,  (n_x\cdot v)_+ \, \dv\d\sigma_x \\
 - \eps^{-1} \int_{\Sigma_-} \RRR_\Theta \gamma_+ \FF\,  \varphi\,  (n_x\cdot v)_- \, \dv\d\sigma_x  = \eps^{-2} \int_{\OO} G\, \varphi \, \dv\dx +  \eps^{-1} \int_{\Sigma_-} \iota \, \psi \,   \varphi\,  (n_x\cdot v)_- \, \dv\d\sigma_x .
\end{multline}
Furthermore, we have that $\lla \FF\rra_\OO = 0$, $\FF$ is the unique solution of Equation~\eqref{eq:PLRBE_S} in the class of $L^\infty_\omega(\OO)$ functions with zero mass, and there holds 
\be
\lvv \FF \rvv_{L^\infty_\omega(\bar \OO)}  \leq  \varpi_\eps  \lVert G \rVert_{L^{\infty}_{\omega\nu^{-1}}(\OO)}  + C_0 \vartheta_0 ,
\label{eq:LinftyPerturbedFiniteTimeDecay_Summary_SS}
\ee
where $\varpi_\eps, C_0>0$ are given by Proposition~\ref{prop:LinftyEstPert}. 
\end{theo}

\begin{proof}
We split the proof into four steps.

\medskip\noindent
\emph{Step 1. (Preliminaries)}
Let us fix $f_0\in L^\infty_\omega(\OO)$ such that $\lla f_0\rra_\OO=0$, and we are interested in studying the asymptotic behavior of Equation~\eqref{eq:PLPBE} with this initial datum. 
We note that a direct application of Theorem~\ref{theo:ExistenceL_infty_linear_transport} gives the existence of $f \in L^\infty_\omega(\UU)$ with an associated trace function $\gamma f\in L^\infty_\omega(\Gamma)$, unique global weak solution to Equation~\eqref{eq:PLPBE}. 

Furthermore, Proposition~\ref{prop:LinftyEstPert} implies that there is $\theta>0$ such that
\be\label{eq:Linfty_SS}
\lVert  f_t \rVert_{L^{\infty}_{\omega}(\OO)} + \lVert \gamma  f_t \rVert_{L^{\infty}_{\omega}(\Sigma)} \leq \varpi_\eps \left( e^{-\theta t} \lvv f_0\rvv_{L^\infty_\omega(\OO)}+  \lVert G \rVert_{L^{\infty}_{\omega\nu^{-1}}(\OO)} \right)  + C_0 \vartheta_0 \qquad \forall t\geq 0,
\ee
for some constants $C_0, \varpi_\eps>0$.

\medskip\noindent
\emph{Step 2. (Existence of solutions)}
We take a sequence $(t_k)_{k\geq 0}$, with $t_0=1$ and $t_k \nearrow +\infty$, and we define the Cesàro means 
$$
F_k = F_k (x,v) := \frac 1{t_k}\int_0^{t_k} f_s (x,v) \,  \ds \qquad \text{ and } \qquad  \bar F_{k, \pm} = \bar F_{k, \pm} (x,v) := \frac 1{t_k}\int_0^{t_k} \gamma_\pm f_s (x,v) \, \ds ,
$$
which are well defined for every $k \in \N$ due to \eqref{eq:Linfty_SS}, and we note that, from the boundary conditions, we have that $\bar F_{k,-} = \RRR_\Theta \bar F_{k,+} + \iota \psi$. We observe now that \eqref{eq:Linfty_SS} futher implies that
\be\label{eq:Linfty_SS_k}
\lVert  F_k \rVert_{L^{\infty}_{\omega}(\OO)} + \lVert \bar F_{k, \pm} \rVert_{L^{\infty}_{\omega}(\Sigma_\pm)} \leq  \frac 1 {\theta t_k} \varpi_\eps \lvv f_0\rvv_{L^\infty_\omega(\OO)} + \varpi_\eps  \lVert G \rVert_{L^{\infty}_{\omega\nu^{-1}}(\OO)}  + C_0 \vartheta_0 ,
\ee
for every $k\in \N$.

\smallskip
Taking then the weak formulation given by Theorem~\ref{theo:ExistenceL_infty_linear_transport} at time $t_k$ with a test function $\varphi \in C^\infty_c( \OO)$, and dividing by $t_k$ there holds
\begin{multline}\label{eq:WeakFormulaCesaro}
- \int_{\OO} \eps^{-2} (K F_k) \, \varphi  + F_k \left( \eps^{-1} v\cdot \grad_x \varphi - \eps^{-2}\nu \, \varphi \right)  
+ \eps^{-1} \int_{\Sigma_+} \bar F_{k,+} \,  \varphi\,  (n_x\cdot v)_+   \\
- \eps^{-1} \int_{\Sigma_-} \RRR_\Theta \bar F_{k,+} \,  \varphi\,  (n_x\cdot v)_- 
=   \int_{\OO} G_k\, \varphi +  \eps^{-1} \int_{\Sigma_-} \iota \, \psi \,   \varphi\,  (n_x\cdot v)_- .
\end{multline}
where we have used the Fubinni theorem to obtain the above formula, and we have defined $G_k :=  \eps^{-2} G + (f_{t_k} - f_0)/ t_k$. We also note that \eqref{eq:Linfty_SS} gives that  
\be\label{eq:CesaroMeanRestConv}
G_k \to \eps^{-2} G  \qquad \text{ strongly in } L^\infty_\omega(\UU) \text{ as } k\to \infty.
\ee

The above formulation implies that $F_k$ solves the steady problem 
$$
- \LLL^\eps  F_k = G_k \qquad \text{ in } \DD'(\OO).
$$
Therefore, since $F_k\in L^\infty(\OO)$, \cite[Theorem~8.6]{sanchez2024kreinrutmantheorem} implies that $F_k$ admits a unique trace $\gamma F_k$, solving the Green formula 
\beqn
- \int_{\OO} \eps^{-2} (K F_k) \, \varphi  + F_k \left( \eps^{-1} v\cdot \grad_x \varphi - \eps^{-2}\nu \, \varphi \right) 
+ \eps^{-1} \int_{\Sigma} \gamma F_k \,  \varphi\,  (n_x\cdot v)_+ 
=   \int_{\OO} G_k\, \varphi  .
\eeqn
Altogether, from the unicity of the trace, the Green formula, and the above weak formulation, we have that $\bar F_{k\pm} = \gamma_\pm F_k$. Hence, we deduce that $\gamma_-  F_k = \RRR_\Theta \gamma_+ F_k + \iota\psi$ and $F_k$ is a weak solution of the following steady problem
 \be
	\left\{\begin{array}{rlll}
		- \LLL^\eps  F_k  &=&  G_k  &\text{ in }\OO\\
		\gamma_-  F_k&=&\RRR_\Theta \gamma_+ F_k + \iota\psi &\text{ on }\Sigma_{-}.
	\end{array}\right.\label{eq:PLBE_SS_Cesaro}
\ee

Our goal then is to prove that $F_k$ admits a limit as $k\to \infty$, and that such a limit solves Equation~\eqref{eq:PLBE_SS}. 

Indeed, we observe that \eqref{eq:Linfty_SS_k} implies that there is a subsequence $(k_n)_{n\geq 0}$ and functions $\FF \in L^\infty_\omega(\OO)$ and $\FF_b \in L^\infty_\omega( \Sigma)$ such that
$$
\beal
F_{{k_n}} (x, v) \wto \FF (x,v) &\quad \text{weakly-$\ast$ in $L^\infty_\omega(\OO)$ as $n\to \infty$, and}\\
\gamma F_{{k_n}} (x, v) \wto \FF_b (x,v) &\quad \text{weakly-$\ast$ in $L^\infty_\omega(\Sigma)$ as $n\to \infty$}.
\eeal
$$
Arguing now as in Steps 1 and 2 of the proof of Theorem~\ref{theo:ExistenceL_infty_linear_transport}, we deduce that $\FF$ admits a trace $\gamma \FF \in L^\infty(\Sigma)$, and there holds $\FF_b = \gamma \FF$. Furthermore, using \eqref{eq:Linfty_SS_k}, \cite[Lemma~6.8-(T2)]{BE_iso} further implies that $\gamma F_{{k_n}} (x, v) \wto \gamma \FF (x,v)$ weakly in $L^1(\Sigma, (n_x\cdot v) \d\sigma_x\dv)$. 

With these informations, and using \eqref{eq:KolmogorovWPL12_Hypo_decay}, we may pass to the limit in \eqref{eq:WeakFormulaCesaro}, and using \eqref{eq:CesaroMeanRestConv} we deduce that $\FF$, with its associated trace $\gamma \FF$, is a weak solution of Equation~\eqref{eq:PLBE_SS}.

\medskip\noindent
\emph{Step 3. (Qualitative properties and uniqueness)}
We recall that $\lla G\rra_\OO = \lla \psi\rra_{\Sigma_-} = 0$ implies that Equation~\eqref{eq:PLPBE} conserves mass (see Subsection~\ref{ssec:Hypo_Perturbed}). Therefore, since $\lla f_0\rra_\OO=0$, we have that the sequence $(f_t)_{t\geq 0}$ given by Step~1 satisfies $\lla f_t \rra_\OO = 0$ for every $t\geq0$. 

In particular, this implies that $\lla F_k\rra_\OO = 0$ for every $k\in \N$, hence we deduce that $\lla \FF \rra_\OO=0$.
Furthermore, we note that, since $\FF$ solves Equation~\eqref{eq:PLBE_SS} and $\lla \FF \rra_\OO=0$,  Proposition~\ref{prop:LinftyEstimatePerturbedFiniteT_Summary_SS} yields \eqref{eq:LinftyPerturbedFiniteTimeDecay_Summary_SS}. 

Finally, the linearity of the problem together with Proposition~\ref{prop:LinftyEstimatePerturbedFiniteT_Summary_SS} again imply the uniqueness of $\FF$ within the class of $L^\infty_\omega(\OO)$ functions with mass zero. 
\end{proof}

 \section{Existence and uniqueness of a non-equilibrium steady state}\label{sec:NESS}
During this section, we employ the a priori estimate established in Sections \ref{sec:Hypo_Pert}, \ref{sec:AprioriLinfty} and \ref{sec:AprioriLinfty_SS}, together with a fixed point argument, to prove the following well-posedness result, equivalent to Theorem~\ref{theo:NESS}.
 
 \begin{theo}\label{theo:NESS_L}
Assume that either Assumption \ref{item:H1} or Assumption \ref{item:H2} holds, and let $\omega$ be an admissible weight function.

There is a constant $\eps_\star > 0$ such that for every $\eps \in (0, \eps_\star)$, there is $\vartheta_\star = \vartheta_\star(\eps)>0$, satisfying $\vartheta_\star (\eps)\to 0$ as $\eps\to 0$, such that for every $\vartheta_0 \in (0, \vartheta_\star)$, there exists \hbox{$\FFF \in L^\infty_\omega(\UU)$}, unique non-equilibrium steady state with zero mass of Equation~\eqref{eq:PLBE_SS_QQ} in the distributional sense. 
Furthermore, there holds
\be\label{eq:BE_NESS_eps_L}
\lvv \FFF \rvv_{L^\infty_\omega(\OO)} \leq  \lambda(\eps),
\ee
for some $\lambda(\eps)>0$, satisfying 
\be\label{eq:varpi_eps_control}
\lambda(\eps) \leq \frac 1{4 \varpi_\eps C_\QQ},
\ee
where $\varpi_\eps$ is given by Proposition~\ref{prop:LinftyEstPert}, and we note that $C_\QQ>0$ is an explicit constant, given by Lemma~\ref{lem:NonlinearGaussianWeightEstimate}. In particular, we have that $\lambda(\eps)\to 0$ as $\eps\to 0$.
\end{theo}

%
%
%
%
%
%

\subsection{Estimates on the Boltzmann collision operator}\label{ssec:NonlinearQQEstimate}

We present the following classical control on the non-linear operator $\QQ$.

\begin{lem}\label{lem:NonlinearGaussianWeightEstimate}
Let $\omega$ be an admissible weight function, and let $g,h\in L^\infty_{\omega}(\OO)$. There holds
$$
\lVert \QQ(g, h) \rVert_{L^\infty_{\omega\nu^{-1}}(\OO)} \leq C_\QQ \lVert  g \rVert_{L^\infty_{\omega}(\OO)}  \lVert   h\rVert_{L^\infty_{\omega}(\OO)}.
$$
for some constant $C_\QQ = C_\QQ(\omega)>0$.
\end{lem}

\begin{rem}\label{rem:NonlinearGaussianWeightEstimate}
This estimate is classical in the study of the Boltzmann equation, see, for instance, \cite{MR3779780, MR2679358, MR1307620, MR3740632}, thus we skip the proof.
\end{rem}

\subsection{Proof of Theorem~\ref{theo:NESS_L}} \label{ssec:ProofThmNESSRescaled} 
 We set $\eps_\star = \eps_1$, and we consider $\vartheta_0 \in (0, \vartheta_1)$, where we recall that $\eps_1, \vartheta_1>0$ are given by Proposition~\ref{prop:LinftyEstPert}. We consider $\lambda>0$ to be fixed later, and we define the Banach space
$$
\ZZ:= \left\{ g\in L^\infty_{\omega}(\OO),  \,  \lvv g\rvv_{L^\infty_\omega(\OO)}  \leq \lambda\right\},
$$
equipped with the strong topology in $L^\infty_{\omega}(\OO)$, which makes $\ZZ$ a bounded, convex, closed subset of $L^\infty_{\omega}(\OO)$.

\smallskip
We consider $g\in \ZZ$, and we study the linear problem
    \be
	\left\{\begin{array}{rlll}
		-\LLL^\eps \FF &=& \eps^{-2} \QQ(g,g) &\text{ in }\OO \\
		\gamma_-  \FF&=&\RRR \gamma_+ \FF + \iota \psi &\text{ on }\Sigma_{-}.
	\end{array}\right.\label{eq:PLRBEGauss_S}
\ee
We note that the fact that $g\in \ZZ$, together with Lemma~\ref{lem:NonlinearGaussianWeightEstimate}, imply that we may apply Theorem~\ref{theo:ExistenceL_infty_linear_transport_ss}, thus providing a unique solution with mass zero for Equation~\eqref{eq:PLRBEGauss_S} in the distributional sense. 
This motivates the introduction of the map $\Psi$ defined by $\Psi(g) = \FF$, where $\FF$ is the solution to Equation~\eqref{eq:PLRBEGauss_S} as defined above, and we note that Theorem~\ref{theo:ExistenceL_infty_linear_transport_ss} further implies that $\lla \FF\rra_\OO = 0$. 

\smallskip
Using then \eqref{eq:LinftyPerturbedFiniteTimeDecay_Summary_SS} we have that there are $C_0>0$ and $\varpi_\eps>0$, satisfying $\varpi_\eps \to \infty$ as $\eps\to 0$, such that for every $\eps\in (0,\eps_\star)$ there holds
\be\label{eq:ControlExistenceNormStrongW_S}
\lvv \FF \rvv_{L^\infty_\omega(\OO)} \leq \varpi_\eps \lvv \QQ(g,g)\rvv_{L^\infty_{\omega\nu^{-1}}(\OO)} + C_0\vartheta_0 
\leq \varpi_\eps C_\QQ  \lambda^2 + C_0\vartheta_0,
\ee
where we have used Lemma~\ref{lem:NonlinearGaussianWeightEstimate}, and the very fact that $g\in \ZZ$, to obtain the second inequality.

\smallskip
We then set $\vartheta_\star  = \min(\vartheta_2, \lambda^2)$, and we choose
$$
\lambda = \lambda(\eps) := \min\left( {1\over  \varpi_\eps C_\QQ+C_0} , {1\over 4\varpi_\eps C_\QQ } \right).
$$
We observe that $\lambda(\eps) \to0$ as $\eps\to 0$; and, moreover, this choice of $\lambda$ together with \eqref{eq:ControlExistenceNormStrongW_S} implies that $\FF\in \ZZ$, thus $\Psi:\ZZ \to \ZZ$.

\smallskip
We take then $g_1, g_2\in \ZZ$, we denote $\FF_i = \Psi(g_i)$ for $i=1,2$, and we set $\bar \FF = \FF_1-\FF_2$. We note that, defined in this way, $\bar\FF$ is the weak solution (in the sense of Theorem~\ref{theo:ExistenceL_infty_linear_transport_ss}) of the following steady equation
    \beqn
	\left\{\begin{array}{rlll}
	-\LLL^\eps \bar\FF &=& \eps^{-2} \QQ(g_1+g_2, g_1-g_2) &\text{ in }\OO \\
		\gamma_- \bar\FF&=&\RRR \gamma_+ \bar\FF &\text{ on }\Sigma_{-}.
	\end{array}\right.\label{eq:PLRBEGauss_psi_k}
\eeqn
Using then again \eqref{eq:LinftyPerturbedFiniteTimeDecay_Summary_SS} (see also \cite[Proposition 6.1]{BE_iso}) we have that 
$$
\beal
\lvv \FF\rvv_{L^\infty_\omega(\OO)} & \leq  \varpi_\eps  \lvv \QQ(g_1+g_2,g_1-g_2)\rvv_{L^\infty_{\omega\nu^{-1}} (\OO)} \leq  \varpi_\eps C_\QQ \lvv g_1+g_2 \rvv_{L^\infty_\omega(\OO)} \lvv g_1- g_2 \rvv_{L^\infty_\omega(\OO)} \\
&\leq 2\lambda \varpi_\eps C_\QQ  \lvv g_1- g_2 \rvv_{L^\infty_\omega(\OO)} ,
\eeal
$$
where we have successively used Lemma~\ref{lem:NonlinearGaussianWeightEstimate}, and the fact that $g_1, g_2\in \ZZ$, to obtain the last estimate. 
The above inequality and the choice of $\lambda$ then implies that $\Psi$ is a contraction in $\ZZ$, and we thus deduce the existence of a unique fixed point $\FFF\in \ZZ$ for this map. Furthermore, we deduce that, from the very definition of the map $\Psi$, $\FFF$ is a weak solution of the nonlinear steady Boltzmann equation \eqref{eq:PLBE_SS}, that $\lla \FFF\rra_\OO = 0$, and we note that \eqref{eq:BE_NESS_eps_L} comes from the fact that $\FFF\in \ZZ$, and the very definition of $\lambda$.
\qed

 \section{Asymptotic stability}\label{sec:Stability}

We consider a function $G:\UU\to \R$, and we define $h_0 := F_0 - \MM - \FFF$ where we recall that $F_0$ is given by \eqref{eq:BEIC} and $\FFF$ is given by Theorem~\ref{theo:NESS_L}. During this section we study the following evolution equation
 \be
	\left\{\begin{array}{llll}
		\partial_{t} h &=& \LLL^\eps h  + \eps^{-2} \QQ(h, \FFF) + \eps^{-2} \QQ(\FFF, h) + \eps^{-2} G  &\text{ in }\UU\\
		\gamma_- h&=&\RRR_\Theta \gamma_+h  &\text{ on }\Gamma_{-}\\
		 h_{t=0}&=& h_0  &\text{ in }\OO.
	\end{array}\right.\label{eq:SLPBE}
\ee
 
 \begin{rem}\label{rem:Eq_h_0}
We note that if we assume $f$ to be a solution of Equation~\eqref{eq:PBE}, then, at least at the level of a priori estimates, $h = f-\FFFF$ solves Equation~\eqref{eq:SLPBE} with $G= \QQ(h,h)$, hence the motivation of the study of this problem. 
\end{rem}

\begin{rem}\label{rem:ZeroMass_h0}
We note that, since $\lla F_0\rra_\OO=1$ by hypothesis, and $\lla \FFF\rra_\OO=0$ as established in Theorem~\ref{theo:NESS_L}, there holds $\lla h_0 \rra_\OO =0$.
\end{rem}

We devote then this section to prove the following result, equivalent to that of Theorem~\ref{theo:Main}.
 \begin{theo}\label{theo:Main_h}
Assume that either Assumption \ref{item:H1} or Assumption \ref{item:H2} holds, and let $\omega$ be an admissible weight function. 
Recall the non-equilibrium steady state $\FFF$, and the parameters $\eps_\star>0$, and $\vartheta_\star =\vartheta_\star(\eps)  >0$, are given by Theorem~\ref{theo:NESS_L}.

For every $\eps\in (0,\eps_\star)$ and every $\vartheta_0 \in (0,\vartheta_{\star})$, there is $\eta(\eps)>0$, satisfying $\eta(\eps)\to 0$ as $\eps\to 0$, such that for every for every $h_0\in L^\infty_\omega(\OO)$ satisfying
\beqn
\lvv  h_0\rvv_{L^\infty_\omega(\OO)} \leq ( \eta(\eps))^2,
\eeqn
there exists  $h\in  L^\infty_\omega(\UU)$, with an associated trace function $\gamma h \in L^\infty_\omega(\Gamma)$, unique solution to Equation~\eqref{eq:SLPBE} in the distributional sense, i.e. for every test function $\varphi \in \DD(\bar \UU)$ there holds 
\begin{multline}\label{eq:WeakFormulationGaussian_h}
\int_{\OO} h(t,\cdot) \, \varphi(t,\cdot)   -  \int_0^t \int_{\OO}  h \left( \partial_t \varphi  +  \eps^{-1} v\cdot \grad_x \varphi  -\eps^{-2} \nu\varphi \right) + \eps^{-2} (K h) \varphi  \\
+ \eps^{-1} \int_0^t \int_{\Sigma} \gamma h\,  \varphi\,  (n_x\cdot v)  
-\eps^{-2}  \int_0^t \int_{\OO}  \varphi \,  \left( \QQ(h, \FFF) +  \QQ(\FFF, h) + \QQ(h,h) \right)
= \int_{\OO} h_0\,  \varphi(0,\cdot ) ,
\end{multline}
for every $t\geq 0$.
Furthermore, there is a constructive constant $\theta>0$ such that 
\be
\lvv h_t \rvv_{L^\infty_\omega(\OO)}  \leq C \, \eta(\eps) e^{-\theta t}   \label{eq:BEdecayFinal_h} \qquad  \forall t\geq0.
\ee
for some constant $C>0$, independent of $\eps$. 
\end{theo}

 \subsection{Weighted $L^\infty$ decay estimate}
In this subsection we present the necessary estimate for the analysis of Equation~\eqref{eq:SLPBE}, which will in turn lead to the proof of Theorem~\ref{theo:Main}.

\begin{prop}\label{prop:LinftyDecay_stabilitity_h}
Assume that either Assumption \ref{item:H1} or Assumption \ref{item:H2} holds, let $\omega$ be an admissible weight function, let $G:\UU\to \R$ satisfy $\lla G_t\rra_{\OO}=0$ for every $t\geq 0$, and let $h$ be a solution of Equation~\eqref{eq:SLPBE}. 
For every $\eps\in (0,\eps_{\star})$ and every $\vartheta_0 \in (0, \vartheta_\star)$, where we recall that $\eps_\star>0$ and $\vartheta_\star =\vartheta_\star(\eps)  >0$ are given by Theorem~\ref{theo:NESS_L}, there holds
\be
\lVert  h_t \rVert_{L^{\infty}_{\omega}(\bar \OO)} \leq  2\varpi_\eps \, e^{-\theta t} \left(   \lVert h_0\rVert_{L^{\infty}_{\omega}(\OO)} +  \underset{s\in[0,t]}{\sup}\left[ e^{\theta s} \lVert G_s \rVert_{L^{\infty}_{\omega\nu^{-1}}(\OO)}\right]\right)   ,
\label{eq:LinftyEstPert_h}
\ee
for every $t\geq 0$, and some $\varpi_\eps>0$ satisfying $\varpi_\eps  \to \infty$ as $\eps \to 0$, given by Proposition~\ref{prop:LinftyEstPert}.
\end{prop}

\begin{rem}
We note that Equation~\eqref{eq:SLPBE}, in contrast with Equation~\eqref{eq:PLPBE}, has homogenous boundary conditions. 
Therefore, modulo the terms $\eps^{-2} \QQ(h, \FFF) + \eps^{-2} \QQ(\FFF, h) + \eps^{-2} G$ and treating them as a perturbation, we can argue in the spirit of Theorem~\ref{theo:Hypo_decay} to obtain $L^2$ decay estimates. 
Combining this then with the $L^2-L^\infty$ techniques from Section~\ref{sec:AprioriLinfty} will allow us to obtain the $L^\infty$ decay estimates we aim to prove during this section. 
\end{rem}

\begin{proof}
We define $\widetilde G :=  \QQ(h, \FFF) + \QQ(\FFF, h) + G$, 
and we repeat the exact same computations as those leading to Proposition~\ref{prop:LinftyEstPert}. We note that the constants $\eps_1, \vartheta_1 >0$ given by Proposition~\ref{prop:LinftyEstPert} satisfy $\eps_\star <\eps_1$ and $\vartheta_\star <\vartheta_1$ due to the very definitions of $\eps_\star$ and $\vartheta_\star$. Therefore, we deduce that for every $\eps\in (0,\eps_{\star})$ and every $\vartheta_0 \in (0, \vartheta_\star)$ there holds
\beqn 
\lVert  h_t \rVert_{L^{\infty}_{\omega}(\bar \OO)} \leq  \varpi_\eps e^{-\theta t} \left(   \lVert h_0\rVert_{L^{\infty}_{\omega}(\OO)} +  \underset{s\in[0,t]}{\sup}\left[ e^{\theta s} \lVert \widetilde G_s \rVert_{L^{\infty}_{\omega\nu^{-1}}(\OO)}\right]\right)   ,
\label{eq:LinftyEstPert_h_tilde}
\eeqn
for every $t\geq 0$. 
Using then Lemma~\ref{lem:NonlinearGaussianWeightEstimate}, \eqref{eq:BE_NESS_eps_L}, and \eqref{eq:varpi_eps_control}, we deduce that 
$$
\lvv \widetilde G\rvv_{L^\infty_{\omega\nu^{-1}}(\OO)} = \lvv \QQ(h, \FFF) +  \QQ(\FFF, h) + G\rvv_{L^\infty_{\omega\nu^{-1}}(\OO)} \leq \frac 1{2\varpi_\eps} \lvv h \rvv_{L^\infty_\omega(\OO)} + \lvv  G\rvv_{L^\infty_{\omega\nu^{-1}}(\OO)}.
$$
We conclude by putting the above estimates together, and absorbing the small contributions.
\end{proof}

\subsection{Proof of Theorem~\ref{theo:Main_h}}
The proof follows exactly as that of \cite[Theorem 7.1]{BE_iso}, using the decay estimate from Proposition \ref{prop:LinftyDecay_stabilitity_h}, Lemma~\ref{lem:NonlinearGaussianWeightEstimate}, and the well-posedness result from Theorem~\ref{theo:ExistenceL_infty_linear_transport}, thus we skip it.
\qed

\bigskip

\textbf{Acknowledgements.} The author deeply thanks Kleber Carrapatoso and Stéphane Mischler for presenting the problem, for pointing important bibliographical references, for the discussions and their comments during the process of the paper. 

This project has received funding from the European Union’s Horizon 2020 research and innovation programme under the Marie Skłodowska-Curie grant agreement No 945332. 

\bigskip

The author has no conflicts of interest to declare that are relevant to the content of this article.
No datasets were generated or analyzed during the current study.

For the purpose of open access, the author has applied a Creative Commons Attribution (CC-BY) license to any Author Accepted Manuscript version arising from this submission.

\bigskip

\bigskip
\bigskip

\end{document}